\documentclass[11pt, twoside, leqno]{article}

\usepackage{amsmath}
\usepackage{amssymb}
\usepackage{titletoc}
\usepackage{mathrsfs}
\usepackage{amsthm}
\usepackage{color}
\usepackage{latexsym,amssymb}

\usepackage{indentfirst}
\usepackage{color}
\usepackage{txfonts}

\allowdisplaybreaks

\def\bint{{\ifinner\rlap{\bf\kern.30em--}
\int\else\rlap{\bf\kern.35em--}\int\fi}\ignorespaces}

\def\sbint{{\ifinner\rlap{\bf\kern.32em--}
\hspace{0.078cm}\int\else\rlap{\bf\kern.45em--}\int\fi}\ignorespaces}

\def\red{\color{red}}

\def\rr{{\mathbb R}}
\def\rn{{\mathbb{R}^n}}
\def\cc{{\mathbb C}}

\def\nn{{\mathbb N}}
\def\zz{{\mathbb Z}}

\def\fz{\infty }
\def\az{\alpha}

\def\dz{\delta}

\def\lf{\left}
\def\r{\right}
\def\ls{\lesssim}
\def\noz{\nonumber}

\def\BMO{\mathop\mathrm{\,BMO\,}}

\def\loc{{\mathrm{loc}}}
\DeclareMathOperator{\supp}{supp}

\def\XXint#1#2#3{{\setbox0=\hbox{$#1{#2#3}{\int}$ }
\vcenter{\hbox{$#2#3$ }}\kern-.6\wd0}}

\newtheorem{theorem}{Theorem}[section]
\newtheorem{lemma}[theorem]{Lemma}
\newtheorem{corollary}[theorem]{Corollary}
\newtheorem{proposition}[theorem]{Proposition}

\theoremstyle{definition}
\newtheorem{remark}[theorem]{Remark}
\newtheorem{definition}[theorem]{Definition}
\renewcommand{\appendix}{\par
   \setcounter{section}{0}%
   \setcounter{subsection}{0}%
   \setcounter{subsubsection}{0}%
   \gdef\thesection{\@Alph\c@section}%
   \gdef\thesubsection{\@Alph\c@section.\@arabic\c@subsection}%
   \gdef\theHsection{\@Alph\c@section.}%
   \gdef\theHsubsection{\@Alph\c@section.\@arabic\c@subsection}%
   \csname appendixmore\endcsname
 }

\numberwithin{equation}{section}

\begin{document}

\arraycolsep=1pt

\title{\bf\Large
Boundedness of Calder\'on--Zygmund
Operators on Special John--Nirenberg--Campanato and
Hardy-Type Spaces \\ via Congruent Cubes
\footnotetext{\hspace{-0.35cm} 2020 {\it
Mathematics Subject Classification}. Primary 42B20;
Secondary 47A30, 42B30, 46E35, 42B35.
\endgraf {\it Key words and phrases.} John--Nirenberg space,
Campanato space, Hardy-kind space, Calder\'on--Zygmund operator,
molecule, dual.
\endgraf This project is partially supported by
the National Natural Science Foundation of China (Grant Nos.\
11971058, 12071197, 12122102 and 11871100) and the National
Key Research and Development Program of China
(Grant No.\ 2020YFA0712900).}}
\date{}
\author{Hongchao Jia, Jin Tao, Dachun Yang\footnote{Corresponding author,
E-mail: \texttt{dcyang@bnu.edu.cn}/{\red August 23, 2021}/Final version.},
\ Wen Yuan and Yangyang Zhang}
\maketitle

\vspace{-0.8cm}

\begin{center}
\begin{minipage}{13cm}
{\small {\bf Abstract}\quad
Let $p\in[1,\infty]$, $q\in(1,\infty)$,
$s\in\mathbb{Z}_+:=\mathbb{N}\cup\{0\}$,
and $\alpha\in\mathbb{R}$.
In this article, the authors introduce a reasonable version
$\widetilde T$ of the Calder\'on--Zygmund operator $T$ on
$JN_{(p,q,s)_\alpha}^{\mathrm{con}}(\mathbb{R}^n)$,
the special John--Nirenberg--Campanato
space via congruent cubes,
which coincides with the Campanato space
$\mathcal{C}_{\alpha,q,s}(\mathbb{R}^n)$ when $p=\infty$.
Then the authors prove that $\widetilde T$ is bounded
on $JN_{(p,q,s)_\alpha}^{\mathrm{con}}(\mathbb{R}^n)$
if and only if, for any $\gamma\in\mathbb{Z}_+^n$
with $|\gamma|\leq s$, $T^*(x^{\gamma})=0$,
which is a well-known assumption.
To this end, the authors find an equivalent version
of this assumption.
Moreover, the authors show that $T$ can be extended to a unique continuous
linear operator on the Hardy-kind space
$HK_{(p,q,s)_{\alpha}}^{\mathrm{con}}(\mathbb{R}^n)$,
the predual space of
$JN_{(p',q',s)_\alpha}^{\mathrm{con}}(\mathbb{R}^n)$
with $\frac{1}{p}+\frac{1}{p'}=1=\frac{1}{q}+\frac{1}{q'}$,
if and only if, for any $\gamma\in\mathbb{Z}_+^n$
with $|\gamma|\leq s$, $T^*(x^{\gamma})=0$.
The main interesting integrands in the latter
boundedness are that,
to overcome the difficulty caused by
that $\|\cdot\|_{HK_{(p,q,s)_{\alpha}}^{\mathrm{con}}(\mathbb{R}^n)}$
is no longer concave, the authors
first find an equivalent norm of
$\|\cdot\|_{HK_{(p,q,s)_\alpha}^{\mathrm{con}}(\mathbb{R}^n)}$, and then
establish a criterion for the boundedness
of linear operators on $HK_{(p,q,s)_\alpha}^{\mathrm{con}}(\mathbb{R}^n)$
via introducing molecules of $HK_{(p,q,s)_\alpha}^{\mathrm{con}}(\mathbb{R}^n)$,
using the boundedness of $\widetilde T$ on
$JN_{(p,q,s)_\alpha}^{\mathrm{con}}(\mathbb{R}^n)$,
and skillfully applying the dual relation
$(HK_{(p,q,s)_{\alpha}}^{\mathrm{con}}(\mathbb{R}^n))^*
=JN_{(p',q',s)_\alpha}^{\mathrm{con}}(\mathbb{R}^n)$.
}
\end{minipage}
\end{center}

\vspace{0.2cm}



\section{Introduction\label{Introduction}}

Let $q\in(0,\infty]$ and $E$ be any given measurable set of $\rn$.
The \emph{Lebesgue space} $L^q(E)$
is defined to be the set of all measurable
functions $f$ on $E$ such that
$$\|f\|_{L^q(E)}:=
\begin{cases}
\displaystyle
\lf[\int_{E}|f(x)|^q\, dx\r]^{\frac{1}{q}}
&\text{when}\quad q\in(0,\fz),\\
\displaystyle
\mathop{\mathrm{ess\,sup}}_{x\in E}\,|f(x)| &\text{when}\quad q=\fz
\end{cases}$$
is finite. Moreover, the \textit{space} $L^q_{\mathrm{loc}}(\rn)$
is defined to be the set of all measurable functions $f$ on $\rn$ such that
$f\mathbf{1}_F\in L^q(\rn)$ for
any given bounded measurable set $F\subset \rn$,
here and thereafter, we use $\mathbf{1}_F$ to
denote the \emph{characteristic function} of $F$.

In this article, a \textit{cube} $Q$ of $\rn$
always has finite side length and all its sides parallel to the coordinate
axes, but $Q$ is not necessary to be open or closed.
Let $Q_0$ be a given cube of $\rn$.
John and Nirenberg \cite{JN} introduced the
well-known space ${\mathop{\mathrm{BMO\,}}} (Q_0)$
and, in the same article, they also introduced
the space $JN_p(Q_0)$ as a natural generalization of the space $\BMO (Q_0)$,
which is now called the John--Nirenberg space
and can be naturally defined on $\rn$ as well
(see, for instance, \cite{TYY19}).
In what follows, for any $f\in L_{\loc}^1(\rn)$
and any bounded measurable set $E\subset\rn$ with $|E|>0$, let
$$f_E:=\fint_{E}f(x)\,dx:=\frac{1}{|E|}\int_E f(x)\,dx.$$

\begin{definition}\label{d-jnp}
	Let $p\in(1,\infty)$ and $Q_0$ be a given cube of $\rn$.
	The \emph{John--Nirenberg space  $JN_p(Q_0)$}
	is defined to be the set of all $f\in L^1(Q_0)$ such that
	\begin{align*}
		\|f\|_{JN_p(Q_0)}:=
		\sup\lf\{\sum_{j}\lf|Q_{j}\r|\lf[\fint_{Q_j}
		\lf|f(x)-f_{Q_{j}}\r|\,dx\r]^{p}\r\}^{\frac{1}{p}}<\infty,
	\end{align*}
	where the supremum is taken over all collections
	$\{Q_j\}_j$ of interior pairwise disjoint subcubes of $Q_0$.
\end{definition}

Recently, the John--Nirenberg space
has attracted more and more attention.
For instance, Aalto et al. \cite{ABKY}
introduced the John--Nirenberg space in the
context of doubling metric measure spaces.
Dafni et al. \cite{DHKY}
showed the non-triviality of $JN_p(Q_0)$ and
introduced a Hardy-kind space $HK_{p'}(Q_0)$
with $\frac{1}{p}+\frac{1}{p'}=1$, which proves
the predual space of the space $JN_p(Q_0)$
(see also \cite{TYY19} for more generalization results
on $Q_0$ and $\rn$).
Very recently, Dom\'inguez and Milman \cite{DM} introduced
and studied sparse Brudnyi and John--Nirenberg spaces.
Moreover, Tao et al. \cite{TYY19}
studied the John--Nirenberg--Campanato space,
which is a generalization of the space $JN_p(Q_0)$, and
Sun et al. \cite{SXY} studied
the localized John--Nirenberg--Campanato space.
We refer the reader to \cite{FPW, FMS, HMV, MP, M, MM, TYY2,TYY20S}
for more studies on John--Nirenberg-type spaces.

It is well known that the real-variable
theory of Hardy-type and Campanato-type spaces on $\rn$,
including the boundedness of Calder\'on--Zygmund operators,
plays an important role in harmonic analysis and partial differential
equations (see, for instance, \cite{CW, MS79, N97, N10, N17, EMS1970}), but,
unfortunately, it is still a challenging
and open question to obtain the boundedness
of some important operators, for instance,
the Hardy--Littlewood maximal operator, the
Calder\'on--Zygmund operator, and the fractional integrals,
on John--Nirenberg spaces and their predual spaces, namely,
some Hardy-kind spaces.

To shed some light on the boundedness of these important operators
on John--Nirenberg spaces,
we in \cite{jtyyz1} studied the
special John--Nirenberg--Campanato space via congruent cubes,
which is actually a generalization of the Campanato space,
and, in this article, we study the boundedness
of Calder\'on--Zygmund operators on
these special John--Nirenberg--Campanato spaces.
In what follows, for any $\ell\in(0,\fz)$, let
$\Pi_{\ell}(\rn)$ be the class of all collections of
interior pairwise disjoint subcubes $\{Q_j\}_j$ of $\rn$ with side length $\ell$;
for any $s\in\zz_+:=\{0,1,2,\ldots\}$,
let $\mathcal{P}_s(\rn)$ denote the set of
all polynomials of degree not greater than $s$ on $\rn$; moreover, for any
$\gamma:=(\gamma_1,\ldots,\gamma_n)\in\zz_+^n:=(\zz_+)^n$
and $x:=(x_1,\ldots,x_n)\in\rn$,
let $|\gamma|:=\gamma_1+\cdots+\gamma_n$ and
$x^{\gamma}:=x_1^{\gamma_1}\cdots x_n^{\gamma_n}$.

\begin{definition}\label{Defin.jncc}
	Let $p\in[1,\infty]$, $q\in[1,\infty)$, $s\in\zz_+$, and $\alpha\in\rr$.
	The \emph{special John--Nirenberg--Campanato space via congruent cubes}
	(for short, \emph{congruent} JNC \emph{space}),
	$JN_{(p,q,s)_\alpha}^{\mathrm{con}}(\rn)$, is defined to be the set of all
	$f\in L^q_{\mathrm{loc}}(\rn)$ such that
\begin{align*}
	&\|f\|_{JN_{(p,q,s)_\alpha}^{\mathrm{con}}(\rn)}\\
	&\quad:=
	\begin{cases}
		\displaystyle
		\sup_{\ell\in(0,\fz)}\sup_{\{Q_j \}_{j}\in\Pi_{\ell}(\rn)}
		\lf[\sum_{j}\lf|Q_{j}\r|\lf\{\lf|Q_{j}
		\r|^{-\alpha}\lf[\fint_{Q_{j}}
		\lf|f(x)-P_{Q_j}^{(s)}(f)(x)\r|^{q}\,dx
		\r]^{\frac{1}{q}}\r\}^{p} \r]^{\frac{1}{p}},
		&p\in[1,\fz), \\
		\displaystyle
		\sup_{{\rm cube\ }Q\subset \rn}|Q|^{-\alpha}\lf[\fint_{Q}
		\lf|f(x)-P_{Q}^{(s)}(f)(x)\r|^{q}\,dx\r]^{\frac{1}{q}},
		&p=\fz
	\end{cases}
\end{align*}
is finite, where, for any cube $Q$ of $\rn$,
$P_{Q}^{(s)}(f)$ denotes the unique polynomial
of degree not greater than $s$ such that,
for any $\gamma\in\zz_+^n$ with $|\gamma|\leq s$,
\begin{align}\label{pq}
	\int_{Q}\lf[f(x)-P_{Q}^{(s)}(f)(x)\r]x^{\gamma}\,dx=0.
\end{align}
\end{definition}

\begin{remark}\label{rem-JN}
\begin{itemize}
\item [(i)]
It is well known that, for any $f\in L^1_{\mathrm{loc}}(\rn)$
and any cube (or ball) $Q$ of $\rn$, $P_{Q}^{(0)}(f)=f_Q$
and, for any $s\in\zz_+$, there exists a constant $C_{(s)}\in[1,\infty)$,
independent of $f$ and $Q$, such that, for any $x\in Q$,
\begin{align}\label{p}
	\lf|P_Q^{(s)}(f)(x)\r|\leq C_{(s)}\fint_Q|f(y)|\,dy;
\end{align}
see, for instance, \cite{L, ES2012} for more properties of $P_Q^{(s)}(f)$.

\item [(ii)]
$JN_{(\fz,q,s)_\alpha}^{\mathrm{con}}(\rn)$
is just the well-known Campanato space
$\mathcal{C}_{\alpha,q,s}(\rn)$ which was introduced
by Campanato \cite{C} and which when $\alpha=0$
coincides with the space $\mathrm{BMO}\,(\rn)$.
\end{itemize}
\end{remark}

Moreover, in \cite{jtyyz1}, we introduced a
Hardy-kind space $HK_{(p,q,s)_\alpha}^{\mathrm{con}}(\rn)$
which proves the predual
space of $JN_{(p',q',s)_\alpha}^{\mathrm{con}}(\rn)$
(see, for instance, \cite[Theorem 4.10]{jtyyz1}),
here and thereafter, $p'$ and $q'$ denote, respectively,
the conjugate indexes of $p$ and $q$, namely, $\frac{1}{p}+\frac{1}{p'}=1$
and $\frac{1}{q}+\frac{1}{q'}=1$.
In this article, we also establish
the boundedness of Calder\'on--Zygmund operators on
such a Hardy-type space
(see Theorem \ref{T-bounded-HK} below).

Furthermore, inspired by the boundedness of
the Calder\'on--Zygmund operator $T$
from $L^{\fz}(\rn)$ to $\BMO (\rn)$
(see, for instance, \cite[Theorem 6.6]{Duo01}),
we also introduce the
following subspace $RM_{p,q,\alpha}^{\mathrm{con}}(\rn)$ of
$JN_{(p,q,s)_\alpha}^{\mathrm{con}}(\rn)$, and further study the boundedness
of $T$ on $RM_{p,q,\alpha}^{\mathrm{con}}(\rn)$;
see Proposition \ref{thm-C-Z-RM} below.

\begin{definition}\label{def-RM}
Let $p$, $q\in[1,\infty]$ and $\alpha\in\rr$.
The \emph{special Riesz--Morrey space via congruent cubes},
$RM_{p,q,\alpha}^{\mathrm{con}}(\rn)$,
is defined to be the set of all $f\in L^q_{\mathrm{loc}}(\rn)$ such that
\begin{align*}
\|f\|_{RM_{p,q,\alpha}^{\rm con}(\rn)}
&:=
\begin{cases}
\displaystyle
\sup_{\ell\in(0,\fz)}\sup_{\{Q_j\}_j\in\Pi_{\ell}(\rn)}
\lf\{\sum_{j}\lf|Q_{j}\r|\lf[\lf|Q_{j}\r|^{-\alpha-\frac{1}{q}}
\|f\|_{L^q(Q_j)}\r]^{p} \r\}^{\frac{1}{p}},
&p\in[1,\fz), \\
\displaystyle
\sup_{{\rm cube\ }Q\subset \rn} |Q|^{-\az-\frac1q}\|f\|_{L^q(Q)},
&p=\fz
\end{cases}
\end{align*}
is finite.
\end{definition}

\begin{remark}\label{rem-RM}
\begin{itemize}
\item [(i)]
If $1\leq q<\alpha<p\leq\fz$, then
the space $RM_{p,q,\alpha}^{\mathrm{con}}(\rn)$
coincides with the space $(L^q, \ell^p)^{\frac{p}{1-p\alpha}}(\rn)$
which was introduced by Fofana \cite{F88};
see \cite{mfs2013, ffk2010, ffk2015} for more studies on
the space $(L^q, \ell^p)^{\alpha}(\rn)$.
	
\item [(ii)]
If we do not require that $\{Q_j\}_j$ have the same size in
the definition of $RM_{p,q,\alpha}^{\mathrm{con}}(\rn)$,
then such a space is just the Riesz--Morrey space $RM_{p,q,\alpha}(\rn)$
introduced by Tao et al. \cite{TYY21} and, independently,
by Fofana et al. \cite{ffk2015}; see also,
for instance, \cite{ffk2015,TYY21,zdjy}
for more studies on $RM_{p,q,\alpha}(\rn)$.

\item[(iii)]
If $p=\infty$, $q\in(0,\infty)$,
and $\alpha\in[-\frac1q,0)$,
then the space $RM_{p,q,\alpha}^{\mathrm{con}}(\rn)$ coincides with the
Morrey space introduced by Morrey \cite{Mo};
see, for instance, \cite{N94, N06, ST05} for more studies on Morrey-type spaces.
\end{itemize}
\end{remark}

Let $T$ be a Calder\'on--Zygmund operator
as in Definition \ref{Def-T-s-v} below.
In this article, we first introduce a reasonable version of
the Calder\'on--Zygmund operator $T$ on
$JN_{(p,q,s)_\alpha}^{\rm con}(\rn)$, denoted by $\widetilde{T}$.
Then we prove that $\widetilde{T}$ is bounded
on $JN_{(p,q,s)_\alpha}^{\mathrm{con}}(\mathbb{R}^n)$
if and only if, for any $\gamma\in\mathbb{Z}_+^n$
with $|\gamma|\leq s$, $T^*(x^{\gamma})=0$,
which is a well-known assumption; see Theorem \ref{thm-bdd-JN} below.
To this end, we find an equivalent version of this assumption;
see Proposition \ref{Assume} below.
Moreover, we show that $T$ can be extended to a unique
continuous linear operator on the Hardy-kind space
$HK_{(p,q,s)_{\alpha}}^{\mathrm{con}}(\mathbb{R}^n)$
if and only if, for any $\gamma\in\mathbb{Z}_+^n$
with $|\gamma|\leq s$, $T^*(x^{\gamma})=0$;
see Theorem \ref{T-bounded-HK} below.
To obtain the desired extension of $T$
on $HK_{(p,q,s)_\alpha}^{\mathrm{con}}(\mathbb{R}^n)$,
one key step is to estimate
$\|T(g)\|_{HK_{(p,q,s)_\alpha}^{\mathrm{con}}(\mathbb{R}^n)}$
for any $g\in HK_{(p,q,s)_{\alpha}}^{\mathrm{con-fin}}(\mathbb{R}^n)$
(see Definition \ref{d3.5} below),
but the classical method used to extend
$T$ from the finite atomic Hardy space
to the corresponding Hardy space is no longer feasible because
$\|\cdot\|_{HK_{(p,q,s)_{\alpha}}^{\mathrm{con}}(\mathbb{R}^n)}$
is not concave with $p\in(1,\fz)$,
which is essentially different from the Hardy space $H^p(\rn)$ with $p\in(0,1]$.
To avoid this, we first find an equivalent norm,
$\|\cdot\|_{(JN_{(p',q',s)_\alpha}^{\mathrm{con}}(\mathbb{R}^n))^*}$, of
$\|\cdot\|_{HK_{(p,q,s)_\alpha}^{\mathrm{con}}(\mathbb{R}^n)}$
in Proposition \ref{Coro-JN} below, and hence it suffices to
estimate $\|T(g)\|_{(JN_{(p',q',s)_\alpha}^{\mathrm{con}}(\mathbb{R}^n))^*}$.
Moreover, we use the properties of molecules of $HK_{(p,q,s)_{\alpha}}^{\mathrm{con}}(\mathbb{R}^n)$
and the dual relation $(HK_{(p,q,s)_{\alpha}}^{\mathrm{con}}(\mathbb{R}^n))^*
=JN_{(p',q',s)_\alpha}^{\mathrm{con}}(\mathbb{R}^n)$
to reduce the estimate of
$\|T(g)\|_{(JN_{(p',q',s)_\alpha}^{\mathrm{con}}(\mathbb{R}^n))^*}$
to the boundedness of $\widetilde{T}$
on $JN_{(p',q',s)_\alpha}^{\mathrm{con}}(\mathbb{R}^n)$
obtained in Theorem \ref{thm-bdd-JN};
see Theorem \ref{Bounded-HK-A} below for more details.
To limit the length of this article, the boundedness of
fractional integrals and Littlewood--Paley operators on
$JN_{(p,q,s)_{\alpha}}^{\rm con}(\rn)$ will be
studied, respectively, in the forthcoming articles \cite{jtyy3,jyyz2}.

The remainder of this article is organized as follows.

The main aim of Section \ref{CZO} is to
establish the boundedness of
Calder\'on--Zygmund operators, respectively, on
$RM_{p,q,\alpha}^{\mathrm{con}}(\rn)$ and
$JN_{(p,q,s)_\alpha}^{\rm con}(\rn)$.
In Subsection \ref{sec-def-CZO-RM}, we study the boundedness of
Calder\'on--Zygmund operators on
$RM_{p,q,\alpha}^{\mathrm{con}}(\rn)$;
see Proposition \ref{thm-C-Z-RM} below.
In Subsection \ref{sec-def-CZO-JN},
we first introduce a reasonable version
of Calder\'on--Zygmund operators
via borrowing some ideas from
\cite[Section 4]{N10}.
Then we show that such a modified
Calder\'on--Zygmund operator is bounded
on $JN_{(p,q,s)_\alpha}^{\mathrm{con}}(\mathbb{R}^n)$
if and only if, for any $\gamma\in\mathbb{Z}_+^n$
with $|\gamma|\leq s$, $T^*(x^{\gamma})=0$;
see Theorem \ref{thm-bdd-JN} below.

In Section \ref{S-C-Z-HK}, we establish the boundedness of
Calder\'on--Zygmund operators on the Hardy-kind space
$HK_{(p,q,s)_\alpha}^{\mathrm{con}}(\rn)$.
In Subsection \ref{Def-HK}, we recall
the notion of $HK_{(p,q,s)_\alpha}^{\mathrm{con}}(\rn)$
in Definition \ref{d3.6} below, and
then prove that $HK_{(p,q,s)_\alpha}^{\mathrm{con}}(\rn)$
is a Banach space in Proposition \ref{HK-banach} below.
In Subsection \ref{Mo}, to overcome the difficulty caused by
the fact that $\|\cdot\|_{HK_{(p,q,s)_{\alpha}}^{\mathrm{con}}(\mathbb{R}^n)}$
is no longer concave, we establish a criterion
(see Theorem \ref{Bounded-HK-A} below) for the boundedness
of linear operators on $HK_{(p,q,s)_\alpha}^{\mathrm{con}}(\mathbb{R}^n)$
via introducing molecules of $HK_{(p,q,s)_\alpha}^{\mathrm{con}}(\mathbb{R}^n)$
and skillfully using the dual relation
$(HK_{(p,q,s)_{\alpha}}^{\mathrm{con}}(\mathbb{R}^n))^*
=JN_{(p',q',s)_\alpha}^{\mathrm{con}}(\mathbb{R}^n)$.
In Subsection \ref{S-T-HK}, as an application of Theorem \ref{Bounded-HK-A},
we prove that the Calder\'on--Zygmund operator
$T$ can be extended to a unique
continuous linear operator on the Hardy-kind space
$HK_{(p,q,s)_{\alpha}}^{\mathrm{con}}(\mathbb{R}^n)$
if and only if, for any $\gamma\in\mathbb{Z}_+^n$
with $|\gamma|\leq s$, $T^*(x^{\gamma})=0$;
see Theorem \ref{T-bounded-HK} below.

Finally, we make some conventions on notation. Let
$\nn:=\{1,2,\ldots\}$ and $\zz_+:=\nn\cup\{0\}$.
For a given $s\in\zz_+$ and a
given ball $B\subset \rn$, $\mathcal{P}_s(B)$
[resp., $\mathcal{P}_s(\rn)$] denotes the set of all
polynomials of degree not greater than $s$ on $B$ (resp., $\rn$).
We always denote by $C$ a \emph{positive constant}
which is independent of the main parameters,
but it may vary from line to line. We also use
$C_{(\alpha,\beta,\ldots)}$ to denote a positive constant depending
on the indicated parameters $\alpha,\beta,\ldots.$
The symbol $f\lesssim g$ means that $f\le Cg$.
If $f\lesssim g$ and $g\lesssim f$,
we then write $f\sim g$. If $f\le Cg$ and
$g=h$ or $g\le h$, we then write $f\ls g\sim h$
or $f\ls g\ls h$, \emph{rather than} $f\ls g=h$ or $f\ls g\le h$.
We use $\mathbf{0}$ to denote the \emph{origin} of $\rn$.
For any $E\subset \rn$, we denote by $\mathbf{1}_E$ its
characteristic function.
We use $\eta\to 0^+$ to denote $\eta \in(0,\fz)$ and $\eta\to 0$.
For any $m\in\nn$, let $m!:=1\times\cdots\times m$ and $0!:=1$.
For any $\gamma:=(\gamma_1,\ldots,\gamma_n)\in\zz_+^n$
and $x:=(x_1,\ldots,x_n)\in\rn$,
let $|\gamma|:=\gamma_1+\cdots+\gamma_n$,
$\gamma!:=\gamma_1!\times\cdots\times\gamma_n!$, and
$x^{\gamma}:=x_1^{\gamma_1}\cdots x_n^{\gamma_n}$.
Moreover, for a given $\lambda\in(0,\fz)$ and
a given ball $B:=B(x,r):=\{y\in\rn:\ |y-x|<r\}\subset\rn$ with $x\in\rn$ and
$r\in(0,\infty)$, let $\lambda B:=B(x,\lambda r)$.
For any $z\in\rn$ and $r\in(0,\fz)$,
let $Q_z(r)$ denote the cube with center $z$ and length $r$.
Finally, for any $q\in[1,\infty]$,
we denote by $q'$ its \emph{conjugate index},
namely, $\frac{1}{q}+\frac{1}{q'}=1$.

\section{Boundedness of Calder\'on--Zygmund Operators on
$RM_{p,q,\alpha}^{\rm con}(\rn)$ and
$JN_{(p,q,s)_\alpha}^{\rm con}(\rn)$\label{CZO}}

In this section, we first establish
the boundedness of Calder\'on--Zygmund operators
on the space $RM_{p,q,\alpha}^{\rm con}(\rn)$.
Moreover, we give a reasonable version
of Calder\'on--Zygmund operators
on $JN_{(p,q,s)_\alpha}^{\rm con}(\rn)$
via borrowing some ideas from \cite[Section 4]{N10}.
Then we obtain the boundedness of
such Calder\'on--Zygmund operators on
$JN_{(p,q,s)_\alpha}^{\rm con}(\rn)$.
To this end, we first find an equivalent version
of the well-known assumption that,
for any $\gamma\in\zz_+^n$ with $|\gamma|\leq s$,
$T^*(x^{\gamma})=0$; see Proposition \ref{Assume} below.

\subsection{Calder\'on--Zygmund Operators
on $RM_{p,q,\alpha}^{\rm con}(\rn)$}\label{sec-def-CZO-RM}

This subsection is devoted to studying the boundedness of
Calder\'on--Zygmund operators on $RM_{p,q,\alpha}^{\rm con}(\rn)$.
We first recall the following notion of the
(generalized) Calder\'on--Zygmund operator $T$
(see, for instance, \cite[Definition 5.11]{Duo01}).

\begin{definition}\label{def-K}
A measurable function $K$ on $\rn\times\rn\setminus \{(x,x):\ x\in\rn\}$
is called a \textit{standard kernel} if there exist a $\dz\in(0,1]$
and a positive constant $C$ such that
\begin{itemize}
	\item[\rm (i)]
	for any $x$, $y\in\rn$ with $x\neq y$,
	\begin{align*}
		|K(x,y)|\le \frac{C}{|x-y|^{n}};
	\end{align*}
	
	\item[\rm (ii)]
	for any $x$, $y$, $\omega\in\rn$ with $x\neq y$
	and $|x-y|\ge2|x-\omega|$,
	\begin{align}\label{regular}
		|K(x,y)-K(\omega,y)|\le C\frac{|x-\omega|^\dz}{|x-y|^{n+\dz}};
	\end{align}
	
	\item[\rm (iii)]
	\eqref{regular} still holds true for the second variable of $K$.
\end{itemize}
\end{definition}

In what follows, for any measurable function $f$ on $\rn$, we define
the \textit{support} $\supp\,(f)$ of $f$ by setting
$$\supp\,(f):=\overline{\{y\in\rn:\ f(y)\neq0\}}.$$

\begin{definition}\label{defin-C-Z-L2}
Let $K$ be a standard kernel as in Definition \ref{def-K}.
A linear operator $T$ is called a \emph{Calder\'on--Zygmund operator
with kernel $K$} if $T$ is bounded on $L^2(\rn)$ and,
for any given $f\in L^2(\rn)$ with compact support,
and for almost every $x\notin\supp\,(f)$,
\begin{align}\label{Tx}
T(f)(x)=\int_{\rn}K(x,y)f(y)\,dy.
\end{align}
\end{definition}

Next, we establish the boundedness of
Calder\'on--Zygmund operators on
$RM_{p,q,\alpha}^{\mathrm{con}}(\rn)$.
We give two lemmas first. The following lemma
gives an equivalent characterization of
$RM_{p,q,\alpha}^{\rm con}(\rn)$,
whose proof is similar to that used in \cite[Proposition 4.1]{ffk2010};
we present the details here for the convenience of the reader.

\begin{lemma}\label{eq-norm}
Let $p$, $q\in[1,\infty)$ and $\alpha\in \rr$.
Then $f\in RM_{p,q,\alpha}^{\rm con}(\rn)$ if and only if
$f\in L^q_{\mathrm{loc}}(\rn)$ and
\begin{align*}
[f]_{RM_{p,q,\alpha}^{\rm con}(\rn)}
:=\sup_{r\in(0,\fz)}\lf[\int_{\rn}\lf\{|B(y,r)|^{-\alpha}
\lf[\fint_{B(y,r)}\lf|f(x)\r|^q dx
\r]^{\frac{1}{q}}\r\}^p\,dy\r]^{\frac{1}{p}}<\fz.
\end{align*}
Moreover, for any $f\in RM_{p,q,\alpha}^{\rm con}(\rn)$,
$$
[f]_{RM_{p,q,\alpha}^{\rm con}(\rn)}\sim \|f\|_{RM_{p,q,\alpha}^{\rm con}(\rn)},
$$
where the positive equivalence constants are independent of $f$.
\end{lemma}

\begin{proof}
	Let $p$, $q$, and $\alpha$ be as in the present lemma.
	We first prove that, for any
	$f\in L^q_{\mathrm{loc}}(\rn)$,
	\begin{align}\label{pp}
		\|f\|_{RM_{p,q,\alpha}^{\mathrm{con}}(\rn)}
		\gtrsim[f]_{RM_{p,q,\alpha}^{\rm con}(\rn)}.
	\end{align}
	Fix $r\in(0,\infty)$.
	Let $2\mathbb{Z}^n:=\{2k:\ k\in\mathbb{Z}^n\}$,
	$I:=\lf\{(i_1,\ldots,i_n):\ i_1\ldots,i_n\in \{0,1\}\r\}$,
	and, for any $k\in\mathbb{Z}^n$, $Q_{r,k}:=rk+[0,r)^n$.
	Observe that, for any given $k\in\mathbb{Z}^n$
	and for any $y\in Q_{r,k}$, $B(y,r/2)\subset2Q_{r,k}$.
	From this, $p\in(1,\fz)$, and the fact that, for any given $i\in I$,
	$\{2Q_{r,k+i}\}_{k\in2\mathbb{Z}^n}$ have
	disjoint interiors, we deduced that,
	for any $f\in RM_{p,q,\alpha}^{\mathrm{con}}(\rn)$,
	\begin{align*}
		&\lf[\int_{\rn}\lf\{|B(y,r/2)|^{-\alpha}
		\lf[\fint_{B(y,r/2)}\lf|f(x)\r|^q\,dx
		\r]^{\frac{1}{q}}\r\}^p\,dy\r]^{\frac{1}{p}}\\
		&\quad=\lf[\sum_{i\in I}\sum_{k\in 2\mathbb{Z}^n}
		\int_{Q_{r,k+i}}\lf\{|B(y,r/2)|^{-\alpha}
		\lf[\fint_{B(y,r/2)}\lf|f(x)\r|^q\,dx
		\r]^{\frac{1}{q}}\r\}^p\,dy\r]^{\frac{1}{p}}\\
		&\quad\lesssim
		\lf[\sum_{i\in I}\sum_{k\in 2\mathbb{Z}^n}
		\int_{Q_{r,k+i}}\lf\{|2Q_{r,k+i}|^{-\alpha}
		\lf[\fint_{2Q_{r,k+i}}\lf|f(x)
		\r|^q\,dx\r]^{\frac{1}{q}}\r\}^p\,dy\r]^{\frac{1}{p}}\\
		&\quad\lesssim
		\sum_{i\in I}\lf[\sum_{k\in 2\mathbb{Z}^n}\lf|2Q_{r,k+i}
		\r|\lf\{\lf|2Q_{r,k+i}\r|^{-\alpha}
		\lf[\fint_{2Q_{r,k+i}}\lf|f(x)\r|^q\,dx
		\r]^{\frac{1}{q}}\r\}^p\r]^{\frac{1}{p}}\\
		&\quad\lesssim\|f\|_{RM_{p,q,\alpha}^{\mathrm{con}}(\rn)},
	\end{align*}
    which implies that \eqref{pp} holds true.
	
	Now, to complete the proof of the present lemma, we only need to show that,
	for any $f\in L^q_{\mathrm{loc}}(\rn)$,
	\begin{align}\label{pp2}
		\|f\|_{RM_{p,q,\alpha}^{\mathrm{con}}(\rn)}
		\lesssim[f]_{RM_{p,q,\alpha}^{\rm con}(\rn)}.
	\end{align}
	Let $\{Q_j\}_{j}\in \Pi_r(\rn)$ with $r\in(0,\infty)$.
	Then it is easy to see that
	\begin{align*}
		&\lf[\sum_{j}\lf|Q_{j}\r|\lf\{\lf|Q_{j}\r|^{-\alpha}\lf[\fint_{Q_{j}}
		\lf|f(x)\r|^{q}\,dx\r]^{\frac{1}{q}}\r\}^{p} \r]^{\frac{1}{p}}\\
		&\quad=\lf[\sum_{j}\int_{Q_j}\lf\{|Q_j|^{-\alpha}
		\lf[\fint_{Q_j}\lf|f(x)\r|^q\,dx
		\r]^{\frac{1}{q}}\r\}^p\,dy\r]^{\frac{1}{p}}\\
		&\quad\lesssim\lf[\sum_{j}\int_{Q_j}\lf\{\lf|B(y,\sqrt nr)\r|^{-\alpha}
		\lf[\fint_{B(y,\sqrt nr)}\lf|f(x)\r|^q\,dx
		\r]^{\frac{1}{q}}\r\}^p\,dy\r]^{\frac{1}{p}}\\
		&\quad\lesssim\lf[\int_{\rn}\lf\{\lf|B(y,\sqrt nr)\r|^{-\alpha}
		\lf[\fint_{B(y,\sqrt nr)}\lf|f(x)\r|^q\,dx
		\r]^{\frac{1}{q}}\r\}^p\,dy\r]^{\frac{1}{p}},
	\end{align*}
    which implies \eqref{pp2} holds true.
	This finishes the proof of Proposition \ref{eq-norm}.
\end{proof}

The following lemma is just \cite[Proposition 8.4(a)]{ap};
see also \cite[Theorem 7]{kntyy2007} for the
corresponding weighted case. Recall that, for any given $r\in(0,\fz)$
and $p$, $q\in(1,\fz)$, the \textit{space} $(E_q^p)_r(\rn)$
is defined to be the set of
all $f\in L_{\mathrm{loc}}^q(\rn)$ such that
$$
\|f\|_{(E_q^p)_r(\rn)}:=\lf\{\int_{\rn}\lf[\fint_{B(y,r)}
\lf|f(x)\r|^q dx\r]^{\frac{p}{q}}\,dy\r\}^\frac1p<\fz;
$$
see, for instance, \cite[p.\,1600]{ap}.

\begin{lemma}\label{lem-C-Z-Am}
	Let $p$, $q\in(1,\infty)$ and $T$
	be a Calder\'on--Zygmund operator as in Definition \ref{defin-C-Z-L2}.
	Then there exists a positive constant $C$ such that,
	for any $r\in(0,\fz)$ and $f\in (E_q^p)_r(\rn)$,
	$$\|T(f)\|_{(E_q^p)_r(\rn)}
	\le C \|f\|_{(E_q^p)_r(\rn)}.$$
\end{lemma}

Using the above two lemmas, we immediately obtain the boundedness
of Calder\'on--Zygmund operators on $RM_{p,q,\alpha}^{\rm con}(\rn)$;
we omit the details here.

\begin{proposition}\label{thm-C-Z-RM}
	Let  $p$, $q\in(1,\infty)$, $\alpha\in\rr$,
	and $T$ be a Calder\'on--Zygmund operator
	as in Definition \ref{defin-C-Z-L2}.
	Then $T$ is bounded on $RM_{p,q,\alpha}^{\rm con}(\rn)$,
	namely, there exists a positive constant $C$ such that,
	for any $f\in RM_{p,q,\alpha}^{\rm con}(\rn)$,
	$$
	\|T(f)\|_{RM_{p,q,\alpha}^{\rm con}(\rn)}
	\leq C\|f\|_{RM_{p,q,\alpha}^{\rm con}(\rn)}.
	$$
\end{proposition}

\subsection{Calder\'on--Zygmund Singular Integral Operators
on $JN_{(p,q,s)_\alpha}^{\rm con}(\rn)$}\label{sec-def-CZO-JN}

In this subsection, we first give a reasonable version
of the Calder\'on--Zygmund operator $T$
on $JN_{(p,q,s)_\alpha}^{\rm con}(\rn)$, denoted by $\widetilde{T}$.
Then we establish the boundedness of $\widetilde{T}$ on
$JN_{(p,q,s)_\alpha}^{\rm con}(\rn)$.
We begin with the notion of the $s$-order standard kernel
(see, for instance \cite[Chapter III]{stein1993}).
In what follows, for any $\gamma=(\gamma_1,\ldots,\gamma_n)\in \zz_+^n$,
any $\gamma$-order differentiable function $F(\cdot,\cdot)$
on $\rn\times \rn$, and any $(x,y)\in \rn\times \rn$, let
$$
\partial_{(1)}^{\gamma}F(x,y):=\frac{\partial^{|\gamma|}}
{\partial x_1^{\gamma_1}\cdots\partial x_n^{\gamma_n}}F(x,y)
\quad\text{and}\quad\partial_{(2)}^{\gamma}F(x,y):=\frac{\partial^{|\gamma|}}
{\partial y_1^{\gamma_1}\cdots\partial y_n^{\gamma_n}}F(x,y).
$$

\begin{definition}\label{def-s-k}
Let $s\in\zz_+$. A measurable function $K$
on $\rn\times \rn\setminus\{(x,x):\ x\in\rn\}$
is called an \textit{$s$-order standard kernel} if
there exist a positive constant $C$ and
a $\dz\in(0,1]$ such that, for any $\gamma\in\zz_+^n$ with $|\gamma|\leq s$,
the following hold true:
\begin{itemize}
	\item[\rm (i)]
	for any $x$, $y\in\rn$ with $x\neq y$,
	\begin{align}\label{size-s'}
		\lf|\partial_{(2)}^{\gamma}K(x,y)\r|\le \frac{C}{|x-y|^{n+|\gamma|}};
	\end{align}
	
	\item[\rm (ii)]
	\eqref{size-s'} still holds true for the first variable of $K$;
	
	\item[\rm (iii)]
	for any $x$, $y$, $z\in\rn$
	with $x\neq y$ and $|x-y|\ge2|y-z|$,
	\begin{align}\label{regular2-s}
		\lf|\partial_{(2)}^{\gamma}K(x,y)-\partial_{(2)}^{\gamma}K(x,z)\r|
		\le C\frac{|y-z|^\dz}{|x-y|^{n+|\gamma|+\dz}};
	\end{align}
	
	\item[\rm (iv)]
	\eqref{regular2-s} still holds true for the first variable of $K$.
\end{itemize}
\end{definition}

In what follows, we use $\eta\to 0^+$
to denote $\eta \in(0,\fz)$ and $\eta\to 0$.

\begin{definition}\label{defin-C-Z-s}
Let $s\in\zz_+$ and $K$ be an $s$-order
standard kernel as in Definition \ref{def-s-k}.
A linear operator $T$ is called an
\emph{$s$-order Calder\'on--Zygmund singular integral operator
with kernel $K$} if $T$ is bounded on $L^2(\rn)$ and,
for any given $f\in L^2(\rn)$
and for almost every $x\in\rn$,
\begin{align}\label{2'}
	T(f)(x)=\lim_{\eta\to0^+}T_\eta (f)(x),
\end{align}
where
\begin{align}\label{def-T-eta}
	T_\eta (f)(x)
	:=\int_{\rn\setminus B(x,\eta)}K(x,y)f(y)\,dy.
\end{align}
\end{definition}

\begin{remark}\label{rem-2.15}
Let $T$ be as in Definition \ref{defin-C-Z-s}. Since \eqref{2'} implies
\eqref{Tx}, we find that $T$
is a Calder\'on--Zygmund operator
as in Definition \ref{defin-C-Z-L2}.
Moreover, by \cite[p.\,102]{Duo01}, we conclude that
$T$ is well defined on
$L^q(\rn)$ with $q\in(1,\fz)$ and, for any $f\in L^q(\rn)$,
\begin{align*}
	T(f)=\lim_{\eta\to0^+}T_\eta (f)
\end{align*}
both almost everywhere on $\rn$ and in $L^q(\rn)$.
\end{remark}

Let $T$ be an $s$-order Calder\'on--Zygmund singular integral operator.
Recall the well-known assumption on $T$ that,
for any $\gamma\in\zz_+^n$ with $|\gamma|\leq s$,
$T^*(x^{\gamma})=0$, namely, for any $a\in L^2(\rn)$ having compact support
and satisfying that, for any $\gamma\in\zz_+^n$ with $|\gamma|\leq s$,
$\int_{\rn} a(x)x^\gamma\,dx=0$, it holds true that
\begin{align}\label{T-x-gam}
	\int_{\rn} T(a)(x)x^\gamma\,dx=0;
\end{align}
see, for instance, \cite[Definition 9.4]{Bo2003}.
\begin{definition}\label{Def-T-s-v}
Let $s\in\zz_+$. An $s$-order
Calder\'on--Zygmund singular integral operator $T$
is said to have the \emph{vanishing moments up to
order $s$} if, for any $\gamma\in\zz_+^n$ with $|\gamma|\leq s$,
$T^*(x^{\gamma})=0$.

\end{definition}

\begin{remark}\label{rem-2.10}
Let $K$ be a measurable function on $\rn\setminus \{\mathbf{0}\}$
such that $K_1(x,y):=K(x-y)$, with $x$, $y\in\rn$
and $x\neq y$, is an $s$-order standard kernel.
Let $T$ be an $s$-order Calder\'on--Zygmund singular integral operator
with kernel $K_1$.
It is well known that such a $T$ satisfies \eqref{T-x-gam};
see, for instance, \cite{L,mc1997}. This indicates that
\eqref{T-x-gam} is a reasonable assumption.
\end{remark}

Next, we give the following notion of the $s$-order modified
Calder\'on--Zygmund operators via borrowing some
ideas from \cite[Section 4]{N10}.

\begin{definition}\label{def-JN-CZO}
	Let $s\in\zz_+$, $K$ be an $s$-order standard kernel,
	and $T$ the $s$-order Calder\'on--Zygmund operator with kernel $K$.
	For any $x$, $y\in\rn$ with $x\neq y$, let
	\begin{align}\label{Kw-K}
		\widetilde{K}(x,y):=K(y,x)
	\end{align}
    and $B_0:=B(x_0,r_0)$ be a given ball
    of $\rn$ with $x_0\in\rn$ and $r_0\in(0,\fz)$.
	The \textit{$s$-order modified Calder\'on--Zygmund
	operator $\widetilde{T}_{B_0}$ with kernel $\widetilde{K}$}
	is defined by setting, for any suitable function $f$
	on $\rn$, and almost every $x\in \rn$,
	\begin{align}\label{2.12x}
	\widetilde{T}_{B_0}(f)(x):=
	\lim_{\eta\to 0^+}\widetilde{T}_{B_0,\eta}(f)(x)
	\end{align}
    pointwisely, where, for any $\eta\in(0,\fz)$,
	\begin{align}\label{T-eta}
		\widetilde{T}_{B_0,\eta}(f)(x):=\int_{\rn\setminus B(x,\eta)}
		\lf[\widetilde{K}(x,y)-\sum_{\{\gamma\in\zz_+^n:\ |\gamma|\leq s\}}
		\frac{\partial_{(1)}^{\gamma}\widetilde{K}(x_0,y)}{\gamma!}
		(x-x_0)^{\gamma}\mathbf{1}_{\rn\setminus B_0}(y)\r]f(y)\,dy.
	\end{align}
\end{definition}

The following conclusion shows that $\widetilde{T}_{B_0}$
coincides with the $s$-order Calder\'on--Zygmund singular integral operator
with kernel $\widetilde{K}$ on $L^q(\rn)$,
with $q\in(1,\fz)$, in the sense of
modulo $\mathcal{P}_s(\rn)$.

\begin{proposition}\label{prop-2.18x}
	Let $s\in\zz_+$, $q\in(1,\fz)$, $K$
	be an $s$-order standard kernel, and $T$ the
	$s$-order Calder\'on--Zygmund
	singular integral operator with kernel $K$.
	Let $\widetilde{K}$ be as in \eqref{Kw-K},
	$T_{(1)}$ an $s$-order Calder\'on--Zygmund
	singular integral operator with kernel $\widetilde{K}$
	as in Definition \ref{defin-C-Z-s},
	and $B_0:=B(x_0,r_0)\subset \rn$ a given ball
	with $x_0\in\rn$ and $r_0\in(0,\fz)$.
	Then, for any $f\in L^q(\rn)$,
    $\widetilde{T}_{B_0}(f)$ in \eqref{2.12x}
	is well defined almost everywhere on $\rn$, and
	$T_{(1)}(f)-\widetilde{T}_{B_0}(f)\in \mathcal{P}_s(\rn)$
	after changing values on a set of measure zero.	
\end{proposition}

\begin{proof}
	Let $s$, $q$, $T$, $K$, $\widetilde{K}$, $T_{(1)}$, and $B_0:=B(x_0,r_0)$
	with $x_0\in\rn$ and $r_0\in(0,\fz)$ be as in the present proposition.
	We first show that, for any $f\in L^q(\rn)$,
	$\widetilde{T}_{B_0}(f)$ exists almost everywhere on $\rn$.
	Indeed, by Remark \ref{rem-2.15},
	\eqref{size-s'}, \eqref{Kw-K}, the H\"older inequality,
	and $q\in(1,\fz)$, we find that,
	for any $f\in L^q(\rn)$ and almost every $x\in\rn$,
	\begin{align}\label{2.14y}
		\lf|\widetilde{T}_{B_0}(f)(x)\r|
		&\leq\lf|\lim_{\eta\to0^+}\int_{\rn\setminus B(x,\eta)}
		\widetilde{K}(x,y)f(y)\,dy\r|\\
		&\quad+\int_{\rn}\sum_{\{\gamma\in\zz_+^n:\ |\gamma|\leq s\}}
		\lf|\frac{\partial_{(1)}^{\gamma}
			\widetilde{K}(x_0,y)}{\gamma!}(x-x_0)^\gamma
		\mathbf{1}_{\rn\setminus B_0}(y)f(y)\r|\,dy\noz\\
		&\ls\lf|T_{(1)}(f)(x)\r|
		+\sum_{\{\gamma\in\zz_+^n:\ |\gamma|\leq s\}}
		|x-x_0|^{|\gamma|}\int_{\rn\setminus B_0}
		\frac{|f(y)|}{|y-x_0|^{n+|\gamma|}}\,dy\noz\\
		&\ls\lf|T_{(1)}(f)(x)\r|+\sum_{\{\gamma\in\zz_+^n:\ |\gamma|\leq s\}}
		|x-x_0|^{|\gamma|}\|f\|_{L^q(\rn)}\lf[\int_{\rn\setminus B_0}
		\frac{1}{|y-x_0|^{q'(n+|\gamma|)}}\,dy\r]^{\frac{1}{q'}}\noz\\
		&<\fz.\noz
	\end{align}
	This implies that
	$\widetilde{T}_{B_0}(f)$ exists almost everywhere on $\rn$.
	
	Moreover, it is easy to see that,
	for any $f\in L^q(\rn)$ and almost every $x\in\rn$,
	\begin{align*}
		T_{(1)}(f)(x)-\widetilde{T}_{B_0}(f)(x)
		=\int_{\rn}
		\lf[\sum_{\{\gamma\in\zz_+^n:\ |\gamma|\leq s\}}
		\frac{\partial_{(1)}^{\gamma}
			\widetilde{K}(x_0,y)}{\gamma!}(x-x_0)^\gamma
		\mathbf{1}_{\rn\setminus B_0}(y)\r]f(y)\,dy.
	\end{align*}
    From this and the estimation of \eqref{2.14y}, we deduce that
	$T_{(1)}(f)-\widetilde{T}_{B_0}(f)\in \mathcal{P}_s(\rn)$
	after changing values on a set of measure zero.
	This finishes the proof of Proposition \ref{prop-2.18x}.
\end{proof}

Now, we show that $\widetilde{T}_{B_0}$
in Definition \ref{def-JN-CZO} is well defined on
$JN_{(p,q,s)_\alpha}^{\mathrm{con}}(\rn)$
(see Proposition \ref{converge} below)
and changing $B_0$ in the definition of $\widetilde{T}_{B_0}$
results in adding a polynomial of $\mathcal{P}_s(\rn)$
(see Remark \ref{rem-Tw-JN} below).
To this end, we first establish some technical lemmas.
The following lemma is the boundedness of Calder\'on--Zygmund operators
on $L^q(\rn)$ with $q\in(1,\fz)$;
see, for instance, \cite[p.\,99, Theorem 5.10]{Duo01}.
\begin{lemma}\label{Duo01}
	Let $q\in(1,\fz)$ and $T$ be a Calder\'on--Zygmund operator
	as in Definition \ref{defin-C-Z-L2}. Then
	$T$ is bounded on $L^q(\rn)$, namely,
	there exists a positive constant $C$ such that,
	for any $f\in L^q(\rn)$,
	$$
	\|T(f)\|_{L^q(\rn)}\leq C\|f\|_{L^q(\rn)}.
	$$
\end{lemma}

The following conclusion might be well known
but, to the best of our knowledge,
we did not find a complete proof.
For the convenience of the reader, we present
the details here.

\begin{lemma}\label{int-B-P}
Let $s\in\zz_+$, a ball $B\subset \rn$, and
$\{a_{\gamma}\}_{\{\gamma\in\zz_+^n:\ |\gamma|
	\leq s\}}\subset L^1_{\mathrm{loc}}(\rn)$.
If, for any $x\in \rn$ (resp., $x\in B$),
\begin{align}\label{2.31x}
	P(x):=\int_{\rn}\sum_{\{\gamma\in\zz_+^n:\ |\gamma|
		\leq s\}}a_{\gamma}(y)x^{\gamma}\,dy
\end{align}
is finite, then $P\in \mathcal{P}_s(\rn)$ [resp., $\mathcal{P}_s(B)$].
\end{lemma}

\begin{proof}
	Let $s$, $\{a_{\gamma}\}_{\{\gamma\in\zz_+^n:\ |\gamma|\leq s\}}$,
	and $P$ be as in the present lemma.
	Without loss of generality, we may only show the case $\rn$.
	To this end, it suffices to prove that,
	for any given $\gamma_0:=(\gamma^{(1)}_{0},
	\ldots,\gamma_{0}^{(n)})\in\zz_+^n$
	with $|\gamma_{0}|\leq s$,
	\begin{align}\label{2.16z}
	\lf|\int_{\rn}a_{\gamma_0}(y)\,dy\r|<\fz.	
	\end{align}
	Indeed, if \eqref{2.16z} holds true, using
	the linearity of the integral, we find that
	\begin{align*}
		P(x)=\sum_{\{\gamma\in\zz_+^n
			:\ |\gamma|\leq s\}}\int_{\rn}a_{\gamma}(y)\,dy\,x^{\gamma}
	\end{align*}
	and hence $P\in \mathcal{P}_s(\rn)$.
	To prove \eqref{2.16z},
	for any $(x_1,\ldots,x_n)\in \rn$ and $y\in\rn$, let
\begin{align}\label{2.17x}
	F_{x_1,\ldots,x_{n-1},y}(x_n)
	:=&\sum_{\{\gamma:=(\gamma_1,\ldots,\gamma_n)\in\zz_+^n:\
		|\gamma|\leq s\}}a_{\gamma}(y)
	x_1^{\gamma_1}\cdots x_n^{\gamma_n}.
\end{align}
Regard $F_{x_1,\ldots,x_{n-1},y}$
as an element of $\mathcal{P}_s(\rr)$.
	Then, from the Lagrange interpolation formula, we deduce that,
	for any $(x_1,\ldots,x_{n})$, $y\in \rr^{n}$,
	\begin{align}\label{2.17y}
		F_{x_1,\ldots,x_{n-1},y}(x_n)
		=\sum_{i=1}^{s+1}F_{x_1,\ldots,x_{n-1},y}(a_i)
		\prod_{j\neq i,\ 1\leq j\leq s+1}
		\frac{x_n-a_j}{a_i-a_j},
	\end{align}
    where $\{a_j\}_{j=1}^{s+1}\subset \rr$
    satisfy, for any $i$, $j\in\{1,\ldots,s+1\}$ with
    $i\neq j$, $a_i\neq a_j$.
    Moreover, by \eqref{2.31x} with $x$
    replaced by $(x_1,\ldots,x_{n-1},a_i)$, and \eqref{2.17x},
    we conclude that, for any $i\in\{1,\ldots,s+1\}$,
    \begin{align*}
    \lf|\int_{\rn}F_{x_1,\ldots,x_{n-1},y}(a_i)\,dy\r|<\fz
    \end{align*}
    From this and \eqref{2.17y}, we deduce that
    the coefficient of the term $x_n$ in the right hand of \eqref{2.17y}
    is integrable on $y$ and, moreover, by \eqref{2.17x}, we find that
    the coefficient of the term $x_n$ in the right hand side of \eqref{2.17y}
    is just
    $$F_{x_1,\ldots,x_{n-1},y}(1)=
    \sum_{\{\gamma:=(\gamma_1,\ldots\gamma_n)\in\zz_+^n:\
    	|\gamma|\leq s\text{ and }\gamma_n=\gamma_0^{(n)}\}}a_{\gamma}(y)
    x_1^{\gamma_1}\cdots x_{n-1}^{\gamma_{n-1}}$$
    and hence,
    for any $(x_1,\ldots,x_{n-1})\in \rr^{n-1}$,
	\begin{align*}
	\lf|\int_{\rn}\sum_{\{\gamma:=(\gamma_1,\ldots\gamma_n)\in\zz_+^n:\
	|\gamma|\leq s\text{ and }\gamma_n=\gamma_0^{(n)}\}}a_{\gamma}(y)
	x_1^{\gamma_1}\cdots x_{n-1}^{\gamma_{n-1}}\,dy\r|<\fz.
	\end{align*}
	Repeating the above procedure $n-1$ times, we conclude that
	\eqref{2.16z} holds true.
	This finishes the proof of Lemma \ref{int-B-P}.
\end{proof}

Next, we give an equivalent version of \eqref{T-x-gam};
see Proposition \ref{Assume} below.
To this end, we first establish a
technical lemma.
In what follows, for any $s\in\zz_+$, any $v\in[1,\infty]$,
and any measurable subset $E\subset\rn$,
the \emph{space $L_s^v(E)$} is defined by setting
\begin{align*}
	L_s^v(E):=\lf\{f\in L^v(E):\ \int_{E}f(x)x^{\gamma}\,dx=0
	\text{ for any } \gamma\in \zz_+^n\text{ with } |\gamma|\leq s\r\}.
\end{align*}

\begin{lemma}\label{assume-lem}
	Let $s\in\zz_+$, $K$ be an $s$-order standard kernel, and $T$ the
	$s$-order Calder\'on--Zygmund singular integral operator with kernel $K$.
	Let $\widetilde{K}$ be as in \eqref{Kw-K},
	$B_0:=B(x_0,r_0)\subset \rn$ with $x_0\in\rn$ and $r_0\in(0,\fz)$
	a given ball, and,
	for any $\nu:=(\nu_1\ldots,\nu_n)\in\zz_+^n$ with $|\nu|\leq s$,
	\begin{align}\label{P(B,B0)}
		\widetilde{T}_{B_0}(y^{\nu})(\cdot)=	
		\lim_{\eta\to0^+}\int_{\rn\setminus B(\cdot,\eta)}
		\lf[\widetilde{K}(\cdot,y)
		-\sum_{\{\gamma\in\zz_+^n:\ |\gamma|\leq s\}}
		\frac{\partial_{x}^{\gamma}
			\widetilde{K}(x_0,y)}{\gamma!}(\cdot-x_0)^\gamma
		\mathbf{1}_{\rn\setminus B_0}(y)\r]y^{\nu}\,dy
	\end{align}
    as in \eqref{2.12x} with $f$ replaced by
    $y^{\nu}:=y_1^{\nu_1}\cdots y_n^{\nu_n}$.
	Then, for any $\nu\in\zz_+^n$ with $|\nu|\leq s$,
	\begin{itemize}
		\item[\rm (i)]
		$\widetilde{T}_{B_0}(y^{\nu})$ is well defined
		almost everywhere on $\rn$, and
		$\widetilde{T}_{B_0}(y^{\nu})\in L^q_{\mathrm{loc}}(\rn)$
		for any given $q\in (1,\fz)$;

		\item[\rm (ii)]
		for any given ball $B_1\subset \rn$,
		$$\widetilde{T}_{B_0}(y^{\nu})-\widetilde{T}_{B_1}(y^{\nu})
		\in \mathcal{P}_s(\rn)$$
		after changing values on a set of measure zero;
		
		\item[\rm (iii)]
		for any given ball $B:=B(z,r)$ of $\rn$
		with $z\in\rn$ and $r\in(0,\fz)$, and for
		any $a\in L^q_{s}(\rn)$ supported in $B$ with $q\in(1,\fz)$,
		\begin{align}\label{0-0'}
			\int_{\rn} T(a)(x)x^{\nu}\,dx
			=\int_{\rn}a(x)\widetilde{T}_{B_0}(y^{\nu})(x)\,dx.
		\end{align}
	\end{itemize}	
\end{lemma}

\begin{proof}
Let $s$, $T$, $K$, $\widetilde{K}$, and $B_0=B(x_0,r_0)$
with $x_0\in\rn$ and $r_0\in(0,\fz)$ be as in the present lemma.
We first show (i). To this end,
let $\widetilde{B}$ be any given ball of $\rn$ and
$$R:=\sup\{2|x-x_0|+2|x_0|+1:\ x\in \widetilde{B}\}.$$
Then it is easy to see that $(B_0\cup \widetilde{B})\subset B(x_0,R)$
and, for any $y\notin B(x_0,R)$,
\begin{align}\label{sim1}
	|y|\geq|y-x_0|-|x_0|\geq R-|x_0|\geq1,
	\quad |y-x_0|\geq2|x-x_0|,
\end{align}
\begin{align*}
|y|\geq |y-x_0|-|x_0|>|y-x_0|-\frac{R}{2}
\geq \frac{1}{2}|y-x_0|,
\end{align*}
and
\begin{align*}
|y|\leq |y-x_0|+|x_0|< |y-x_0|+\frac{R}{2}
\leq \frac{3}{2}|y-x_0|,
\end{align*}
and hence
\begin{align}\label{sim2}
	|y|\sim |y-x_0|.
\end{align}
Moreover, by \eqref{size-s'}, \eqref{Kw-K},
the Taylor remainder theorem, \eqref{sim1}, and \eqref{sim2},
we find that, for any $\nu\in\zz_+^n$ with $|\nu|\leq s$,
and $x\in B$, there exists an
$\widetilde{x}\in\{\theta x+(1-\theta)x_0:\ \theta\in(0,1)\}$ such that
\begin{align}\label{2.21x}
	\lf|\widetilde{T}_{B_0}(y^{\nu})(x)\r|
	&\leq\lf|\lim_{\eta\to0^+}\int_{B(x_0,R)\setminus B(x,\eta)}
	\widetilde{K}(x,y)y^{\nu}\,dy\r|\\
	&\quad+\lf|\lim_{\eta\to0^+}\int_{B(x_0,R)\setminus B(x,\eta)}
	\lf[\sum_{\{\gamma\in\zz_+^n:\ |\gamma|\leq s\}}
	\frac{\partial_{(1)}^{\gamma}\widetilde{K}(x_0,y)}{\gamma!}(x-x_0)^\gamma
	\mathbf1_{\rn\setminus B_0}(y)\r]y^{\nu}\,dy\r|\noz\\
	&\quad+\int_{\rn\setminus B(x_0,R)}
	\lf|\widetilde{K}(x,y)-\sum_{\{\gamma\in\zz_+^n:\ |\gamma|\leq s\}}
	\frac{\partial_{(1)}^{\gamma}\widetilde{K}(x_0,y)}
	{\gamma!}(x-x_0)^\gamma\r||y|^{|\nu|}\,dy\noz\\
	&\leq\lf|T_{(1)}(y^{\nu}\mathbf{1}_{B(x_0,R)})(x)\r|
	+\sum_{\{\gamma\in\zz_+^n:\ |\gamma|\leq s\}}
	\int_{B(x_0,R)\setminus B_0}\frac{|y|^{|\nu|}
		|\partial_{(1)}^{\gamma}\widetilde{K}(x_0,y)|
		|x-x_0|^{|\gamma|}}{\gamma!}\,dy\noz\\
	&\quad+\int_{\rn\setminus B(x_0,R)}
	\lf|\sum_{\{\gamma\in\zz_+^n:\ |\gamma|=s\}}
	\frac{\partial_{(1)}^{\gamma}\widetilde{K}(\widetilde{x},y)
		-\partial_{(1)}^{\gamma}\widetilde{K}(x_0,y)}
	{\gamma!}(x-x_0)^\gamma\r||y|^{|\nu|}\,dy\noz\\
	&\ls\lf|T_{(1)}(y^{\nu}\mathbf{1}_{B(x_0,R)})(x)\r|
	+\sum_{\{\gamma\in\zz_+^n:\ |\gamma|\leq s\}}
	|x-x_0|^{|\gamma|}\int_{B(x_0,R)\setminus B_0}
	\frac{|y|^{|\nu|}}{|y-x_0|^{n+|\gamma|}}\,dy\noz\\
	&\quad+\int_{\rn\setminus B(x_0,R)}
	\frac{|y|^{s}|x_0-\widetilde{x}
		|^{\dz}|x-x_0|^s}{|y-x_0|^{n+s+\dz}}\,dy\noz\\
	&\ls\lf|T_{(1)}(y^{\nu}\mathbf{1}_{B(x_0,R)})(x)\r|+1,\noz
\end{align}
where the implicit positive constants depend on $x_0$ and $R$,
$T_{(1)}$ denotes the $s$-order Calder\'on--Zygmund
singular integral operator with kernel $\widetilde{K}$,
and, in the penultimate step, we used \eqref{regular2-s}
and \eqref{Kw-K} together with $|y-x_0|\geq 2|x_0-\widetilde{x}|$.
This and Remark \ref{rem-2.15} then finish the proof of (i).

Next, we prove (ii). Indeed, let
$B_1:=B(x_1,r_1)$ be a given ball of $\rn$ with $x_1\in\rn$
and $r_1\in(0,\fz)$. For any $x\in\rn$,
choose a ball $B_2\subset \rn$ such that
$x\in B_2$ and $(B_0\cup B_1)\subset B_2$.
Then, from the Lebesgue dominated convergence theorem,
we deduce that, for any $\nu\in\zz_+^n$ with $|\nu|\leq s$,
and almost every $x\in\rn$,
\begin{align*}
	&\widetilde{T}_{B_0}(y^{\nu})(x)
	-\widetilde{T}_{B_1}(y^{\nu})(x)\\
	&\quad=\int_{\rn\setminus B_2}
	\sum_{\{\gamma\in\zz_+^n:\ |\gamma|\leq s\}}
	\lf[\frac{\partial_{(1)}^{\gamma}\widetilde{K}(x_1,y)}{\gamma!}
	(x-x_1)^{\gamma}\mathbf1_{\rn\setminus B_1}(y)
	-\frac{\partial_{(1)}^{\gamma}\widetilde{K}(x_0,y)}{\gamma!}
	(x-x_0)^{\gamma}\mathbf1_{\rn\setminus B_0}(y)\r]y^{\nu}\,dy\\
	&\qquad+\lim_{\eta\to 0^+}\int_{B_2\setminus B(x,\eta)}
	\sum_{\{\gamma\in\zz_+^n:\ |\gamma|\leq s\}}
	\lf[\frac{\partial_{(1)}^{\gamma}\widetilde{K}(x_1,y)}{\gamma!}
	(x-x_1)^{\gamma}\mathbf1_{\rn\setminus B_1}(y)\r.\\
	&\quad\qquad\lf.-\frac{\partial_{(1)}^{\gamma}\widetilde{K}(x_0,y)}{\gamma!}
	(x-x_0)^{\gamma}\mathbf1_{\rn\setminus B_0}(y)\r]y^{\nu}\,dy\\
	&\quad=\int_{\rn}\sum_{\{\gamma\in\zz_+^n:\ |\gamma|\leq s\}}
	\lf[\frac{\partial_{(1)}^{\gamma}\widetilde{K}(x_1,y)}{\gamma!}
	(x-x_1)^{\gamma}\mathbf1_{\rn\setminus B_1}(y)
	-\frac{\partial_{(1)}^{\gamma}\widetilde{K}(x_0,y)}{\gamma!}
	(x-x_0)^{\gamma}\mathbf1_{\rn\setminus B_0}(y)\r]y^{\nu}\,dy.
\end{align*}
For any $x\in \rn$, let
$$
F(x):=\int_{\rn}\sum_{\{\gamma\in\zz_+^n:\ |\gamma|\leq s\}}
\lf[\frac{\partial_{(1)}^{\gamma}\widetilde{K}(x_1,y)}{\gamma!}
(x-x_1)^{\gamma}\mathbf1_{\rn\setminus B_1}(y)
-\frac{\partial_{(1)}^{\gamma}\widetilde{K}(x_0,y)}{\gamma!}
(x-x_0)^{\gamma}\mathbf1_{\rn\setminus B_0}(y)\r]y^{\nu}\,dy.
$$
Using \eqref{size-s'}, \eqref{Kw-K}, and
an argument similar to the estimation of \eqref{2.21x},
we find that, for any $x\in\rn$,
\begin{align}\label{2.22x}
|F(x)|
&\leq \int_{\rn\setminus B_2}
\lf|\sum_{\{\gamma\in\zz_+^n:\ |\gamma|\leq s\}}
\frac{\partial_{(1)}^{\gamma}\widetilde{K}(x_1,y)}
{\gamma!}(x-x_1)^\gamma-\widetilde{K}(x,y)\r||y|^{|\nu|}\,dy\\
&\quad+\int_{\rn\setminus B_2}
\lf|\widetilde{K}(x,y)-\sum_{\{\gamma\in\zz_+^n:\ |\gamma|\leq s\}}
\frac{\partial_{(1)}^{\gamma}\widetilde{K}(x_0,y)}
{\gamma!}(x-x_0)^\gamma\r||y|^{|\nu|}\,dy\noz\\
&\quad+\sum_{\{\gamma\in\zz_+^n:\ |\gamma|\leq s\}}\int_{B_2}
\lf|\frac{\partial_{(1)}^{\gamma}\widetilde{K}(x_1,y)}{\gamma!}
(x-x_1)^{\gamma}\mathbf1_{\rn\setminus B_1}(y)\r.\noz\\
&\qquad\lf.-\frac{\partial_{(1)}^{\gamma}\widetilde{K}(x_0,y)}{\gamma!}
(x-x_0)^{\gamma}\mathbf1_{\rn\setminus B_0}(y)\r|\lf|y^{\nu}\r|\,dy\noz\\
&<\fz,\noz
\end{align}
which, combined with Lemma \ref{int-B-P},
further implies that, for any $\nu\in\zz_+^n$ with $|\nu|\leq s$,
\begin{align}\label{P-B0-B1}
	\widetilde{T}_{B_0}(y^{\nu})
	-\widetilde{T}_{B_1}(y^{\nu})\in \mathcal{P}_s(\rn)
\end{align}
after changing values on a set of measure zero.
This finishes the proof of (ii).

Now, we prove (iii). Indeed, from the definition of $T$, and \eqref{Kw-K},
we deduce that, for any $a\in L^q_{s}(\rn)$ supported in $B:=B(z,r)$,
with $q\in(1,\fz)$, $z\in \rn$, and $r\in(0,\fz)$,
and for any $\nu\in\zz_+^n$ with $|\nu|\leq s$,
\begin{align}\label{2x'}
	&\int_{\rn}T(a)(x)x^{\nu}\,dx\\
	&\quad=\int_{\rn}\lim_{\eta\to0^+}\int_{B\setminus B(x,\eta)}
	K(x,y)a(y)\,dy\,x^{\nu}\,dx\noz\\
	&\quad=\int_{\rn}\lim_{\eta\to0^+}\int_{B\setminus B(x,\eta)}
	\lf[K(x,y)-\sum_{\{\gamma\in\zz_+^n:\ |\gamma|\leq s\}}
	\frac{\partial_{(2)}^{\gamma}K(x,z)}{\gamma!}(y-z)^\gamma
	\mathbf{1}_{\rn\setminus 2B}(x)\r]
	a(y)\,dy\,x^{\nu}\,dx\noz\\
	&\quad=\int_{2B}\lim_{\eta\to0^+}\int_{B\setminus B(x,\eta)}
	K(x,y)a(y)\,dy\,x^{\nu}\,dx\noz\\
	&\qquad+\int_{\rn\setminus 2B}\int_{B}
	\lf[K(x,y)-\sum_{\{\gamma\in\zz_+^n:\ |\gamma|\leq s\}}
	\frac{\partial_{(2)}^{\gamma}K(x,z)}{\gamma!}
	(y-z)^\gamma\r]a(y)\,dy\,x^{\nu}\,dx\noz\\
	&\quad=:\mathrm{I}_1+\mathrm{I}_2.\noz
\end{align}
Thus, to show \eqref{0-0'}, we only need to calculate
$\mathrm{I}_1$ and $\mathrm{I}_2$, respectively.

We first consider $\mathrm{I}_1$.
Indeed, for any $\eta\in(0,\fz)$ and $x$, $y\in\rn$, let
\begin{align}\label{p20x}
	K_\eta(x,y):=
	\begin{cases}
		\displaystyle
		K(x,y), &|x-y|\geq\eta,\\
		\displaystyle
		0, &\text{otherwise}.
	\end{cases}	
\end{align}
Then, for any $x\in\rn$,
$$T_\eta (a)(x)=\int_{B}K_{\eta}(x,y)a(y)\,dy,$$
where $T_\eta$ is as in \eqref{def-T-eta}.
By this, $a\in L^q(\rn)$, Remark \ref{rem-2.15},
and the fact that $x^{\nu}\mathbf{1}_{2B}\in L^{q'}(\rn)$,
we find that,
for any $\nu\in\zz_+^n$ with $|\nu|\leq s$,
\begin{align*}
	&\lf|\int_{2B}
	T_\eta (a)(x)\,x^{\nu}\,dx
	-\int_{2B}\lim_{\eta\to0^+}
	T_\eta (a)(x)x^{\nu}\,dx\r|\\
	&\quad\leq\int_{2B}
	\lf|T_\eta (a)(x)-\lim_{\eta\to0^+}
	T_\eta (a)(x)\r|\lf|x^{\nu}\r|\,dx\\
	&\quad\leq\lf\|T_\eta (a)-\lim_{\eta\to0^+}
	T_\eta (a)\r\|_{L^q(\rn)}\lf(\int_{2B}
	\lf|x^{\nu}\r|^{q'}\,dx\r)^{\frac{1}{q'}}
	\to0
\end{align*}
as $\eta\to 0^+$, which implies that
\begin{align}\label{2.26x}
\lim_{\eta\to0^+}\int_{2B}
T_\eta (a)(x)\,x^{\nu}\,dx
=\int_{2B}\lim_{\eta\to0^+}
T_\eta (a)(x)x^{\nu}\,dx.
\end{align}
From this and the Fubini theorem, we deduce that,
for any $\nu\in\zz_+^n$ with $|\nu|\leq s$,
\begin{align}\label{2xx'}
	\mathrm{I}_1
	&=\lim_{\eta\to0^+}\int_{2B}
	\int_{B}K_{\eta}(x,y)a(y)\,dy\,x^{\nu}\,dx\\
	&=\lim_{\eta\to0^+}\int_{B}a(y)
	\int_{2B}K_{\eta}(x,y)x^{\nu}\,dx\,dy
	=\int_{B}a(y)\lim_{\eta\to0^+}
	\int_{2B}K_{\eta}(x,y)x^{\nu}\,dx\,dy,\noz
\end{align}
where the proof of the last equality
is similar to \eqref{2.26x}.
This is a desired conclusion of $\mathrm{I}_1$.

Now, we consider $\mathrm{I}_2$.
To this end, we first show that,
for any $\nu\in\zz_+^n$ with $|\nu|\leq s$,
\begin{align}\label{3'}
	\widetilde{\mathrm{I}}_2:=\int_{\rn\setminus 2B}\int_{B}
	\lf|K(x,y)-\sum_{\{\gamma\in\zz_+^n:\ |\gamma|\leq s\}}
	\frac{\partial_{(2)}^{\gamma}K(x,z)}{\gamma!}(y-z)^\gamma\r|
	|a(y)||x|^{|\nu|}\,dy\,dx<\fz.
\end{align}
Indeed, by the Tonelli theorem, the Taylor remainder theorem,
we find that, for any $y\in B$,
there exists a $\widetilde{y}\in B$ such that,
for any $\nu\in\zz_+^n$ with $|\nu|\leq s$,
\begin{align*}
	\widetilde{\mathrm{I}}_2
	&=\int_{B}\int_{\rn\setminus 2B}
	\lf|K(x,y)-\sum_{\{\gamma\in\zz_+^n:\ |\gamma|\leq s\}}
	\frac{\partial_{(2)}^{\gamma}K(x,z)}{\gamma!}(y-z)^\gamma\r|
	|x|^{|\nu|}\,dx\,|a(y)|\,dy\\
	&=\int_{B}\int_{\rn\setminus 2B}
	\lf|\sum_{\{\gamma\in\zz_+^n:\ |\gamma|=s\}}
	\frac{\partial_{(2)}^{\gamma}K(x,\widetilde{y})
		-\partial_{(2)}^{\gamma}K(x,z)}{\gamma!}
	(y-z)^\gamma\r||x|^{|\nu|}\,dx\,|a(y)|\,dy\\
	&\ls\int_{B}\int_{\rn\setminus 2B}
	\frac{|\widetilde{y}-z|^{\dz}|y-z|^s}{|x-z|^{n+s+\dz}}
	|x|^{|\nu|}\,dx\,|a(y)|\,dy\\
	&\lesssim r^{s+\dz}\|a\|_{L^1(B)}\int_{\rn\setminus 2B}
	\frac{|x|^{|\nu|}}{|x-z|^{n+s+\dz}}\,dx<\fz,
\end{align*}
where, in the penultimate step, we used \eqref{regular2-s}
and \eqref{Kw-K} together with $|x-z|\geq 2r> 2|\widetilde{y}-z|$.
This implies that \eqref{3'} holds true.
Moreover, using \eqref{3'} and
the Fubini theorem, we find that
\begin{align}\label{4x'}
	\mathrm{I}_2=\int_{B}\int_{\rn\setminus 2B}
	\lf[K(x,y)-\sum_{\{\gamma\in\zz_+^n:\ |\gamma|\leq s\}}
	\frac{\partial_{(2)}^{\gamma}K(x,z)}{\gamma!}
	(y-z)^\gamma\r]x^{\nu}\,dx\,a(y)\,dy.
\end{align}
This is a desired conclusion of $\mathrm{I}_2$.

Next, we show \eqref{0-0'}. Indeed,
from \eqref{2x'}, \eqref{2xx'}, \eqref{4x'},
\eqref{Kw-K}, and \eqref{P-B0-B1}, we deduce that,
for any $\nu\in \zz_+^n$ with $|\nu|\leq s$,
\begin{align*}
	&\int_{\rn}T(a)(x)x^{\nu}\,dx\\
	&\quad=\int_{B}a(y)\lim_{\eta\to0^+}\int_{\rn\setminus B(y,\eta)}
	\lf[K(x,y)-\sum_{\{\gamma\in\zz_+^n:\ |\gamma|\leq s\}}
	\frac{\partial_{(2)}^{\gamma}K(x,z)}{\gamma!}(y-z)^\gamma
	\mathbf{1}_{\rn\setminus 2B}(x)\r]x^{\nu}\,dx\,dy\noz\\
	&\quad=\int_{B}a(y)\lim_{\eta\to0^+}\int_{\rn\setminus B(y,\eta)}
	\lf[\widetilde{K}(y,x)-\sum_{\{\gamma\in\zz_+^n:\ |\gamma|\leq s\}}
	\frac{\partial_{(1)}^{\gamma}\widetilde{K}(z,x)}{\gamma!}(y-z)^\gamma
	\mathbf{1}_{\rn\setminus 2B}(x)\r]x^{\nu}\,dx\,dy\\
	&\quad=\int_{B}a(y)\widetilde{T}_{2B}(x^{\nu})(y)\,dy
	=\int_{B}a(y)\widetilde{T}_{B_0}(x^{\nu})(y)\,dy.
\end{align*}
This shows that \eqref{0-0'} holds true and hence
finishes the proof of Lemma \ref{assume-lem}.
\end{proof}

\begin{proposition}\label{Assume}
Let $s\in\zz_+$, $K$ be an $s$-order standard kernel, and $T$ an
$s$-order Calder\'on--Zygmund singular integral operator with kernel $K$.
Let $\widetilde{K}$ be as in \eqref{Kw-K},
$B_0:=B(x_0,r_0)\subset \rn$ a given ball
with $x_0\in\rn$ and $r_0\in(0,\fz)$,
and, for any $\nu\in\zz_+^n$ with $|\nu|\leq s$,
$\widetilde{T}_{B_0}(y^{\nu})$ as \eqref{P(B,B0)}.
Then \eqref{T-x-gam} holds true if and only if,
for any $\nu\in\zz_+^n$ with $|\nu|\leq s$,
$\widetilde{T}_{B_0}(y^{\nu})(\cdot)\in\mathcal{P}_s(\rn)$
after changing values on a set of measure zero.
\end{proposition}

\begin{proof}
Let $s$, $T$, $K$, $\widetilde{K}$, and $B_0:=B(x_0,r_0)$
with $x_0\in\rn$ and $r_0\in(0,\fz)$ be as in the present proposition.
Now, we show the necessity.
Indeed, if \eqref{T-x-gam} holds true, then, using \eqref{pq},
we conclude that, for any $h\in L^{2}(\rn)$
supported in a ball $B\subset \rn$,
$h-P_{B}^{(s)}(h)\in L^{2}_s(B)$. By this,
\eqref{T-x-gam} with $a$ replaced by
$[h-P_{B}^{(s)}(h)]\mathbf{1}_B$, \eqref{0-0'} with $q:=2$
and $a$ replaced by $[h-P_{B}^{(s)}(h)]\mathbf{1}_B$,
and \eqref{pq}, we conclude that, for any $h\in L^{2}(\rn)$
supported in a ball $B\subset \rn$,
\begin{align}\label{arb-h}
0&=\int_{\rn} T\lf(\lf[h-P_{B}^{(s)}(h)\r]
\mathbf{1}_B\r)(x)x^{\nu}\,dx\\
&=\int_{B}\lf[h(x)-P_{B}^{(s)}(h)(x)\r]
\widetilde{T}_{B_0}(y^{\nu})(x)\,dx\noz\\
&=\int_{B}\lf[h(x)-P_{B}^{(s)}(h)(x)\r]
\lf[\widetilde{T}_{B_0}(y^{\nu})(x)
-P_{B}^{(s)}\lf(\widetilde{T}_{B_0}(y^{\nu})\r)(x)\r]\,dx\noz\\
&=\int_{B}h(x)\lf[\widetilde{T}_{B_0}(y^{\nu})(x)
-P_{B}^{(s)}\lf(\widetilde{T}_{B_0}(y^{\nu})\r)(x)\r]\,dx.\noz
\end{align}
Moreover, using Lemma \ref{assume-lem}(i), we have
$\widetilde{T}_{B_0}(y^{\nu})
-P_{B}^{(s)}(\widetilde{T}_{B_0}(y^{\nu}))\in L^2(B)$.
From this and \eqref{arb-h}, we deduce that,
for almost every $x\in B$,
\begin{align*}
\widetilde{T}_{B_0}(y^{\nu})(x)
=P_{B}^{(s)}\lf(\widetilde{T}_{B_0}(y^{\nu})\r)(x).
\end{align*}
Repeating the above procedure,
we find that, for any $k\in\nn$ and almost every $x\in 2^kB$,
\begin{align*}
\widetilde{T}_{B_0}(y^{\nu})(x)
=P_{2^kB}^{(s)}\lf(\widetilde{T}_{B_0}(y^{\nu})\r)(x)
\end{align*}
and hence, for almost every $x\in\rn$,
\begin{align*}
\widetilde{T}_{B_0}(y^{\nu})(x)
=P_{B}^{(s)}\lf(\widetilde{T}_{B_0}(y^{\nu})\r)(x).
\end{align*}
This finishes the proof of the necessity.

Moreover, the sufficiency
immediately follows from \eqref{0-0'}.
This finishes the proof of Proposition \ref{Assume}.
\end{proof}

\begin{remark}
Proposition \ref{Assume} might be well known
but, to the best of our knowledge, we did not
find a complete proof.
\end{remark}

\begin{corollary}\label{coro2.18}
Let $s\in\zz_+$, $K$ be an $s$-order standard kernel, and $T$ an
$s$-order Calder\'on--Zygmund singular integral operator with kernel $K$. Then
\eqref{T-x-gam} holds true if and only if, for any $q\in(1,\fz)$,
$a\in L^q_s(\rn)$ having compact support,
and $\gamma\in\zz_+^n$ with $|\gamma|\leq s$,
\begin{align*}
	\int_{\rn} T(a)(x)x^\gamma\,dx=0.
\end{align*}
\end{corollary}

\begin{proof}
Let $s$, $K$, and $T$ be as in the present
corollary. The sufficiency obviously holds true by taking $q=2$.
Now, we show the necessity. By \eqref{T-x-gam} and
Proposition \ref{Assume}, we conclude that, for any
$\gamma\in\zz_+^n$ with $|\gamma|\leq s$,
$\widetilde{T}_{B_0}(y^{\gamma})\in \mathcal{P}_s(\rn)$
after changing values on a set of measure zero,
where $\widetilde{T}_{B_0}$ is as in Definition \ref{def-JN-CZO}.
Using this and \eqref{0-0'}, we have, for
any $q\in(1,\fz)$, $a\in L^q_{s}(\rn)$ having compact support,
and $\gamma\in\zz_+^n$ with $|\gamma|\leq s$,
$$
\int_{\rn} T(a)(x)x^\gamma\,dx
=\int_{\rn}a(x)\widetilde{T}_{B_0}(y^{\gamma})(x)\,dx=0.
$$
This finishes the proof of the necessity
and hence of Corollary \ref{coro2.18}.
\end{proof}

The following conclusion is well known; we present
the details here for the convenience of the reader.

\begin{lemma}\label{N12-lemma}
	Let $s\in\zz_+$.
	Then there exists a positive constant $C$
	such that, for any ball (or cube) $B\subset \rn$,
	any $P\in \mathcal{P}_s(\rn)$, and any $\lambda\in(1,\fz)$,
	\begin{align}\label{mcs}
		\|P\|_{L^{\fz}(\lambda B)}
		\leq C\lambda^{\deg P}\|P\|_{L^{\fz}(B)},
	\end{align}
where $\deg P$ denotes the degree of $P$.
\end{lemma}

\begin{proof}
	Let $s\in \zz_+$ and,
	for any $x\in\rn$,
	$P(x):=\sum_{\{\gamma\in \zz_+^n:\ |\gamma|\leq \deg P\}}
	a_{\gamma}x^{\gamma}$ with $\deg P\leq s$ and $\{a_\gamma\}
	_{\{\gamma\in \zz_+^n:\ |\gamma|\leq \deg P\}}\subset \mathbb{C}$.
	We only show the case of balls because the proof of the case of cubes
	is similar.
	Indeed, by the norm-equivalence theorem
	(see, for instance, \cite[Corollary 3.3]{stein2011}),
	we find that
	\begin{align}\label{xjc1}
		\|P\|_{L^{\fz}(B(\mathbf{0},1))}
		\sim \max\lf\{|a_{\gamma}|
		:\ \gamma\in \zz_+^n\text{ and }|\gamma|\leq \deg P\r\},
	\end{align}
    where the positive equivalence constants are independent of $P$
    but depend on $s$ and $n$.
	Moreover, for any given ball $B(x_B,r_B)\subset \rn$
	with $x_B\in\rn$ and $r_B\in(0,\fz)$, and for any $x\in\rn$, let
	$$
	\widetilde{P}(x)
	:=\sum_{\{\gamma\in \zz_+^n:\ |\gamma|\leq \deg P\}}
	\widetilde{a}_{\gamma}x^{\gamma}:=P(r_Bx+x_B),
	$$
	where $\{\widetilde{a}_\gamma\}
	_{\{\gamma\in \zz_+^n:\ |\gamma|\leq \deg P\}}\subset \mathbb{C}$.
	Then, using \eqref{xjc1}, we find that
	\begin{align*}
		\|P\|_{L^{\fz}(\lambda B(x_B,r_B))}
		&=\lf\|\widetilde{P}\r\|_{L^{\fz}(\lambda B(\mathbf{0},1))}
		\ls \lambda^{\deg P}\max\lf\{|\widetilde{a}_{\gamma}|
		:\ \gamma\in \zz_+^n\text{ and }|\gamma|\leq \deg P\r\}\\
		&\sim \lambda^{\deg P}
		\lf\|\widetilde{P}\r\|_{L^\fz(B(\mathbf{0},1))}
		\sim \lambda^{\deg P}\|P\|_{L^{\fz}(B(x_B,r_B))},
	\end{align*}
	which implies that \eqref{mcs} holds true.
	This finishes the proof of Lemma \ref{N12-lemma}.
\end{proof}

By Lemma \ref{N12-lemma}, we have the following two conclusions,
which play vital roles in the proofs
of Propositions \ref{converge} and \ref{M-JN1} below.

\begin{lemma}\label{sum-g}
	Let $q\in[1,\infty)$, $s\in\zz_+$, and $\theta\in(0,2^{-s})$.
	Then there exists a positive constant $C$
	such that, for any $f\in L^1_{\mathrm{loc}}(\rn)$
	and any ball (or cube) $B\subset\rn$,
	\begin{align*}
		\sum_{k=1}^{\infty}\theta^k\lf[\fint_{2^kB}\lf|f(x)
		-P_{B}^{(s)}(f)(x)\r|^q\,dx\r]^{\frac{1}{q}}
		\leq C\sum_{k=1}^{\fz}\theta^k\lf[\fint_{2^kB}\lf|f(x)
		-P_{2^kB}^{(s)}(f)(x)\r|^q\,dx\r]^{\frac{1}{q}}.
	\end{align*}
\end{lemma}

\begin{proof}
	Let $q$, $s$, and $\theta$ be as in the present lemma.
	We only show the case of balls because the proof
	of the case of cubes is similar.
	Obviously, for any $f\in L^1_{\mathrm{loc}}(\rn)$
	and any ball $B\subset\rn$, we have
	\begin{align}\label{2.33x}
		&\sum_{k=1}^{\infty}\theta^k\lf[\fint_{2^kB}\lf|f(x)
		-P_{B}^{(s)}(f)(x)\r|^q\,dx\r]^{\frac{1}{q}}\\
		&\quad\leq \sum_{k=1}^{\infty}\theta^k
		\lf[\fint_{2^kB}\lf|f(x)
		-P_{2^kB}^{(s)}(f)(x)\r|^q\,dx\r]^{\frac{1}{q}}\noz\\
		&\qquad+\sum_{k=1}^{\infty}\theta^k\sum_{j=1}^{k}
		2^{(k-j+1)s}\lf\|P_{2^{j-1}B}^{(s)}\lf(f
		-P_{2^{j}B}^{(s)}(f)\r)\r\|_{L^{\fz}(2^{j-1} B)}\noz\\
		&\quad=:\sum_{k=1}^{\infty}\theta^k
		\lf[\fint_{2^kB}\lf|f(x)
		-P_{2^kB}^{(s)}(f)(x)\r|^q\,dx\r]^{\frac{1}{q}}+\mathrm{J}.\noz
	\end{align}
From $P_{B}^{(s)}(P)=P$ for any $P\in \mathcal{P}_s(\rn)$
and any ball $B\subset \rn$,
Lemma \ref{N12-lemma}, \eqref{p}, the H\"older inequality,
the Fubini theorem, and the fact that
$2^s\theta\in(0,1)$, we deduce that
	\begin{align*}
		\mathrm{J}
		\ls&\sum_{k=1}^{\infty}\theta^k\sum_{j=1}^{k}
		2^{(k-j+1)s}\fint_{2^{j-1}B}\lf|f(x)
		-P_{2^jB}^{(s)}(f)(x)\r|\,dx\\
		\ls&\sum_{j=1}^{\infty}\sum_{k=j}^{\infty}\theta^k2^{(k-j+1)s}
		\fint_{2^{j-1}B}\lf|f(x)
		-P_{2^jB}^{(s)}(f)(x)\r|\,dx\\
		\ls&\sum_{j=1}^{\fz}\theta^j\fint_{2^jB}
		\lf|f(x)-P_{2^jB}^{(s)}(f)(x)\r|\,dx\\
		\ls&\sum_{j=1}^{\fz}\theta^j\lf[\fint_{2^jB}
		\lf|f(x)-P_{2^jB}^{(s)}(f)(x)\r|^q\,dx\r]^{\frac{1}{q}}.
	\end{align*}
	This and \eqref{2.33x} imply the desired conclusion and hence
	finish the proof of Lemma \ref{sum-g}.
\end{proof}

\begin{lemma}\label{I-JN}
Let $p\in[1,\fz]$, $q\in(1,\fz)$, $s\in\zz_+$, $\beta\in(s,\fz)$,
and $\alpha\in (-\fz,\frac{1}{p}+\frac{\beta}{n})$.
Then there exists a positive constant $C$ such that,
for any $f\in JN_{(p,q,s)_\alpha}^{\mathrm{con}}(\rn)$ and
any ball $B(x,r)$ with $x\in\rn$ and $r\in(0,\fz)$,
\begin{align*}
\int_{\rn\setminus B(x,r)}\frac{|f(y)
-P_{B(x,r)}^{(s)}(f)(y)|}{|x-y|^{n+\beta}}\,dy
\leq Cr^{-\frac{n}{p}-\beta+\alpha n}
\|f\|_{JN_{(p,q,s)_\alpha}^{\mathrm{con}}(\rn)}.
\end{align*}
\end{lemma}

\begin{proof}
Let $p$, $q$, $s$, $\beta$, and $\alpha$ be as in the present lemma.
By Lemma \ref{sum-g} with $\theta:=2^{-\beta}$
[Here, we need $\beta\in(s,\fz)$], the H\"older inequality,
and $\alpha\in (-\fz,\frac{1}{p}+\frac{\beta}{n})$,
we conclude that, for any
$f\in JN_{(p,q,s)_\alpha}^{\mathrm{con}}(\rn)$ and any
ball $B(x,r)$ with $x\in\rn$ and $r\in(0,\fz)$,
\begin{align}\label{JN-I-E}
&\int_{\rn\setminus B(x,r)}\frac{|f(y)
	-P_{B(x,r)}^{(s)}(f)(y)|}{|x-y|^{n+\beta}}\,dy\\
&\quad\sim \sum_{k\in\nn}(2^kr)^{-n-\beta}
\int_{B(x,2^{k}r)\setminus B(x,2^{k-1}r)}
\lf|f(y)-P_{B(x,r)}^{(s)}(f)(y)\r|\,dy\noz\\
&\quad\lesssim \sum_{k\in\nn}(2^kr)^{-\beta}
\fint_{B(x,2^{k}r)}\lf|f(y)-P_{B(x,r)}^{(s)}(f)(y)\r|\,dy\noz\\
&\quad\lesssim \sum_{k\in\nn}(2^kr)^{-\beta}
\fint_{B(x,2^{k}r)}\lf|f(y)-P_{B(x,2^{k}r)}^{(s)}(f)(y)\r|\,dy\noz\\
&\quad\lesssim \sum_{k\in\nn}(2^kr)^{-\beta}
\lf[\fint_{B(x,2^{k}r)}\lf|f(y)
-P_{B(x,2^{k}r)}^{(s)}(f)(y)\r|^q\,dy\r]^{\frac{1}{q}},\noz
\end{align}
which, together with $\alpha\in (-\fz,\frac{1}{p}+\frac{\beta}{n})$
and the definition of $\|\cdot\|_{JN_{(p,q,s)_\alpha}^{\mathrm{con}}(\rn)}$,
further implies that
\begin{align*}
&\int_{\rn\setminus B(x,r)}\frac{|f(y)
	-P_{B(x,r)}^{(s)}(f)(y)|}{|x-y|^{n+\beta}}\,dy\\
&\quad\sim\sum_{k\in\nn}(2^kr)^{-\frac{n}{p}-\beta+\alpha n}
\lf|B(x,2^{k}r)\r|^{\frac{1}{p}-\alpha}
\lf[\fint_{B(x,2^{k}r)}\lf|f(y)
-P_{B(x,2^{k}r)}^{(s)}(f)(y)\r|^q\,dy\r]^{\frac{1}{q}}\\
&\quad\lesssim r^{-\frac{n}{p}-\beta+\alpha n}
\|f\|_{JN_{(p,q,s)_\alpha}^{\mathrm{con}}(\rn)}.
\end{align*}
This finishes the proof of Lemma \ref{I-JN}.
\end{proof}

\begin{proposition}\label{converge}
Let $p\in[1,\fz]$, $q\in(1,\fz)$, $s\in\zz_+$,
and $\az\in(-\fz,\frac1p+\frac{s+\dz}{n})$
with $\dz\in(0,1]$ as in Definition \ref{def-s-k}.	
Let $K$ be an $s$-order standard kernel
as in Definition \ref{def-s-k} such that
there exists an $s$-order Calder\'on--Zygmund
singular integral operator $T$ with kernel $K$
having vanishing moments up to order $s$.
Let $\widetilde{K}$ be as in \eqref{Kw-K} and
$B_0:=B(x_0,r_0)$ as in \eqref{T-eta}
with $x_0\in\rn$ and $r_0\in(0,\fz)$.
Then, for any $f\in JN_{(p,q,s)_\alpha}^{\mathrm{con}}(\rn)$,
$\widetilde{T}_{B_0}(f)$ in \eqref{2.12x}
is well defined almost everywhere on $\rn$.
\end{proposition}

\begin{proof}
Let $p$, $q$, $s$, $\alpha$, $\dz$,
$K$, $\widetilde{K}$, and $B_0:=B(x_0,r_0)$
with $x_0\in\rn$ and $r_0\in(0,\fz)$ be as in the present proposition.
Also, let $\widetilde{T}_{B_0,\eta}$
with $\eta\in(0,\fz)$ be as in \eqref{T-eta}.
We only prove the case $p\in[1,\fz)$
because the proof of $p=\fz$ is similar.
To this end, let $f\in JN_{(p,q,s)_\alpha}^{\mathrm{con}}(\rn)$
and $\widetilde{B}:=B(z,r)$ be any given ball of
$\rn$ with $z\in\rn$ and $r\in(0,\fz)$.
Then it is easy to see that, for any given $\eta\in(0,r)$
[and hence $B(x,\eta)\subset 2\widetilde{B}$]
and for any $x\in \widetilde{B}$,
$$\widetilde{T}_{B_0,\eta}(f)(x)
=E_{\widetilde{B},\eta}(x)+E_{\widetilde{B}}(x)
+E_{\widetilde{B},B_0,\eta}^{(1)}(x)+E_{\widetilde{B},B_0,\eta}^{(2)}(x),$$
where
\begin{align*}
E_{\widetilde{B},\eta}(x):=\int_{2\widetilde{B}\setminus B(x,\eta)}
\widetilde{K}(x,y)\lf[f(y)-P_{2\widetilde{B}}^{(s)}(f)(y)\r]\,dy,
\end{align*}
\begin{align*}
E_{\widetilde{B}}(x):=\int_{\rn\setminus2\widetilde{B}}
\lf[\widetilde{K}(x,y)-\sum_{\{\gamma\in\zz_+^n:\ |\gamma|\leq s\}}
\frac{\partial_{(1)}^{\gamma}\widetilde{K}(z,y)}{\gamma!}(x-z)^{\gamma}\r]
\lf[f(y)-P_{2\widetilde{B}}^{(s)}(f)(y)\r]\,dy,
\end{align*}
\begin{align*}
E_{\widetilde{B},B_0,\eta}^{(1)}(x)
:=\int_{\rn\setminus B(x,\eta)}L_{\widetilde{B},B_0}(x,y)\lf[f(y)
-P_{2\widetilde{B}}^{(s)}(f)(y)\r]\,dy
\end{align*}
with
\begin{align*}
L_{\widetilde{B},B_0}(x,y)
&:=\sum_{\{\gamma\in\zz_+^n:\ |\gamma|\leq s\}}
\frac{\partial_{(1)}^{\gamma}\widetilde{K}(z,y)}{\gamma!}(x-z)^{\gamma}
\mathbf1_{\rn\setminus 2\widetilde{B}}(y)\\
&\quad-\sum_{\{\gamma\in\zz_+^n:\ |\gamma|\leq s\}}
\frac{\partial_{(1)}^{\gamma}\widetilde{K}(x_0,y)}{\gamma!}(x-x_0)^{\gamma}
\mathbf1_{\rn\setminus B_0}(y),
\end{align*}
and
$$E_{\widetilde{B},B_0,\eta}^{(2)}(x)
:=\int_{\rn\setminus B(x,\eta)}
\lf[\widetilde{K}(x,y)-\sum_{\{\gamma\in\zz_+^n:\ |\gamma|\leq s\}}
\frac{\partial_{(1)}^{\gamma}\widetilde{K}(x_0,y)}{\gamma!}(x-x_0)^{\gamma}
\mathbf1_{\rn\setminus B_0}(y)\r]P_{2\widetilde{B}}^{(s)}(f)(y)\,dy.$$

We first consider $E_{B,\eta}(x)$ for almost every $x\in \widetilde{B}$.
Indeed, since $f\in JN_{(p,q,s)_\alpha}^{\mathrm{con}}(\rn)$,
from the definition of $\|\cdot\|_{JN_{(p,q,s)_\alpha}^{\mathrm{con}}(\rn)}$,
it is easy to deduce that
\begin{align*}
\lf|2\widetilde{B}\r|^{\frac1p-\az}\lf[\fint_{2\widetilde{B}}
\lf|f(y)-P_{2\widetilde{B}}^{(s)}(f)(y)\r|^q\,dy\r]^\frac1q
\ls\|f\|_{JN_{(p,q,s)_\alpha}^{\mathrm{con}}(\rn)}<\fz,
\end{align*}
which further implies that
$[f-P_{2\widetilde{B}}^{(s)}(f)]\mathbf1_{2\widetilde{B}}\in L^q(\rn)$,
and hence $E_{\widetilde{B},\eta}(x)$ is well defined.
Moreover, let $T_{(1)}$ be the $s$-order Calder\'on--Zygmund
singular integral operator
with kernel $\widetilde{K}$. Then,
by Remark \ref{rem-2.15}, we know that
\begin{align}\label{T1}
\lim_{\eta\to0^+}E_{\widetilde{B},\eta}
=T_{(1)}\lf(\lf[f-P_{2\widetilde{B}}^{(s)}(f)\r]\mathbf1_{2\widetilde{B}}\r)
\end{align}
almost everywhere on $\widetilde{B}$.

Next, we estimate $E_{\widetilde{B}}$.
Noticing that $E_{\widetilde{B}}$ is independent of $\eta$,
we only need to show that $E_{\widetilde{B}}$ is bounded on $\widetilde{B}$.
Indeed, using the Taylor remainder theorem,
we find that, for any $x\in \widetilde{B}$,
there exists an $\widetilde{x}\in \widetilde{B}$ such that
\begin{align}\label{T2}
\lf|E_{\widetilde{B}}(x)\r|
&\le\int_{\rn\setminus2\widetilde{B}}\lf|\widetilde{K}(x,y)
-\sum_{\{\gamma\in\zz_+^n:\ |\gamma|\leq s\}}
\frac{\partial_{(1)}^{\gamma}
\widetilde{K}(z,y)}{\gamma!}(x-z)^{\gamma}\r|\lf|f(y)
-P_{2\widetilde{B}}^{(s)}(f)(y)\r|\,dy\\
&\ls\int_{\rn\setminus 2\widetilde{B}}
\lf|\sum_{\{\gamma\in\zz_+^n:\ |\gamma|=s\}}
\frac{\partial_{(1)}^{\gamma}\widetilde{K}(\widetilde{x},y)-
\partial_{(1)}^{\gamma}\widetilde{K}(z,y)}{\gamma!}
(x-z)^{\gamma}\r|\lf|f(y)-P_{2\widetilde{B}}^{(s)}(f)(y)\r|\,dy\noz\\
&\ls|z-\widetilde{x}|^\dz|x-z|^{s}
\int_{\rn\setminus 2\widetilde{B}}\frac{|f(y)
-P_{2\widetilde{B}}^{(s)}(f)(y)|}{|y-z|^{n+s+\dz}}\,dy,\noz
\end{align}
where, in the last inequality, we used \eqref{regular2-s}
and \eqref{Kw-K} together with $|y-z|\geq 2r\geq|\widetilde{x}-z|$.
From this, Lemma \ref{I-JN} with $\beta:=s+\dz\in(s,\fz)$,
and $\az\in(-\fz,\frac1p+\frac{s+\dz}{n})$,
we deduce that, for any $x\in \widetilde{B}$,
$$
\lf|E_{\widetilde{B}}(x)\r|\ls r^{-\frac{n}{p}+\alpha n}
\|f\|_{JN_{(p,q,s)_\alpha}^{\mathrm{con}}(\rn)}<\fz.
$$
This further implies that $E_{\widetilde{B}}(x)$
is well defined for any $x\in \widetilde{B}$.

Now, we estimate $E_{\widetilde{B},B_0,\eta}^{(1)}$.
To this end, let $R:=3|z-x_0|+2r+r_0$.
It is easy to see that
\begin{align}\label{3B}
\lf(2\widetilde{B}\cup B\r)\subset B(z,R)
\end{align}
and, for any $y\in \rn\setminus B(z,R)$,
\begin{align}\label{2.36x}
|y-x_0|\leq |y-z|+|z-x_0|<2|y-z|.
\end{align}
Moreover, if $x\in \widetilde{B}$ and
$y\in B(z,R)$, then, by \eqref{size-s'} and \eqref{Kw-K}, we have
\begin{align}\label{2.44x}
\lf|L_{\widetilde{B},B_0}(x,y)\r|
&\leq\lf|\sum_{\{\gamma\in\zz_+^n:\ |\gamma|\leq s\}}
\frac{\partial_{(1)}^{\gamma}\widetilde{K}(z,y)}{\gamma!}(x-z)^{\gamma}
\mathbf1_{\rn\setminus 2\widetilde{B}}(y)\r|\\
&\quad+\lf|\sum_{\{\gamma\in\zz_+^n:\ |\gamma|\leq s\}}
\frac{\partial_{(1)}^{\gamma}\widetilde{K}(x_0,y)}{\gamma!}(x-x_0)^{\gamma}
\mathbf1_{\rn\setminus B_0}(y)\r|\noz\\
&\ls\sum_{\{\gamma\in\zz_+^n:\ |\gamma|\leq s\}}
\frac{|x-z|^{|\gamma|}}{r^{n+|\gamma|}}
+\sum_{\{\gamma\in\zz_+^n:\ |\gamma|\leq s\}}
\frac{|x-x_0|^{|\gamma|}}{r_0^{n+|\gamma|}}\noz\\
&\ls\frac{1}{r^n}+\sum_{\{\gamma\in\zz_+^n:\ |\gamma|\leq s\}}
\frac{(r+|z-x_0|)^{|\gamma|}}{r_0^{n+|\gamma|}}
\ls 1,\noz
\end{align}
where the implicit positive constants depend on
$r$, $z$, and $x_0$, but are independent of $\eta$.
If $x\in \widetilde{B}$ and $y\in \rn\setminus B(z,R)$,
then, using \eqref{3B}, the Taylor remainder
theorem, and \eqref{regular2-s} and \eqref{Kw-K}
together with \eqref{2.36x}, we find that
there exist a $z_1\in \widetilde{B}$ and a
$z_2\in B(x_0,|z-x_0|+r)$ such that
\begin{align}\label{3.37x}
\lf|L_{\widetilde{B},B_0}(x,y)\r|
&=\lf|\sum_{\{\gamma\in\zz_+^n:\ |\gamma|\leq s\}}
\frac{\partial_{(1)}^{\gamma}\widetilde{K}(z,y)}{\gamma!}(x-z)^{\gamma}
-\sum_{\{\gamma\in\zz_+^n:\ |\gamma|\leq s\}}\frac{\partial_{(1)}^{\gamma}
\widetilde{K}(x_0,y)}{\gamma!}(x-x_0)^{\gamma}\r|\\
&\leq\lf|\sum_{\{\gamma\in\zz_+^n:\ |\gamma|\leq s\}}
\frac{\partial_{(1)}^{\gamma}
\widetilde{K}(z,y)}{\gamma!}(x-z)^{\gamma}-\widetilde{K}(x,y)\r|\noz\\
&\quad+\lf|\widetilde{K}(x,y)-\sum_{\{\gamma\in\zz_+^n:\ |\gamma|\leq s\}}
\frac{\partial_{(1)}^{\gamma}
\widetilde{K}(x_0,y)}{\gamma!}(x-x_0)^{\gamma}\r|\noz\\
&\sim\lf|\sum_{\{\gamma\in\zz_+^n:\ |\gamma|=s\}}
\frac{\partial_{(1)}^{\gamma}\widetilde{K}(z,y)
-\partial_{(1)}^{\gamma}\widetilde{K}(z_1,y)}{\gamma!}
(x-z)^{\gamma}\r|\noz\\
&\quad+\lf|\sum_{\{\gamma\in\zz_+^n:\ |\gamma|=s\}}
\frac{\partial_{(1)}^{\gamma}\widetilde{K}(z_2,y)
-\partial_{(1)}^{\gamma}\widetilde{K}(x_0,y)}{\gamma!}
(x-x_0)^{\gamma}\r|\noz\\
&\ls\sum_{\{\gamma\in\zz_+^n:\ |\gamma|=s\}}
\frac{|z-z_1|^{\delta}|x-z|^{|\gamma|}}{|y-z|^{n+|\gamma|+\dz}}
+\sum_{\{\gamma\in\zz_+^n:\ |\gamma|=s\}}
\frac{|z_2-x_0|^{\delta}|x-x_0|^{|\gamma|}}{|y-z|^{n+|\gamma|+\dz}}\noz\\
&\ls\frac{r^{s+\delta}+(r+|z-x_0|)^{s+\dz}}{|y-z|^{n+s+\dz}}
\sim \frac{1}{|y-z|^{n+s+\dz}},\noz
\end{align}
where the implicit positive constants depend on
$r$, $z$, and $x_0$, but are independent of $\eta$.
From this, \eqref{2.44x}, \eqref{3.37x}, and
Lemma \ref{I-JN} with $\beta:=s+\dz\in(s,\fz)$
and $\az\in(-\fz,\frac1p+\frac{s+\dz}{n})$, we deduce that,
for any $x\in \widetilde{B}$,
\begin{align}\label{2.37x}
\lf|E_{\widetilde{B},B_0,\eta}^{(1)}(x)\r|
&\leq\int_{\rn\setminus B(x,\eta)}\lf|L_{\widetilde{B},B_0}(x,y)\lf[f(y)
-P_{2\widetilde{B}}^{(s)}(f)(y)\r]\r|\,dy\\
&\ls\int_{B(z,R)}\lf|f(y)
-P_{2\widetilde{B}}^{(s)}(f)(y)\r|\,dy+\int_{\rn\setminus B(z,R)}
\frac{|f(y)-P_{2\widetilde{B}}^{(s)}(f)(y)|}{|y-z|^{n+s+\dz}}\,dy\noz\\
&\ls\int_{B(z,R)}\lf|f(y)
-P_{2\widetilde{B}}^{(s)}(f)(y)\r|\,dy+\int_{\rn\setminus 2\widetilde{B}}
\frac{|f(y)-P_{2\widetilde{B}}^{(s)}(f)(y)|}{|y-z|^{n+s+\dz}}\,dy\noz\\
&\ls\lf\|f-P_{2\widetilde{B}}^{(s)}(f)\r\|_{L^1(B(z,R))}
+\|f\|_{JN_{(p,q,s)_\alpha}^{\mathrm{con}}(\rn)}
<\fz,\noz
\end{align}
where the implicit positive constants depend on
$r$, $z$, and $x_0$, but are independent of $\eta$.
Using \eqref{2.37x} and the Lebesgue dominated
convergence theorem, we conclude that
\begin{align*}
P_{1}(\widetilde{B},B_0;f)(\cdot)
:=\lim_{\eta\to0^+}E_{\widetilde{B},B_0,\eta}^{(1)}(\cdot)
=\int_{\rn}L_{\widetilde{B},B_0}(\cdot,y)
\lf[f(y)-P_{\widetilde{B}}^{(s)}(f)(y)\r]\,dy
\end{align*}
is finite everywhere on $\widetilde{B}$, which, combined with
Lemma \ref{int-B-P}, further implies that
\begin{align}\label{T3-lim}
P_{1}(\widetilde{B},B_0;f)\in \mathcal{P}_s(B).
\end{align}

Finally, it remains to estimate $E_{\widetilde{B},B_0,\eta}^{(2)}$.
Indeed, by Proposition \ref{Assume}, we find that
there exists a polynomial $P_2(\widetilde{B},B_0;f)\in \mathcal{P}_s(\rn)$
such that, for almost every $x\in \widetilde{B}$,
\begin{align}\label{T4}
\lim_{\eta\to 0^+}E_{\widetilde{B},B_0,\eta}^{(2)}(x)
=P_2(\widetilde{B},B_0;f)(x).
\end{align}
This is a desired estimate of $E_{\widetilde{B},B_0,\eta}^{(2)}$.

From \eqref{T1}, \eqref{T2}, \eqref{T3-lim}, and \eqref{T4},
we deduce that, for any given ball $\widetilde{B}\subset\rn$
and for almost every $x\in \widetilde{B}$,
\begin{align}\label{limit}
\widetilde{T}_{B_0}(f)(x)=&
\lim_{\eta\to0^+}\widetilde{T}_{B_0,\eta}(f)(x)\\
=&T_{(1)}\lf(\lf[f-P_{2\widetilde{B}}^{(s)}(f)
\r]\mathbf1_{2\widetilde{B}}\r)(x)
+E_{\widetilde{B}}(x)+P_1(\widetilde{B},B_0;f)(x)
+P_2(\widetilde{B},B_0;f)(x).\noz
\end{align}
Therefore, $\widetilde{T}_{B_0}(f)$
is well defined almost everywhere on $\rn$.
This finishes the proof of Proposition \ref{converge}.
\end{proof}

\begin{remark}\label{rem-Tw-JN}
	Let all the symbols be as in  Proposition \ref{converge}.
	We now claim that, for any given ball $B_1\subset \rn$,
	$$
	\widetilde{T}_{B_0}(f)-\widetilde{T}_{B_1}(f)\in \mathcal{P}_s(\rn)
	$$
	after changing values on a set of measure zero. Indeed, let
	$B_1:=B(x_1,r_1)$ be any given ball of $\rn$ with $x_1\in\rn$
	and $r_1\in(0,\fz)$. Then, for almost every $x\in\rn$,
	choosing a ball $\widetilde{B}\subset \rn$
	such that $x\in \widetilde{B}$ and $(B_0\cup B_1)\subset \widetilde{B}$,
	by the Lebesgue dominated convergence theorem, we obtain
	\begin{align}\label{2.45x}
		&\widetilde{T}_{B_0}(f)(x)-\widetilde{T}_{B_1}(f)(x)\\
		&\quad=\int_{\rn\setminus \widetilde{B}}
		\sum_{\{\gamma\in\zz_+^n:\ |\gamma|\leq s\}}
		\lf[\frac{\partial_{(1)}^{\gamma}\widetilde{K}(x_1,y)}{\gamma!}
		(x-x_1)^{\gamma}\mathbf1_{\rn\setminus B_1}(y)\r.\noz\\
		&\qquad\lf.-\frac{\partial_{(1)}^{\gamma}\widetilde{K}(x_0,y)}{\gamma!}
		(x-x_0)^{\gamma}\mathbf1_{\rn\setminus B_0}(y)\r]f(y)\,dy\noz\\
		&\qquad+\lim_{\eta\to0^+}\int_{\widetilde{B}\setminus B(x,\eta)}
		\sum_{\{\gamma\in\zz_+^n:\ |\gamma|\leq s\}}
		\lf[\frac{\partial_{(1)}^{\gamma}\widetilde{K}(x_1,y)}{\gamma!}
		(x-x_1)^{\gamma}\mathbf1_{\rn\setminus B_1}(y)\r.\noz\\
		&\quad\qquad\lf.-\frac{\partial_{(1)}^{\gamma}\widetilde{K}(x_0,y)}{\gamma!}
		(x-x_0)^{\gamma}\mathbf1_{\rn\setminus B_0}(y)\r]f(y)\,dy\noz\\
		&\quad=\int_{\rn}\sum_{\{\gamma\in\zz_+^n:\ |\gamma|\leq s\}}
		\lf[\frac{\partial_{(1)}^{\gamma}\widetilde{K}(x_1,y)}{\gamma!}
		(x-x_1)^{\gamma}\mathbf1_{\rn\setminus B_1}(y)\r.\noz\\
		&\qquad\lf.-\frac{\partial_{(1)}^{\gamma}\widetilde{K}(x_0,y)}{\gamma!}
		(x-x_0)^{\gamma}\mathbf1_{\rn\setminus B_0}(y)\r]f(y)\,dy.\noz
	\end{align}
    For any $x\in\rn$, let
    \begin{align*}
    D(x)&:=\int_{\rn}\sum_{\{\gamma\in\zz_+^n:\ |\gamma|\leq s\}}
    \lf[\frac{\partial_{(1)}^{\gamma}\widetilde{K}(x_1,y)}{\gamma!}
    (x-x_1)^{\gamma}\mathbf1_{\rn\setminus B_1}(y)\r.\\
    &\qquad\qquad\quad\lf.-\frac{\partial_{(1)}^{\gamma}
    \widetilde{K}(x_0,y)}{\gamma!}
    (x-x_0)^{\gamma}\mathbf1_{\rn\setminus B_0}(y)\r]f(y)\,dy.
    \end{align*}
	By \eqref{2.45x}, an argument similar to the estimation
	of \eqref{2.37x}, and \eqref{2.22x}, we conclude that, for any $x\in\rn$,
	\begin{align*}
	|D(x)|
	&\leq \int_{\rn\setminus \widetilde{B}}
	\lf|\sum_{\{\gamma\in\zz_+^n:\ |\gamma|\leq s\}}
	\lf[\frac{\partial_{(1)}^{\gamma}\widetilde{K}(x_1,y)}{\gamma!}
	(x-x_1)^{\gamma}\mathbf1_{\rn\setminus B_1}(y)\r.\r.\\
	&\quad\lf.\lf.-\frac{\partial_{(1)}^{\gamma}\widetilde{K}(x_0,y)}{\gamma!}
	(x-x_0)^{\gamma}\mathbf1_{\rn\setminus B_0}(y)\r]\r|\lf|f(y)
	-P_{2B_1}^{(s)}(f)(y)\r|\,dy\\
    &\quad+\lf|\int_{\rn\setminus \widetilde{B}}
    \sum_{\{\gamma\in\zz_+^n:\ |\gamma|\leq s\}}
    \lf[\frac{\partial_{(1)}^{\gamma}\widetilde{K}(x_1,y)}{\gamma!}
    (x-x_1)^{\gamma}\mathbf1_{\rn\setminus B_1}(y)\r.\r.\\
    &\qquad\lf.\lf.-\frac{\partial_{(1)}^{\gamma}\widetilde{K}(x_0,y)}{\gamma!}
    (x-x_0)^{\gamma}\mathbf1_{\rn\setminus B_0}(y)\r]
    P_{2B_1}^{(s)}(f)(y)\,dy\r|\\
    &\quad+\sum_{\{\gamma\in\zz_+^n:\ |\gamma|\leq s\}}\int_{\widetilde{B}}
    \lf|\frac{\partial_{(1)}^{\gamma}\widetilde{K}(x_1,y)}{\gamma!}
    (x-x_1)^{\gamma}\mathbf1_{\rn\setminus B_1}(y)\r.\\
    &\qquad\lf.-\frac{\partial_{(1)}^{\gamma}\widetilde{K}(x_0,y)}{\gamma!}
    (x-x_0)^{\gamma}\mathbf1_{\rn\setminus B_0}(y)\r||f(y)|\,dy\\
    &<\fz,
	\end{align*}
	which, together with Lemma \ref{int-B-P}, further implies that
	$\widetilde{T}_{B_0}(f)-\widetilde{T}_{B_1}(f)\in \mathcal{P}_s(\rn)$
	after changing values on a set of measure zero.
	This finishes the proof of the above claim.
	Based on this claim, in what follows, we write $\widetilde{T}$
	instead of $\widetilde{T}_{B_0}$ if there exists no confusion.
\end{remark}

The following conclusion is just \cite[Proposition 2.2]{jtyyz1},
which gives an equivalent characterization of
$JN_{(p,q,s)_\alpha}^{\mathrm{con}}(\rn)$.

\begin{lemma}\label{AC}
	Let $p$, $q\in[1,\infty)$, $s\in\zz_+$, and $\alpha\in \rr$.
	Then $f\in JN_{(p,q,s)_\alpha}^{\mathrm{con}}(\rn)$ if and only if
	$f\in L^q_{\mathrm{loc}}(\rn)$ and
	\begin{align*}
		[f]_{JN_{(p,q,s)\alpha}^{\mathrm{con}}(\rn)}
		:=\sup_{r\in(0,\infty)}\lf[\int_{\rn}\lf\{|B(y,r)|^{-\alpha}
		\lf[\fint_{B(y,r)}\lf|f(x)-P_{B(y,r)}^{(s)}(f)(x)\r|^q\,dx
		\r]^{\frac{1}{q}}\r\}^p\,dy\r]^{\frac{1}{p}}<\fz.
	\end{align*}
    Moreover, for any $f\in JN_{(p,q,s)_\alpha}^{\mathrm{con}}(\rn)$,
    $$
    \|f\|_{JN_{(p,q,s)\alpha}^{\mathrm{con}}(\rn)}
    \sim [f]_{JN_{(p,q,s)\alpha}^{\mathrm{con}}(\rn)},
    $$
	where the positive equivalence constants are independent of $f$.
\end{lemma}

Next, we show the boundedness of $\widetilde{T}$ on
$JN_{(p,q,s)_\alpha}^{\mathrm{con}}(\rn)$.

\begin{theorem}\label{thm-bdd-JN}
Let $p\in[1,\fz]$, $q\in(1,\fz)$, $s\in\zz_+$,
and $\az\in(-\fz,\frac{s+\dz}{n})$ with
$\dz\in(0,1]$ as in Definition \ref{def-s-k}.
Let $K$ be an $s$-order standard kernel
as in  Definition \ref{def-s-k} and
$T$ the $s$-order Calder\'on--Zygmund
singular integral operator  with kernel $K$.
Let $\widetilde{K}$ be as in \eqref{Kw-K}
and $\widetilde{T}$ an $s$-order modified
Calder\'on--Zygmund operator with kernel $\widetilde{K}$.
Then $\widetilde{T}$ is bounded on
$JN_{(p,q,s)_\alpha}^{\mathrm{con}}(\rn)$, namely,
there exists a positive constant $C$ such that,
for any $f\in JN_{(p,q,s)_\alpha}^{\mathrm{con}}(\rn)$,
$$\lf\|\widetilde{T}(f)\r\|_{JN_{(p,q,s)_\alpha}^{\mathrm{con}}(\rn)}
\le C \|f\|_{JN_{(p,q,s)_\alpha}^{\mathrm{con}}(\rn)}$$
if and only if, for any $\gamma\in\zz_+^n$ with $|\gamma|\leq s$,
$T^*(x^{\gamma})=0$.
\end{theorem}

\begin{proof}
Let $p$, $q$, $s$, $\alpha$, $\dz$,
and $\widetilde{T}$ be as in the present theorem.
We only prove the case $p\in[1,\fz)$ because the proof of the case
$p=\fz$ is similar. We first show the sufficiency.
Indeed, using \eqref{p}, it is easy to see that,
for any $g\in L^1_{\mathrm{loc}}(\rn)$ and any ball $B$ of $\rn$,
\begin{align*}
	\lf[\fint_B\lf|g(x)-P_B^{(s)}(g)(x)
	\r|^q\,dx\r]^{\frac{1}{q}}
	\sim\inf_{P\in \mathcal{P}_s(B)}
	\lf[\fint_B|g(x)-P(x)|^q\,dx\r]^{\frac{1}{q}},
\end{align*}
where $\mathcal{P}_s(B)$ denotes the set
of all polynomials with degree not greater than $s$ on $B$
(see, for instance, \cite[(2.12)]{jtyyz1}).
From this, \eqref{limit}, the Minkowski inequality,
and $P_{B(z,r)}^{(s)}(P)=P$ for any
$P\in \mathcal{P}_s(\rn)$, we deduce that,
for any given $f\in JN_{(p,q,s)_\alpha}^{\mathrm{con}}(\rn)$
and for any $r\in(0,\fz)$,
\begin{align}\label{F3+F4}
&\lf[\int_{\rn}\lf\{|B(z,r)|^{-\alpha}
\lf[\fint_{B(z,r)}\lf|\widetilde{T}(f)(x)
-P_{B(z,r)}^{(s)}(\widetilde{T}(f))(x)\r|^q\,dx
\r]^{\frac1q}\r\}^p\,dz\r]^{\frac1p}\\
&\quad\ls\lf[\int_{\rn}\lf\{|B(z,r)|^{-\alpha}
\lf[\fint_{B(z,r)}\lf|T_{(1)}\lf(\lf[f-P_{2B(z,r)}^{(s)}(f)\r]
\mathbf1_{2B(z,r)}\r)(x)
\r|^q\,dx\r]^{\frac1q}\r\}^p\,dz\r]^{\frac1p}\notag\\
&\qquad+\lf[\int_{\rn}\lf\{|B(z,r)|^{-\alpha}
\lf[\fint_{B(z,r)}\lf|E_{B(z,r)}(x)\r|^q\,dx
\r]^{\frac1q}\r\}^p\,dz\r]^{\frac1p}\noz\\
&\quad\sim{\rm F}_1+{\rm F}_2,\notag
\end{align}
where $T_{(1)}$ and $E_{B(z,r)}$ are as in \eqref{limit},
$$
{\rm F}_1:=\lf[\int_{\rn}\lf\{|B(z,r)|^{-\alpha}
\lf[\fint_{B(z,r)}\lf|T_{(1)}\lf(\lf[f-P_{2B(z,r)}^{(s)}(f)\r]
\mathbf1_{2B(z,r)}\r)(x)\r|^q\,dx
\r]^{\frac1q}\r\}^p\,dz\r]^{\frac1p},
$$
and
$$
{\rm F}_2:=\lf[\int_{\rn}\lf\{|B(z,r)|^{-\alpha}
\lf[\fint_{B(z,r)}\lf|E_{B(z,r)}(x)\r|^q\,dx
\r]^{\frac1q}\r\}^p\,dz\r]^{\frac1p}.
$$

We now estimate $\mathrm{F}_1$.
Indeed, by the fact that
$[f-P_{2B(z,r)}^{(s)}(f)]\mathbf1_{2B(z,r)}\in L^q(\rn)$, and
Lemmas \ref{Duo01} and \ref{AC}, we find that
\begin{align}\label{F3}
{\rm F}_1\ls\lf[\int_{\rn}\lf\{|B(z,r)|^{-\alpha}
\lf[\fint_{2B(z,r)}\lf|f(x)-P_{2B(z,r)}^{(s)}(f)(x)\r|^q\,dx
\r]^{\frac1q}\r\}^p\,dz\r]^{\frac1p}
\ls\|f\|_{JN_{(p,q,s)_\alpha}^{\mathrm{con}}(\rn)}.
\end{align}
This is a desired estimate of $\mathrm{F}_1$.

Next, we estimate $\mathrm{F}_2$.
Indeed, using \eqref{T2}, \eqref{JN-I-E},
the Minkowski inequality, Lemma \ref{AC},
and $\az\in(-\fz,\frac{s+\dz}{n})$, we find that
\begin{align*}
{\rm F}_2
&\ls\lf\{\int_{\rn}\lf[|B(z,r)|^{-\alpha}
r^{s+\dz}\int_{\rn\setminus 2B(z,r)}\frac{|f(y)
-P_{2B(z,r)}^{(s)}(f)(y)|}{|y-z|^{n+s+\dz}}\,dy\r]^p\,dz\r\}^{\frac1p}\\
&\ls\lf[\int_{\rn}\lf\{|B(z,r)|^{-\alpha}
\sum_{k\in\nn}\frac1{2^{k(s+\dz)}}\lf[\fint_{B(z,2^{k+1} r)}
\lf|f(y)-P_{B(z,2^{k+1}r)}^{(s)}(f)(y)\r|^q
\,dy\r]^{\frac{1}{q}}\r\}^p\,dz\r]^{\frac1p}\\
&\ls\sum_{k\in\nn}\frac1{2^{k(s+\dz)}}\lf[\int_{\rn}\lf\{|B(z,r)|^{-\alpha}
\lf[\fint_{B(z,2^{k+1} r)}\lf|f(y)-P_{B(z,2^{k+1}r)}^{(s)}(f)(y)\r|^q
\,dy\r]^{\frac{1}{q}}\r\}^p\,dz\r]^{\frac1p}\\
&\sim\sum_{k\in\nn}\frac1{2^{k(s+\dz-n\alpha)}}
\lf[\int_{\rn}\lf\{\lf|B(z,2^{k+1}r)\r|^{-\alpha}
\lf[\fint_{B(z,2^{k+1} r)}\lf|f(y)-P_{B(z,2^{k+1}r)}^{(s)}(f)(y)\r|^q
\,dy\r]^{\frac{1}{q}}\r\}^p\,dz\r]^{\frac1p}\\
&\ls\|f\|_{JN_{(p,q,s)_\alpha}^{\mathrm{con}}(\rn)}.
\end{align*}
From this, \eqref{F3+F4}, \eqref{F3}, and Lemma \ref{AC},
we deduce that, for any $f\in JN_{(p,q,s)_\alpha}^{\mathrm{con}}(\rn)$,
$\widetilde{T}(f)\in JN_{(p,q,s)_\alpha}^{\mathrm{con}}(\rn)$ and
$$\lf\|\widetilde{T}(f)\r\|_{JN_{(p,q,s)_\alpha}^{\mathrm{con}}(\rn)}
\ls\|f\|_{JN_{(p,q,s)_\alpha}^{\mathrm{con}}(\rn)}.$$
This finishes the proof of the sufficiency.

We now show the necessity. Indeed, if $\widetilde{T}$
is bounded on $JN_{(p,q,s)_\alpha}^{\mathrm{con}}(\rn)$,
then, for any $\gamma\in\zz_+^n$ with $|\gamma|\leq s$,
$$
\lf\|\widetilde{T}(x^{\gamma})\r\|_{JN_{(p,q,s)_\alpha}^{\mathrm{con}}(\rn)}
\ls \lf\|x^{\gamma}\r\|_{JN_{(p,q,s)_\alpha}^{\mathrm{con}}(\rn)}\sim0
$$
and hence $\widetilde{T}(x^{\gamma})\in \mathcal{P}_s(\rn)$.
Using this and Proposition \ref{Assume}, we find that,
for any $\gamma\in\zz_+^n$ with $|\gamma|\leq s$, $T^*(x^{\gamma})=0$,
which completes the proof of the necessity
and hence of Theorem \ref{thm-bdd-JN}.
\end{proof}

\begin{remark}\label{rem-JN-CZO}
\begin{enumerate}
\item[\rm (i)]	
We should point out that Peetre \cite[Remark 1.5]{Pe1996}
claimed, without giving a proof, that the Calder\'on--Zygmund operator on
the generalized Campanato space $L_s^{q,\Phi}$, with $s\in\nn$,
which when $\Phi(r):=r^{\alpha pn+n}$ coincides with
$\mathcal{C}_{\alpha,q,s}(\rn)$, is bounded
under some unspecific assumptions on the kernel
of the Calder\'on--Zygmund operator.
It is doubtful whether or not this claim
holds true for the Riesz transforms.

\item[\rm (ii)]	
By Proposition \ref{Assume}, we further conclude that
Theorem \ref{thm-bdd-JN} coincides with \cite[Theorem 4.21]{kk2013}
when $p=\fz$.
\end{enumerate}
\end{remark}

\section{Calder\'on--Zygmund Operators on Hardy-Type Spaces\label{S-C-Z-HK}}

This section is devoted to studying the boundedness of
Calder\'on--Zygmund operators on the
Hardy-kind space $HK_{(p,q,s)_\alpha}^{\mathrm{con}}(\rn)$
which proves the predual space of
$JN_{(p',q',s)_\alpha}^{\mathrm{con}}(\rn)$
(see, for instance, \cite[Theorem 4.10]{jtyyz1}).
To overcome the difficulty caused by
the fact that $\|\cdot\|_{HK_{(p,q,s)_{\alpha}}^{\mathrm{con}}(\mathbb{R}^n)}$
is no longer concave, we establish a criterion for the boundedness
of linear operators on $HK_{(p,q,s)_\alpha}^{\mathrm{con}}(\rn)$
via using molecules of $HK_{(p,q,s)_\alpha}^{\mathrm{con}}(\rn)$
and the dual relation $(HK_{(p,q,s)_{\alpha}}^{\mathrm{con}}(\mathbb{R}^n))^*
=JN_{(p',q',s)_\alpha}^{\mathrm{con}}(\mathbb{R}^n)$.

\subsection{Hardy-Type Spaces\label{Def-HK}}

We first recall the notion of the Hardy-kind space
$HK_{(p,q,s)_\alpha}^{\mathrm{con}}(\rn)$
which was introduced in \cite[Section 4]{jtyyz1}.

\begin{definition}\label{d3.2}
Let $p\in(1,\infty)$, $q\in(1,\infty]$, $s\in\zz_+$, and $\alpha\in\rr$.
A measurable function $a$ on $\rn$ is called a $(p,q,s)_\alpha$-\emph{atom}
supported in a cube $Q\subset \rn$ if
\begin{enumerate}
\item[\rm(i)] $\{x\in\rn:\ a(x)\neq 0\}\subset Q$;
\item[\rm(ii)] $\|a\|_{L^q(Q)}\leq|Q|^{\frac{1}{q}-\frac{1}{p}-\alpha}$;
\item[\rm(iii)] $\int_Qa(x)x^{\gamma}\,dx=0$
for any $\gamma\in\zz_+^n$ with $|\gamma|\leq s$.
\end{enumerate}
\end{definition}

The following conclusion is just \cite[Proposition 4.3]{jtyyz1}.

\begin{lemma}\label{p3.1}
Let $p\in(1,\infty)$, $q\in(1,\infty]$, $s\in\zz_+$, $\alpha\in\rr$,
$a$ be a $(p,q,s)_\alpha$-atom supported in the cube $Q\subset \rn$,
and, for any $f\in JN_{(p',q',s)_\alpha}^{\mathrm{con}}(\rn)$,
$$\langle a,f\rangle:=\int_{Q}f(x)a(x)\,dx,$$
where $\frac{1}{p}+\frac{1}{p'}=1=\frac{1}{q}+\frac{1}{q'}$. Then
$\langle a,\cdot\rangle\in
(JN_{(p',q',s)_\alpha}^{\mathrm{con}}(\rn))^*,$
where $(JN_{(p',q',s)_\alpha}^{\mathrm{con}}(\rn))^*$
denotes the dual space of $JN_{(p',q',s)_\alpha}^{\mathrm{con}}(\rn)$,
which is defined to be the set of all
continuous linear functionals on
$JN_{(p',q',s)_\alpha}^{\mathrm{con}}(\rn)$
equipped with the weak-$\ast$ topology.
\end{lemma}

Now, we say `$(p,q,s)_\alpha$-atom
$a\in (JN_{(p',q',s)_\alpha}^{\mathrm{con}}(\rn))^*$'
instead of `$\langle a,\cdot\rangle
\in (JN_{(p',q',s)_\alpha}^{\mathrm{con}}(\rn))^*$'.
Then we have the following notion of congruent $(p,q,s)_\alpha$-polymers.
Recall that, for any $\ell\in(0,\fz)$,
$\Pi_{\ell}(\rn)$ denotes the class of all collections of
interior pairwise disjoint subcubes $\{Q_j\}_j$ of $\rn$ with side length $\ell$.

\begin{definition}\label{d3.3}
Let $p\in(1,\infty)$, $q\in(1,\infty]$, $s\in\zz_+$, and $\alpha\in\rr$.
The \emph{space of congruent $(p,q,s)_\alpha$-polymers},
$\widetilde{HK}_{(p,q,s)_\alpha}^{\mathrm{con}}(\rn)$,
is defined to be the set of all
$g\in (JN_{(p',q',s)_\alpha}^{\mathrm{con}}(\rn))^*$ satisfying that
there exist an $\ell\in(0,\fz)$,
$(p,q,s)_{\alpha}$-atoms $\{a_j\}_j$ supported, respectively, in
$\{Q_j\}_j\in \Pi_{\ell}(\rn)$, and $\sum_{j}|\lambda_j|^p<\infty$
such that $g=\sum_{j}\lambda_ja_j$ in
$(JN_{(p',q',s)_\alpha}^{\mathrm{con}}(\rn))^*$,
where $\frac{1}{p}+\frac{1}{p'}=1=\frac{1}{q}+\frac{1}{q'}$.
Moreover, any $g\in\widetilde{HK}_{(p,q,s)_\alpha}^{\mathrm{con}}(\rn)$
is called a \emph{congruent $(p,q,s)_\alpha$-polymer} with its norm
$\|g\|_{\widetilde{HK}_{(p,q,s)_\alpha}^{\mathrm{con}}(\rn)}$ defined by setting
$$\|g\|_{\widetilde{HK}_{(p,q,s)_\alpha}^{\mathrm{con}}(\rn)}
:=\inf\lf(\sum_{j}|\lambda_j|^p\r)^{\frac{1}{p}},$$
where the infimum is taken over all decompositions of $g$ as above.
\end{definition}

\begin{definition}\label{d3.6}
Let $p\in(1,\infty)$, $q\in(1,\infty]$, $s\in\zz_+$, and $\alpha\in\rr$.
The \emph{Hardy-type space via congruent cubes},
$HK_{(p,q,s)_\alpha}^{\mathrm{con}}(\rn)$, is defined by setting
\begin{align*}
HK_{(p,q,s)_\alpha}^{\mathrm{con}}(\rn)
:=\lf\{g\in(JN_{(p',q',s)_\alpha}^{\mathrm{con}}(\rn))^*
:\ g=\sum_ig_i\text{ in }(JN_{(p',q',s)_\alpha}^{\mathrm{con}}(\rn))^*,\r.\\
\quad\lf.\{g_i\}_i\subset\widetilde{HK}_{(p,q,s)_\alpha}^{\mathrm{con}}(\rn),
\text{ and }\sum_i\|g_i
\|_{\widetilde{HK}_{(p,q,s)_\alpha}^{\mathrm{con}}(\rn)}<\infty\r\},
\end{align*}
where $\frac{1}{p}+\frac{1}{p'}=1=\frac{1}{q}+\frac{1}{q'}$.
Moreover, for any $g\in HK_{(p,q,s)_\alpha}^{\mathrm{con}}(\rn)$, let
$$\|g\|_{HK_{(p,q,s)_\alpha}^{\mathrm{con}}(\rn)}
:=\inf\sum_i\|g_i\|_{\widetilde{HK}_{(p,q,s)_\alpha}^{\mathrm{con}}(\rn)},$$
where the infimum is taken over all decompositions of $g$ as above.
\end{definition}

\begin{definition}\label{d3.5}
Let $p\in(1,\infty)$, $q\in(1,\infty]$, $s\in\zz_+$, and $\alpha\in\rr$.
The \emph{finite atomic Hardy-type space via congruent cubes},
$HK_{(p,q,s)_\alpha}^{\mathrm{con-fin}}(\rn)$, is defined to be the set of all
\begin{align*}
g=\sum_{j=1}^M\lambda_ja_j
\end{align*}
pointwisely, where $M\in\nn$, $\{a_j\}_{j=1}^M$
are $(p,q,s)_{\alpha}$-atoms supported, respectively,
in cubes $\{Q_j\}_{j=1}^M$ of $\rn$,
and $\{\lambda_j\}_{j=1}^M\subset \mathbb{C}$.
\end{definition}

The following conclusion is just
\cite[Propositions 4.9]{jtyyz1}.

\begin{lemma}\label{HK-normed}
Let $p\in(1,\infty)$, $q\in(1,\infty]$, $s\in\zz_+$,
and $\alpha\in\rr$. Then
$HK_{(p,q,s)_\alpha}^{\mathrm{con-fin}}(\rn)$
is dense in $HK_{(p,q,s)_\alpha}^{\mathrm{con}}(\rn)$
in the sense that, for any $g\in HK_{(p,q,s)_{\alpha}}^{\mathrm{con}}(\rn)$,
there exists a sequence $\{g_m\}_{m\in\nn}$ of
$HK_{(p,q,s)_{\alpha}}^{\mathrm{con-fin}}(\rn)$
such that
$$\lim_{m\to\fz}\|g-g_m\|_{HK_{(p,q,s)_{\alpha}}^{\mathrm{con}}(\rn)}=0.$$
\end{lemma}

The following conclusion is an immediate
corollary of \cite[Theorem 4.10]{jtyyz1}
and \cite[Proposition 2.7]{jtyyz1}; we omit the details here.

\begin{lemma}\label{t3.9}
	Let $p\in(1,\infty)$, $q\in(1,\infty)$,
	$\frac{1}{p}+\frac{1}{p'}=1=\frac{1}{q}+\frac{1}{q'}$,
	$s\in\zz_+$, and $\alpha\in\rr$.
	Then $$\lf(HK_{(p',q',s)_{\alpha}}^{\mathrm{con}}(\rn)\r)^*
	=JN_{(p,q,s)_{\alpha}}^{\mathrm{con}}(\rn)$$
	with equivalent norms in the following sense:
	\begin{enumerate}
		\item[\rm (i)] any given
		$f\in JN_{(p,q,s)_{\alpha}}^{\mathrm{con}}(\rn)$
		induces a linear functional $\mathcal{L}_f$ given by setting,
		for any $g\in HK_{(p',q',s)_{\alpha}}^{\mathrm{con}}(\rn)$
		and $\{g_i\}_i\subset
		\widetilde{HK}_{(p',q',s)_{\alpha}}^{\mathrm{con}}(\rn)$
		with $g=\sum_i g_i$ in$(JN_{(p,q,s)_{\alpha}}^{\mathrm{con}}(\rn))^*$,
		\begin{align}\label{3.0x}
			\mathcal{L}_f(g):=\langle g,f\rangle
			=\sum_i\langle g_i,f\rangle.
		\end{align}
		Moreover,
		for any $g\in HK_{(p',q',s)_{\alpha}}^{\mathrm{con-fin}}(\rn)$,
        $\mathcal{L}_f(g)=\int_{\rn}f(x)g(x)\,dx$
		and there exists a positive constant $C$ such that
		\begin{align*}
			\lf\|\mathcal{L}_f\r
			\|_{(HK_{(p',q',s)_{\alpha}}^{\mathrm{con}}(\rn))^*}
			\leq C\|f\|_{JN_{(p,q,s)_{\alpha}}^{\mathrm{con}}(\rn)};
		\end{align*}
		\item[\rm (ii)] conversely, for any continuous
		linear functional $\mathcal{L}$ on
		$HK_{(p',q',s)_{\alpha}}^{\mathrm{con}}(\rn)$,
		there exists a unique
		$f\in JN_{(p,q,s)_{\alpha}}^{\mathrm{con}}(\rn)$ such that,
		for any $g\in HK_{(p',q',s)_{\alpha}}^{\mathrm{con-fin}}(\rn)$,
		$\mathcal{L}(g)=\int_{\rn}f(x)g(x)\,dx$
		and there exists a positive constant $C$ such that
		$$
		\|f\|_{JN_{(p,q,s)_{\alpha}}^{\mathrm{con}}(\rn)}
		\leq C\|\mathcal{L}\|_{(HK_{(p',q',s)_{\alpha}}^{\mathrm{con}}(\rn))^*}.
		$$
	\end{enumerate}
\end{lemma}

The following conclusion can immediately be deduced
from \cite[Corollary 4.11]{jtyyz1},
which plays a key role in the proof of
Theorem \ref{T-bounded-HK} below; we omit the details here.

\begin{proposition}\label{Coro-JN}
Let $p\in(1,\infty)$, $q\in(1,\infty)$, $s\in\zz_+$, and $\alpha\in\rr$. Then
$\|\cdot\|_{HK_{(p,q,s)_\alpha}^{\mathrm{con}}(\rn)}$
and $\|\cdot\|_{(JN_{(p',q',s)_\alpha}^{\mathrm{con}}(\rn))^{*}}$
are equivalent norms of $HK_{(p,q,s)_\alpha}^{\mathrm{con}}(\rn)$.
\end{proposition}

Now, we show that $(HK_{(p,q,s)_\alpha}^{\mathrm{con}}(\rn),
\|\cdot\|_{HK_{(p,q,s)_\alpha}^{\mathrm{con}}(\rn)})$
is a Banach space.

\begin{proposition}\label{HK-banach}
	Let $p\in(1,\infty)$, $q\in(1,\infty]$,
	$s\in\zz_+$, and $\alpha\in\rr$. Then
	$$\lf(HK_{(p,q,s)_\alpha}^{\mathrm{con}}(\rn),
	\|\cdot\|_{HK_{(p,q,s)_\alpha}^{\mathrm{con}}(\rn)}\r)$$
	is a Banach space.
\end{proposition}

\begin{proof}
	Let $p$, $q$, $s$, and $\alpha$ be as in the present proposition.
	Since \cite[Proposition 4.8]{jtyyz1} proves that
	$(HK_{(p,q,s)_\alpha}^{\mathrm{con}}(\rn),
	\|\cdot\|_{HK_{(p,q,s)_\alpha}^{\mathrm{con}}(\rn)})$
	is a linear normed space, from this, we deduce that,
	to prove the present proposition, it suffices to show
	that it is complete.
	To this end, let $\{g^{(m)}\}_{m\in\nn}$ be any given
	Cauchy sequence of $(HK_{(p,q,s)_\alpha}^{\mathrm{con}}(\rn),
	\|\cdot\|_{HK_{(p,q,s)_\alpha}^{\mathrm{con}}(\rn)})$,
	and $g^{(0)}$ the zero element of
	$HK_{(p,q,s)_\alpha}^{\mathrm{con}}(\rn)$.
	Without loss of generality, we may assume that,
	for any $m\in\nn$,
	\begin{align}\label{gm-01}
		\lf\|g^{(m)}-g^{(m-1)}
		\r\|_{HK_{(p,q,s)_\alpha}^{\mathrm{con}}(\rn)}
		< \frac{1}{2^m}.
	\end{align}
	By $HK_{(p,q,s)_\alpha}^{\mathrm{con}}(\rn)
	\subset (JN_{(p',q',s)_\alpha}^{\mathrm{con}}(\rn))^*$,
	and Proposition \ref{Coro-JN},
	we conclude that $\{g^{(m)}\}_{m\in\nn}$ is also a
	Cauchy sequence of $(HK_{(p,q,s)_\alpha}^{\mathrm{con}}(\rn),
	\|\cdot\|_{(JN_{(p',q',s)_\alpha}^{\mathrm{con}}(\rn))^*})$
	and hence, from the completeness of
	$(JN_{(p',q',s)_\alpha}^{\mathrm{con}}(\rn))^*$,
	we deduce that there exists an element
	$g\in (JN_{(p',q',s)_\alpha}^{\mathrm{con}}(\rn))^*$
	such that
	\begin{align}\label{3.1x}
	\lim_{N\to\fz}\lf\|g-g^{(N)}
	\r\|_{(JN_{(p',q',s)_\alpha}^{\mathrm{con}}(\rn))^*}=0
	\end{align}
	and, for any $f\in JN_{(p',q',s)_\alpha}^{\mathrm{con}}(\rn)$,
	\begin{align}\label{g-00}
		\langle g,f\rangle=\lim_{N\to\fz}\lf\langle g^{(N)},f\r\rangle.
	\end{align}

    Next, we show that $g\in HK_{(p,q,s)_\alpha}^{\mathrm{con}}(\rn)$.
	Indeed, using \eqref{gm-01} and the definition
	$HK_{(p,q,s)_\alpha}^{\mathrm{con}}(\rn)$,
	we find that, for any $m\in \nn$,
	there exists a sequence $\{g^{(m)}_{i}\}_{i\in\nn}$ of
	$\widetilde{HK}_{(p,q,s)_\alpha}^{\mathrm{con}}(\rn)$
	such that
	\begin{align}\label{g-04}
		g^{(m)}-g^{(m-1)}=\sum_{i\in\nn}g^{(m)}_{i}
	\end{align}
	in $(JN_{(p',q',s)_\alpha}^{\mathrm{con}}(\rn))^*$, and
	\begin{align}\label{gm-02}
		\sum_{i\in\nn}\lf\|g^{(m)}_{i}
		\r\|_{\widetilde{HK}_{(p,q,s)_\alpha}^{\mathrm{con}}(\rn)}
		< \frac{1}{2^{m}}.
	\end{align}
Moreover, from \eqref{gm-01}, we deduce that,
for any $f\in JN_{(p',q',s)_\alpha}^{\mathrm{con}}(\rn)$ and $N\in\nn$,
\begin{align*}
&\lf\|\sum_{m=1}^{N}\lf[g^{(m)}-g^{(m-1)}\r]
\r\|_{(JN_{(p',q',s)_\alpha}^{\mathrm{con}}(\rn))^*}\\
&\quad\sim\lf\|\sum_{m=1}^{N}\lf[g^{(m)}-g^{(m-1)}\r]
\r\|_{HK_{(p,q,s)_\alpha}^{\mathrm{con}}(\rn)}
\ls\sum_{m\in\nn}\lf\|g^{(m)}-g^{(m-1)}
\r\|_{HK_{(p,q,s)_\alpha}^{\mathrm{con}}(\rn)}
\ls1.
\end{align*}
Using this and the completeness of
$(JN_{(p',q',s)_\alpha}^{\mathrm{con}}(\rn))^*$, we find that
$$\sum_{m\in\nn}\lf[g^{(m)}-g^{(m-1)}\r]\in
\lf(JN_{(p',q',s)_\alpha}^{\mathrm{con}}(\rn)\r)^*,$$
which, together with \eqref{g-00} and \eqref{g-04}, implies that,
for any $f\in JN_{(p',q',s)_\alpha}^{\mathrm{con}}(\rn)$,
\begin{align*}
	\langle g,f\rangle
	&=\lim_{N\to\fz}\lf\langle g^{(N)},f\r\rangle
	=\lim_{N\to\fz}\lf\langle
	\sum_{m=1}^{N}\lf[g^{(m)}-g^{(m-1)}\r],f\r\rangle\\
	&=\lf\langle \sum_{m\in\nn}\lf[g^{(m)}-g^{(m-1)}\r],f\r\rangle
	=\lf\langle \sum_{m\in\nn}\sum_{i\in\nn}g^{(m)}_{i},f\r\rangle.
\end{align*}
Thus, $g=\sum_{m\in\nn}\sum_{i\in\nn}g^{(m)}_{i}$
in $(JN_{(p',q',s)_\alpha}^{\mathrm{con}}(\rn))^*$,
which, combined with \eqref{gm-02}, further implies that
$g\in HK_{(p,q,s)_\alpha}^{\mathrm{con}}(\rn)$.
By this, \eqref{3.1x}, and Lemma \ref{Coro-JN}, we find that
\begin{align*}
	\lim_{N\to\fz}\lf\|g-g^{(N)}
	\r\|_{HK_{(p,q,s)_\alpha}^{\mathrm{con}}(\rn)}=0
\end{align*}
and hence completes the proof of Proposition \ref{HK-banach}.
\end{proof}

\subsection{Molecules of
$HK_{(p,q,s)_\alpha}^{\mathrm{con}}(\rn)$\label{Mo}}

In this subsection, we introduce the notion of
$(p,q,s,\alpha,\epsilon)$-molecules
of $HK_{(p,q,s)_\alpha}^{\mathrm{con}}(\rn)$,
which plays an important role in the proof of the boundedness
of Calder\'on--Zygmund operators
on $HK_{(p,q,s)_\alpha}^{\mathrm{con}}(\rn)$.

Via borrowing some ideas from
\cite[Definition 1.2]{HYZ2009}, we introduce the following
$(p,q,s,\alpha,\epsilon)$-molecule.
In what follows, for any $z\in\rn$ and $r\in(0,\fz)$,
we denote by $Q_z(r)$ the cube with center $z$ and side length $r$.

\begin{definition}\label{Def-mole}
Let $p\in (1,\fz)$, $q\in(1,\fz]$, $s\in\zz_+$,
$\alpha\in(\frac{1}{q}-\frac{1}{p},\fz)$,
and $\epsilon\in (0,\fz)$.
A measurable function $M$ on $\rn$ is called
a \textit{$(p,q,s,\alpha,\epsilon)$-molecule}
centered at the cube $Q_z(r)$ with center
$z\in\rn$ and length $r\in(0,\fz)$ if
\begin{enumerate}
\item[\rm (i)] $\|M\mathbf{1}_{Q_z(r)}\|_{L^q(\rn)}
\leq |Q_z(r)|^{\frac{1}{q}-\frac{1}{p}-\alpha}$;
\item[\rm (ii)] for any $j\in \nn$, $\|M\mathbf{1}_{Q_z(2^{j}r)
\setminus Q_z(2^{j-1}r)}\|_{L^q(\rn)}
\leq 2^{\frac{jn}{\epsilon}(\frac{1}{q}-\frac{1}{p}-\alpha)}
|Q_z(r)|^{\frac{1}{q}-\frac{1}{p}-\alpha}$;
\item[\rm (iii)] $\int_{\rn}M(x)x^{\gamma}\,dx=0$
for any $\gamma\in\zz_+^n$ with $|\gamma|\leq s$.
\end{enumerate}
\end{definition}

Now, we show that any
$(p,q,s,\alpha,\epsilon)$-molecule
induces an element of $(JN_{(p',q',s)_\alpha}^{\mathrm{con}}(\rn))^*$.

\begin{proposition}\label{M-JN1}
Let $p\in (1,\fz)$, $q\in(1,\fz]$,
$\frac{1}{p}+\frac{1}{p'}=1=\frac{1}{q}+\frac{1}{q'}$, $s\in\zz_+$,
$\alpha\in(\frac{1}{q}-\frac{1}{p},\fz)$,
and $\epsilon\in (0,1)$ be such that
$\frac{1}{\epsilon}(\frac{1}{q}-\frac{1}{p}-\alpha)+\frac{1}{q'}+\frac{s}{n}<0$.
Let $M$ be a $(p,q,s,\alpha,\epsilon)$-molecule.
Then, for any $f\in JN_{(p',q',s)_\alpha}^{\mathrm{con}}(\rn)$,
$$\langle M,f\rangle:=\int_{\rn}f(x)M(x)\,dx$$
induces a continuous linear functional on
$JN_{(p',q',s)_\alpha}^{\mathrm{con}}(\rn)$.
\end{proposition}

\begin{proof}
Let $p,$ $q$, $s$, $\alpha$, and $\epsilon$ be as in the present proposition.
Let $M$ be a $(p,q,s,\alpha,\epsilon)$-molecule
centered at the cube $Q_z(r)$ with $z\in\rn$ and $r\in(0,\fz)$.
Let $L_0:=Q_z(r)$ and, for any $j\in\nn$,
let $L_j:=Q_z(2^jr)\setminus Q_z(2^{j-1}r)$.
Then, from Definition \ref{Def-mole}(iii),
the H\"older inequality, Definition \ref{Def-mole}(ii), Lemma \ref{sum-g}
with $$\theta:=2^{\frac{n}{\epsilon}
	(\frac{1}{q}-\frac{1}{p}-\alpha)+\frac{n}{q'}}\in(0,2^{-s}),$$
$\frac{1}{\epsilon}
(\frac{1}{q}-\frac{1}{p}-\alpha)+\frac{1}{q'}+\frac{s}{n}<0$,
and $\alpha\in(\frac{1}{q}-\frac{1}{p},\fz)$,
we deduce that, for any
$f\in JN_{(p',q',s)_\alpha}^{\mathrm{con}}(\rn)$,
\begin{align*}
&\lf|\int_{\rn}M(x)f(x)\,dx\r|\\
&\quad=\lf|\int_{\rn}M(x)\lf[f(x)-P_{Q_z(r)}^{(s)}(f)(x)\r]\,dx\r|\\
&\quad\leq \sum_{j\in\zz_+}\int_{L_j}\lf|M(x)
\lf[f(x)-P_{Q_z(r)}^{(s)}(f)(x)\r]\r|\,dx\\
&\quad\leq \sum_{j\in\zz_+}\lf\|M\mathbf{1}_{L_j}
\r\|_{L^q(\rn)}\lf|Q_z(2^{j}r)\r|^{\frac{1}{q'}}
\lf[\fint_{Q_z(2^{j}r)}\lf|f(x)
-P_{Q_z(r)}^{(s)}(f)(x)\r|^{q'}\,dx\r]^{\frac{1}{q'}}\\
&\quad\leq \sum_{j\in\zz_+}2^{jn[\frac{1}{\epsilon}
(\frac{1}{q}-\frac{1}{p}-\alpha)+\frac{1}{q'}]}
\lf|Q_z(r)\r|^{\frac{1}{p'}-\alpha}
\lf[\fint_{Q_z(2^{j}r)}\lf|f(x)
-P_{Q_z(r)}^{(s)}(f)(x)\r|^{q'}\,dx\r]^{\frac{1}{q'}}\\
&\quad\ls \sum_{j\in\zz_+}2^{jn[\frac{1}{\epsilon}
(\frac{1}{q}-\frac{1}{p}-\alpha)+\frac{1}{q'}]}
\lf|Q_z(r)\r|^{\frac{1}{p'}-\alpha}
\lf[\fint_{Q_z(2^{j}r)}\lf|f(x)
-P_{Q_z(2^jr)}^{(s)}(f)(x)\r|^{q'}\,dx\r]^{\frac{1}{q'}}\\
&\quad\ls \sum_{j\in\zz_+}2^{jn(\frac{1}{\epsilon}-1)
(\frac{1}{q}-\frac{1}{p}-\alpha)}
\|f\|_{JN_{(p',q',s)_\alpha}^{\mathrm{con}}(\rn)}
\lesssim \|f\|_{JN_{(p',q',s)_\alpha}^{\mathrm{con}}(\rn)},
\end{align*}
which implies the desired conclusion and hence
completes the proof of Proposition \ref{M-JN1}.
\end{proof}

Next, we show that any
$(p,q,s,\alpha,\epsilon)$-molecule $M$
is in $HK_{(p,q,s)_\alpha}^{\mathrm{con}}(\rn)$
via borrowing some ideas from
\cite[Theorem 7.2]{L} and \cite[Lemma 6.3]{N17}.
To this end, we first establish several technical lemmas.
The following lemma is widely used
(see, for instance, \cite[p.\,86]{L}),
whose proof is similar to that of
\cite[Lemma 4.1]{L}; we omit the details here.

\begin{lemma}\label{lem-4.13x}
Let $s\in\zz_+$. Then
there exists a positive constant $C$ such that,
for any cube (or ball) $Q$ of $\rn$, any $f\in L^1_{\mathrm{loc}}(\rn)$,
and any $x\in 2Q\setminus Q$,
\begin{align}\label{pq2}
\lf|P_{2Q\setminus Q}^{(s)}(f)(x)\r|
\leq C\fint_{2Q\setminus Q}|f(x)|\,dx,
\end{align}	
where $P_{2Q\setminus Q}^{(s)}(f)$ denotes the
unique polynomial $P\in \mathcal{P}_s(\rn)$ such that,
for any $\gamma\in\zz_+^n$ with $|\gamma|\leq s$,
\begin{align*}
	\int_{2Q\setminus Q}\lf[f(x)
	-P_{2Q\setminus Q}^{(s)}(f)(x)\r]x^{\gamma}\,dx=0.
\end{align*}
\end{lemma}

The following Abel transformation is well known.

\begin{lemma}\label{Abel}
Let $\{a_j\}_{j\in\zz_+}$, $\{b_j\}_{j\in\zz_+}\subset \cc$. Then,
for any $k\in\nn$,
$$
\sum_{j=1}^ka_jb_j=a_k\sum_{j=1}^k b_j
-\sum_{j=1}^{k-1}\lf[\lf(\sum_{i=1}^jb_i\r)(a_{j+1}-a_j)\r].
$$
\end{lemma}

We also need the following technical lemma.

\begin{lemma}\label{BMO-JN}
Let $u\in (1,\fz)$, $v\in[1,\fz)$, $s\in\zz_+$,
and $\alpha\in \rr$.
Then there exists a positive constant
$C$ such that, for any
cube $Q_z(r)\subset \rn$ with $z\in\rn$ and $r\in (0,\fz)$,
any $k\in\nn$, and any $f\in JN_{(u,v,s)_\alpha}^{\mathrm{con}}(\rn)$,
\begin{align*}
\lf[\fint_{Q_z(2^kr)}\lf|f(x)
-P_{Q_z(r)}^{(s)}(f)(x)\r|^v\,dy\r]^{\frac{1}{v}}
\leq C k\lf[2^{ks}+2^{nk(\alpha-\frac{1}{u})}\r]
r^{n(\alpha-\frac{1}{u})}\|f\|_{JN_{(u,v,s)_\alpha}^{\mathrm{con}}(\rn)}.
\end{align*}
\end{lemma}

\begin{proof}
Let $u$, $v$, $s$, and $\alpha$
be as in the present lemma.
By Lemma \ref{N12-lemma} with $\lambda:=2$, \eqref{p},
the H\"older inequality, and the definition
of $\|\cdot\|_{JN_{(u,v,s)_\alpha}^{\mathrm{con}}(\rn)}$,
we find that, for any cube $Q_z(r)\subset \rn$
with $z\in\rn$ and $r\in (0,\fz)$, any $k\in \nn$, and any
$f\in JN_{(u,v,s)_\alpha}^{\mathrm{con}}(\rn)$,
\begin{align*}
&\lf[\fint_{Q_z(2^kr)}\lf|f(x)
-P_{Q_z(r)}^{(s)}(f)(x)\r|^v\,dx\r]^{\frac{1}{v}}\\
&\quad\leq\lf[\fint_{Q_z(2^kr)}\lf|f(x)
-P_{Q_z(2^kr)}^{(s)}(f)(x)\r|^v\,dx\r]^{\frac{1}{v}}
+\sum_{j=1}^k\lf\|P_{Q_z(2^jr)}^{(s)}(f)
-P_{Q_z(2^{j-1}r)}^{(s)}(f)\r\|_{L^{\fz}(Q_z(2^kr))}\\
&\quad\ls\lf[\fint_{Q_z(2^kr)}\lf|f(x)
-P_{Q_z(2^kr)}^{(s)}(f)(x)\r|^v\,dx\r]^{\frac{1}{v}}\\
&\qquad+\sum_{j=1}^k2^{(k-j+1)s}\lf\|P_{Q_z(2^{j-1}r)}^{(s)}
\lf(f-P_{Q_z(2^jr)}^{(s)}(f)\r)\r\|_{L^{\fz}(Q_z(2^{j-1}r))}\\
&\quad\ls\lf[\fint_{Q_z(2^kr)}\lf|f(x)
-P_{Q_z(2^kr)}^{(s)}(f)(x)\r|^v\,dx\r]^{\frac{1}{v}}\\
&\qquad+\sum_{j=1}^k2^{(k-j+1)s}\fint_{Q_z(2^{j-1}r)}
\lf|f(x)-P_{Q_z(2^jr)}^{(s)}(f)(x)\r|\,dx\\
&\quad\ls\sum_{j=1}^k 2^{(k-j+1)s}\lf[\fint_{Q_z(2^{j}r)}
\lf|f(x)-P_{Q_z(2^jr)}^{(s)}(f)(x)\r|^v\,dx\r]^{\frac{1}{v}}\\
&\quad\sim\sum_{j=1}^k 2^{(k-j+1)s}\lf|Q_z(2^jr)\r|^{\alpha-\frac{1}{u}}
\lf|Q_z(2^jr)\r|^{-\alpha+\frac{1}{u}}\lf[\fint_{Q_z(2^{j}r)}
\lf|f(x)-P_{Q_z(2^jr)}^{(s)}(f)(x)\r|^v\,dx\r]^{\frac{1}{v}}\\
&\quad\ls k\lf[2^{ks}+2^{nk(\alpha-\frac{1}{u})}\r]r^{n(\alpha-\frac{1}{u})}
\|f\|_{JN_{(u,v,s)_\alpha}^{\mathrm{con}}(\rn)}.
\end{align*}
This finishes the proof of Lemma \ref{BMO-JN}.
\end{proof}

\begin{proposition}\label{HK-mole}
Let $p\in (1,\fz)$, $q\in(1,\fz]$,
$\frac{1}{q}+\frac{1}{q'}=1$, $s\in\zz_+$,
$\alpha\in (\frac{1}{q}-\frac{1}{p},\fz)$,
and $\epsilon\in(0,1)$ satisfy
$\frac{1}{\epsilon}
(\frac{1}{q}-\frac{1}{p}-\alpha)+\frac{1}{q'}+\frac{s}{n}<0$.
If $M$ is a $(p,q,s,\alpha,\epsilon)$-molecule,
then $M\in HK_{(p,q,s)_\alpha}^{\mathrm{con}}(\rn)$
and $\|M\|_{HK_{(p,q,s)_\alpha}^{\mathrm{con}}(\rn)}\leq C$,
where the positive constant $C$ is independent of $M$.
\end{proposition}

\begin{proof}
Let	$p$, $q$, $s$, $\alpha$,
and $\epsilon$ be as in the present proposition.
Also, let $M$ be a $(p,q,s,\alpha,\epsilon)$-molecule
centered at the cube $Q_z(r)$ with $z\in\rn$ and $r\in(0,\fz)$.
Without loss of generality, we may assume that $z:=\mathbf{0}$
and $M$ is a real-valued function.
Let $L_0:=Q_{\mathbf{0}}(r)$ and, for any $j\in\nn$,
$L_j:=Q_{\mathbf{0}}(2^jr)\setminus Q_{\mathbf{0}}(2^{j-1}r)$.
For any $j\in\zz_+$ and $x\in \rn$, let
$P_{L_j}^{(s)}(M)$ denote the
unique polynomial of $\mathcal{P}_s(\rn)$
such that, for any $\gamma\in\zz_+^n$ with
$|\gamma|\leq s$, and $j\in\zz_+$,
\begin{align}\label{HK-M-1}
\int_{L_j}\lf[M(x)-P_{L_j}^{(s)}(M)(x)\r]x^{\gamma}\,dx=0.
\end{align}

To prove the present proposition, we first recall some known
facts on $P_j:=P_{L_j}^{(s)}(M)$ with $j\in\zz_+$
(see, for instance, \cite[pp.\,85-88]{L}).
For any given $j\in\zz_+$,
let $\{\phi_\nu^{(j)}:\ \nu\in\zz_+^n\text{ and }|\nu|\leq s\}$
be the \textit{orthogonal polynomials} with weight $\frac{1}{|L_j|}$
by means of the Gram--Schmidt method from
$\{x^\nu:\ \nu\in\zz_+^n\text{ and }|\nu|\leq s\}$
restricted on $L_j$, namely,
$\{\phi_\nu^{(j)}:\ \nu\in\zz_+^n\text{ and }
|\nu|\leq s\}\subset \mathcal{P}_s(\rn)$
and, for any $\nu_1$, $\nu_2\in\zz_+^n$ with
$|\nu_1|\leq s$ and $|\nu_2|\leq s$,
\begin{align*}
	\lf\langle \phi_{\nu_1}^{(j)}, \phi_{\nu_2}^{(j)}\r\rangle_{L_j}
	:=\frac{1}{|L_j|}\int_{L_j}
	\phi_{\nu_1}^{(j)}(x)\phi_{\nu_2}^{(j)}(x)\,dx
	=\begin{cases}
		\displaystyle
		1,\ & \nu_1=\nu_2,\\
		\displaystyle
		0,\ & \nu_1\neq \nu_2.
	\end{cases}
\end{align*}
It is easy to see that, for any $x\in L_j$ with $j\in\zz_+$,
\begin{align}\label{P-j}
	P_j(x)=\sum_{\{\nu\in\zz_+^n:\ |\nu|\leq s\}}
	\lf\langle M,\phi_{\nu}^{(j)}
	\r\rangle_{L_j}\phi_\nu^{(j)}(x).
\end{align}
Moreover, for any given $j\in\zz_+$, let
$\{\psi_{\nu}^{(j)}:\ \nu\in\zz_+^n\text{ and }|\nu|\leq s\}$
be the \emph{dual basis} of
$\{x^{\nu}:\ \nu\in\zz_+^n\text{ and }|\nu|\leq s\}$
restricted on $L_j$ with respect to the weight $\frac{1}{|L_j|}$, namely,
$\psi_{\nu}^{(j)}\in \mathcal{P}_s(\rn)$ and,
for any $\nu_1$, $\nu_2\in\zz_+^n$ with $|\nu_1|\leq s$ and $|\nu_2|\leq s$,
\begin{align}\label{4.10x}
	\lf\langle \psi_{\nu_1}^{(j)}, x^{\nu_2}\r\rangle_{L_j}
	:=\frac{1}{|L_j|}\int_{L_j}\psi_{\nu_1}^{(j)}(x)x^{\nu_2}\,dx
	=\begin{cases}
		\displaystyle
		1,\ & \nu_1=\nu_2,\\
		\displaystyle
		0,\ & \nu_1\neq \nu_2.
	\end{cases}
\end{align}
Using the definition of the dual basis, we find that,
if, for any $\nu\in\zz_+^n$ with $|\nu|\leq s$,
$j\in\zz_+$, and $x\in L_j$, $\phi_{\nu}^{(j)}(x)
=\sum_{\{\gamma\in\zz_+^n:\ |\gamma|\leq s\}}
m_{\nu,\gamma}^{(j)}x^{\gamma}$
with $\{m_{\nu,\gamma}^{(j)}\}\subset \rr$, then
$$
\psi_{\nu}^{(j)}=\sum_{\{\gamma\in\zz_+^n:\ |\gamma|\leq s\}}
m_{\gamma,\nu}^{(j)}\phi_{\gamma}^{(j)}.
$$
From this and \eqref{P-j},
we deduce that, for any $j\in\zz_+$,
\begin{align}\label{3.8x}
	P_j\mathbf{1}_{L_j}=\sum_{\{\nu\in\zz_+^n:\ |\nu|\leq s\}}
	\lf\langle M,x^{\nu}\r\rangle_{L_j}\psi_{\nu}^{(j)}\mathbf{1}_{L_j}.
\end{align}

To show $M\in HK_{(p,q,s)_{\alpha}}^{\mathrm{con}}(\rn)$,
we first estimate $M\mathbf{1}_{Q_{\mathbf{0}}(2^lr)}$
for any $l\in\nn$. Indeed, for any $l\in\nn$ and $x\in\rn$,
\begin{align}\label{M-l'}
	M\mathbf{1}_{Q_{\mathbf{0}}(2^lr)}(x)
	&=\sum_{j=1}^l \lf[M\mathbf{1}_{L_j}(x)-P_j\mathbf{1}_{L_j}(x)\r]
	+\sum_{j=0}^{l}P_j\mathbf{1}_{L_j}(x)\\
	&=\sum_{j=1}^l \alpha_j(x)
	+\sum_{j=0}^{l}P_j\mathbf{1}_{L_j}(x),\noz
\end{align}
where $\alpha_j:=(M-P_j)\mathbf{1}_{L_j}$.
We then show that, for any $j\in\zz_+$, $\alpha_j$
is a $(p,q,s)_\alpha$-atom multiplying a positive constant.
Indeed, by \eqref{HK-M-1}, we conclude that, for any $j\in\zz_+$
and  $\gamma\in\zz_+^n$ with $|\gamma|\leq s$,
\begin{align}\label{3.10x}
	\int_{\rn}\alpha_j(x)x^\gamma\,dx=0
	\quad\text{and}\quad \supp\,(\alpha_j)\subset Q_{\mathbf{0}}(2^jr).
\end{align}
Thus, for any $j\in\zz_+$,
$\alpha_j$ satisfies (i) and (iii) of Definition \ref{d3.2}.
Moreover, we show that, for any $j\in\zz_+$, $\alpha_j$
satisfies Definition \ref{d3.2}(ii) in the sense
of multiplying a positive constant. Indeed,
from \eqref{pq2}, the H\"older inequality,
and Definition \ref{Def-mole}(ii),
we deduce that, for any $j\in\zz_+$,
\begin{align*}
	\|\alpha_j\|_{L^q(\rn)}
	&\leq (1+C)\lf\|M\mathbf{1}_{L_j}\r\|_{L^q(\rn)}
	\leq (1+C)2^{\frac{jn}{\epsilon}(\frac{1}{q}-\frac{1}{p}-\alpha)}
	|Q_{\mathbf{0}}(r)|^{\frac{1}{q}-\frac{1}{p}-\alpha}\\
	&=(1+C)2^{jn[(\frac{1}{\epsilon}-1)(\frac{1}{q}-\frac{1}{p}-\alpha)]}
	\lf|Q_{\mathbf{0}}(2^jr)\r|^{\frac{1}{q}-\frac{1}{p}-\alpha}
	=\lambda_j\lf|Q_{\mathbf{0}}(2^jr)
	\r|^{\frac{1}{q}-\frac{1}{p}-\alpha},
\end{align*}
where $C$ is as in \eqref{pq2}, and
\begin{align}\label{lam-j}
\lambda_j:=(1+C)2^{jn(\frac{1}{\epsilon}-1)(\frac{1}{q}-\frac{1}{p}-\alpha)}.
\end{align}
This and \eqref{3.10x} imply that, for any $j\in\zz_+$,
$A_j:=\frac{\alpha_j}{\lambda_j}$
is a $(p,q,s)_\alpha$-atom.

To end the estimation of $M\mathbf{1}_{Q_{\mathbf{0}}(2^lr)}$
for any $l\in\nn$, it remains to consider $\sum_{j=0}^{l}P_j\mathbf{1}_{L_j}$
for any $l\in\nn$. For any
$\nu\in\zz_+^n$ with $|\nu|\leq s$, and $j\in\zz_+$, let
\begin{align}\label{4.12x}
	\eta^{(j)}_{\nu}:=\int_{\rn\setminus Q_{\mathbf{0}}(2^jr)}M(x)x^{\nu}
	\,dx\quad\text{and}\quad \eta_{\nu}^{(-1)}
	:=\int_{\rn}M(x)x^{\nu}\,dx=0.
\end{align}
Then, by \eqref{3.8x}, we conclude that,
for any $j\in\zz_+$,
\begin{align*}
	P_j\mathbf{1}_{L_j}
	=\sum_{\{\nu\in\zz_+^n:\ |\nu|\leq s\}}
	\frac{1}{|L_j|}\int_{L_j}M(x)x^{\nu}
	\,dx\,\psi_{\nu}^{(j)}\mathbf{1}_{L_j}
	=\sum_{\{\nu\in\zz_+^n:\ |\nu|\leq s\}}
	\frac{\eta_{\nu}^{(j-1)}-\eta_{\nu}^{(j)}}
	{|L_j|}\psi_{\nu}^{(j)}\mathbf{1}_{L_j},
\end{align*}
which, together with Lemma \ref{Abel} and \eqref{4.12x},
further implies that, for any $l\in\nn$,
\begin{align}\label{sum-Pj}
	\sum_{j=0}^{l}P_j\mathbf{1}_{L_j}
	&=\sum_{\{\nu\in\zz_+^n:\ |\nu|\leq s\}}
	\sum_{j=0}^{l}\lf[\eta_{\nu}^{(j-1)}-\eta_{\nu}^{(j)}\r]
	\frac{\psi_{\nu}^{(j)}\mathbf{1}_{L_j}}{|L_j|}\\
	&=\sum_{\{\nu\in\zz_+^n:\ |\nu|\leq s\}}
	\sum_{j=0}^{l-1}\eta_{\nu}^{(j)}\lf[\frac{\psi_{\nu}^{(j+1)}
		\mathbf{1}_{L_{j+1}}}{|L_{j+1}|}
	-\frac{\psi_{\nu}^{(j)}\mathbf{1}_{L_j}}{|L_j|}\r]
	-\sum_{\{\nu\in\zz_+^n:\ |\nu|\leq s\}}
	\frac{\eta_{\nu}^{(l)}\psi_{\nu}^{(l)}\mathbf{1}_{L_l}}{|L_l|}\noz\\
	&=\sum_{\{\nu\in\zz_+^n:\ |\nu|\leq s\}}
	\sum_{j=0}^{l-1}\widetilde{\alpha}_\nu^{(j)}
	-\sum_{\{\nu\in\zz_+^n:\ |\nu|\leq s\}}
	\frac{\eta_{\nu}^{(l)}\psi_{\nu}^{(l)}\mathbf{1}_{L_l}}{|L_l|},\noz
\end{align}
where, for any $\nu\in\zz_+^n$ with $|\nu|\leq s$, and $j\in\{0,\ldots,l-1\}$,
$$
\widetilde{\alpha}_\nu^{(j)}
:=\eta_{\nu}^{(j)}\lf[\frac{\psi_{\nu}^{(j+1)}
\mathbf{1}_{L_{j+1}}}{|L_{j+1}|}
-\frac{\psi_{\nu}^{(j)}\mathbf{1}_{L_j}}{|L_j|}\r].
$$
Now, we claim that, for any $\nu\in\zz_+$
with $|\nu|\leq s$, and $j\in\zz_+$, $\widetilde{\alpha}_\nu^{(j)}$
is a $(p,q,s)_\alpha$-atom multiplying a positive constant
and, for any $f\in JN_{(p',q',s)_\alpha}^{\mathrm{con}}(\rn)$,
\begin{align}\label{eta-3'}
\lim_{l\to\fz}\lf|\int_{\rn}\frac{\eta_{\nu}^{(l)}\psi_{\nu}^{(l)}(x)
\mathbf{1}_{L_l}(x)}{|L_l|}f(x)\,dx\r|
=0.
\end{align}
Indeed, using \eqref{4.10x}, we conclude that,
for any $\nu$, $\gamma\in\zz_+^n$ with $|\nu|\leq s$
and $|\gamma|\leq s$, and $j\in\zz_+$,
\begin{align}\label{4.14x}
\int_{\rn}\widetilde{\alpha}_\nu^{(j)}(x)x^{\gamma}\,dx=0\quad\text{and}\quad
\supp\,(\widetilde{\alpha}_\nu^{(j)})\subset Q_{\mathbf{0}}(2^{j+1}r).
\end{align}
Moreover, from \cite[(7.2)]{L},
it follows that there exists a positive constant
$C_0$, independent of $\nu$, $j$, and $r$, such that,
for any $\nu\in\zz_+^n$ with $|\nu|\leq s$,
and any $x\in L_j$ with $j\in\zz_+$,
\begin{align}\label{L-7.2}
\lf|\psi_{\nu}^{(j)}(x)\r|\leq \frac{C_0}{(2^{j-1}r)^{|\nu|}}.
\end{align}
By this, \eqref{4.12x}, the H\"older inequality,
Definition \ref{Def-mole}(ii),
and $\frac{1}{\epsilon}
(\frac{1}{q}-\frac{1}{p}-\alpha)+\frac{1}{q'}+\frac{s}{n}<0$,
we conclude that, for any
$\nu\in \zz_+^n$ with $|\nu|\leq s$, and $j\in\zz_+$,
\begin{align}\label{eta-1}
\lf|\eta_{\nu}^{(j)}\r|
&\leq\int_{\rn\setminus Q_{\mathbf{0}}(2^jr)}\lf|M(x)x^{\nu}\r|\,dx
=\sum_{i=j+1}^{\fz}\int_{L_i}\lf|M(x)x^{\nu}\r|\,dx\\
&\leq \sum_{i=j+1}^{\fz}\lf\|M\mathbf{1}_{L_i}\r\|_{L^q(\rn)}
|L_i|^{\frac{1}{q'}}\lf(2^{i}\sqrt nr\r)^{|\nu|}\noz\\
&\ls\sum_{i=j+1}^{\fz}
2^{\frac{in}{\epsilon}(\frac{1}{q}-\frac{1}{p}-\alpha)}
|Q_{\mathbf{0}}(r)|^{\frac{1}{q}-\frac{1}{p}-\alpha}
\lf|Q_{\mathbf{0}}(2^ir)\r|^{\frac{1}{q'}+\frac{|\nu|}{n}}\noz\\
&\sim\sum_{i=j+1}^{\fz}2^{in[\frac{1}
{\epsilon}(\frac{1}{q}-\frac{1}{p}-\alpha)+\frac{1}{q'}+\frac{|\nu|}{n}]}
|Q_{\mathbf{0}}(r)|^{\frac{1}{p'}-\alpha+\frac{|\nu|}{n}}\noz\\
&\sim2^{jn[(\frac{1}{\epsilon}-1)(\frac{1}{q}-\frac{1}{p}-\alpha)]}
\lf|Q_{\mathbf{0}}(2^{j+1}r)\r|^{\frac{1}{p'}-\alpha+\frac{|\nu|}{n}}\noz.
\end{align}
From \eqref{L-7.2} and \eqref{eta-1}, it follows that,
for any $\nu\in\zz_+^n$ with $|\nu|\leq s$, and $j\in\zz_+$,
\begin{align*}
\lf\|\widetilde{\alpha}_\nu^{(j)}\r\|_{L^q(\rn)}
&\leq \lf\|\widetilde{\alpha}_\nu^{(j)}
\r\|_{L^{\fz}(\rn)}\lf|Q_{\mathbf{0}}(2^{j+1}r)\r|^{\frac{1}{q}}
\ls\frac{|\eta_{\nu}^{(j)}|
|Q_{\mathbf{0}}(2^{j+1}r)|^{\frac{1}{q}}}{(2^{j-1}r)^{|\nu|}|L_j|}\\
&\sim\frac{|\eta_{\nu}^{(j)}|}{(2^{j-1}r)^{|\nu|}}
\lf|Q_{\mathbf{0}}(2^{j+1}r)\r|^{-\frac{1}{q'}}
\sim\lf|\eta_{\nu}^{(j)}\r|
\lf|Q_{\mathbf{0}}(2^{j+1}r)\r|^{-\frac{1}{q'}-\frac{|\nu|}{n}}\\
&\ls2^{jn[(\frac{1}{\epsilon}-1)(\frac{1}{q}-\frac{1}{p}-\alpha)]}
\lf|Q_{\mathbf{0}}(2^{j+1}r)\r|^{\frac{1}{q}-\frac{1}{p}-\alpha}.
\end{align*}
This and \eqref{4.14x} imply that  there exists a positive constant
$\widetilde{C}$ such that, for any
$\nu\in\zz_+^n$ with $|\nu|\leq s$,
and $j\in\zz_+$, $\widetilde{A}_{\nu}^{(j)}
:=\frac{\widetilde{\alpha}_\nu^{(j)}}{\widetilde{\lambda}_{j}}$
is a $(p,q,s)_\alpha$-atom,
where
\begin{align}\label{4.13x}
	\widetilde{\lambda}_{j}:=\widetilde{C}
	(2^{jn})^{(\frac{1}{\epsilon}-1)(\frac{1}{q}-\frac{1}{p}-\alpha)}.
\end{align}
Thus, the first claim holds true.
Then we show \eqref{eta-3'}.
Indeed, by \eqref{L-7.2}, the H\"older inequality,
Lemma \ref{BMO-JN} with $u:=p'$ and $v:=q'$,
\eqref{eta-1}, \eqref{p}, Lemma \ref{N12-lemma},
$\alpha\in(\frac{1}{q}-\frac{1}{p},\fz)$,
$\epsilon\in(0,1)$, and $\frac{1}{\epsilon}(\frac{1}{q}
-\frac{1}{p}-\alpha)+\frac{1}{q'}+\frac{s}{n}<0$,
we conclude that, for any $l\in\nn$
and $f\in JN_{(p',q',s)_\alpha}^{\mathrm{con}}(\rn)$,
\begin{align*}\
	&\lf|\int_{\rn}\frac{\eta_{\nu}^{(l)}\psi_{\nu}^{(l)}(x)
		\mathbf{1}_{L_l}(x)}{|L_l|}f(x)\,dx\r|\\
	&\quad\lesssim\frac{|\eta_\nu^{(l)}|}{2^{(l-1)|\nu|}|L_l|}
	\int_{L_l}\lf|f(x)\r|\,dx
	\lesssim\frac{|\eta_\nu^{(l)}|}{2^{(l-1)|\nu|}}
	\fint_{Q_{\mathbf{0}}(2^lr)}\lf|f(x)\r|\,dx\\
	&\quad\lesssim \frac{|\eta_\nu^{(l)}|}{2^{(l-1)|\nu|}}
	\lf[\fint_{Q_{\mathbf{0}}(2^lr)}\lf|f(x)\r|^{q'}\,dx\r]^{\frac{1}{q'}}\\
	&\quad\lesssim\frac{|\eta_\nu^{(l)}|}{2^{(l-1)|\nu|}}
	\lf\{\lf[\fint_{Q_{\mathbf{0}}(2^lr)}
	\lf|f(x)-P_{Q_{\mathbf{0}}(r)}^{(s)}(f)(x)\r|^{q'}\,dx\r]^{\frac{1}{q'}}
	+\lf\|P_{Q_{\mathbf{0}}(r)}^{(s)}(f)
	\r\|_{L^{\fz}(Q_{\mathbf{0}}(2^lr))}\r\}\\
	&\quad\ls2^{nl[\frac{1}{\epsilon}
		(\frac{1}{q}-\frac{1}{p}-\alpha)+\frac{1}{q'}]}
	l\lf[2^{ls}+2^{nl(\alpha-\frac{1}{p'})}\r]
	\|f\|_{JN_{(p',q',s)_\alpha}^{\mathrm{con}}(\rn)}\\
	&\qquad+2^{nl[\frac{1}{\epsilon}
		(\frac{1}{q}-\frac{1}{p}-\alpha)+\frac{1}{q'}+\frac{s}{n}]}
	\lf\|P_{Q_{\mathbf{0}}(r)}^{(s)}(f)
	\r\|_{L^{\fz}(Q_{\mathbf{0}}(r))}\\
	&\quad \ls 2^{nl[\frac{1}{\epsilon}
		(\frac{1}{q}-\frac{1}{p}-\alpha)+\frac{1}{q'}+\frac{s}{n}]}l
	\|f\|_{JN_{(p',q',s)_\alpha}^{\mathrm{con}}(\rn)}
	+2^{nl[(\frac{1}{\epsilon}-1)
		(\frac{1}{q}-\frac{1}{p}-\alpha)]}l
	\|f\|_{JN_{(p',q',s)_\alpha}^{\mathrm{con}}(\rn)}\\
	&\qquad+2^{nl[\frac{1}{\epsilon}
		(\frac{1}{q}-\frac{1}{p}-\alpha)+\frac{1}{q'}+\frac{s}{n}]}
	\fint_{Q_{\mathbf{0}}(r)}|f(x)|\,dx\\
	&\quad\to 0
\end{align*}
as $l\to\fz$, where the implicit positive constants depend on $r$.
This further implies that \eqref{eta-3'} holds true.

Finally, we show that $M\in HK_{(p,q,s)_\alpha}^{\mathrm{con}}(\rn)$.
Indeed, using \eqref{M-l'} and \eqref{sum-Pj},
we conclude that, for any $l\in\nn$ and $x\in\rn$,
\begin{align}\label{3.19x}
	M\mathbf{1}_{Q_{\mathbf{0}}(2^lr)}(x)
	=\sum_{j=0}^{l}\lambda_j A_j(x)+
	\sum_{\{\nu\in\zz_+^n:\ |\nu|\leq s\}}
	\sum_{j=0}^{l-1}
	\widetilde{\lambda}_{j}\widetilde{A}_\nu^{(j)}(x)
	-\sum_{\{\nu\in\zz_+^n:\ |\nu|\leq s\}}
	\frac{\eta_{\nu}^{(l)}\psi_{\nu}^{(l)}(x)\mathbf{1}_{L_l}(x)}{|L_l|}.
\end{align}
By this, Lemma \ref{p3.1}, Proposition \ref{M-JN1},
the Lebesgue dominated convergence theorem,
\eqref{3.19x}, and \eqref{eta-3'}, we conclude that,
for any $l\in\nn$ and $f\in JN_{(p',q',s)_\alpha}^{\mathrm{con}}(\rn)$,
\begin{align*}
&\lf|\lf\langle M-\sum_{j=0}^{l}\lambda_j A_j
-\sum_{j=0}^{l-1}\sum_{\{\nu\in\zz_+^n:\ |\nu|\leq s\}}
\widetilde{\lambda}_{j}\widetilde{A}_\nu^{(j)},f\r\rangle\r|\\
&\quad=\lf|\int_{\rn}\lf[M(x)-\sum_{j=1}^{l}\lambda_j A_j(x)
-\sum_{j=0}^{l-1}\sum_{\{\nu\in\zz_+^n:\ |\nu|\leq s\}}
\widetilde{\lambda}_{j}\widetilde{A}_\nu^{(j)}(x)\r]f(x)\,dx\r|\\
&\quad\leq\lf|\int_{\rn\setminus Q_{\mathbf{0}}(2^lr)}M(x)f(x)\,dx\r|
+\sum_{\{\nu\in\zz_+^n:\ |\nu|\leq s\}}
\lf|\int_{\rn}\frac{\eta_{\nu}^{(l)}\psi_{\nu}^{(l)}(x)
	\mathbf{1}_{L_l}(x)}{|L_l|}f(x)\,dx\r|\\
&\quad\to0
\end{align*}
as $l\to\fz$, and hence
\begin{align}\label{M-Ml-JN*}
M=\sum_{j\in\zz_+}\lambda_j A_j
+\sum_{j\in\zz_+}\sum_{\{\nu\in\zz_+^n:\ |\nu|\leq s\}}
\widetilde{\lambda}_{j}\widetilde{A}_\nu^{(j)}
\end{align}
in $(JN_{(p',q',s)_\alpha}^{\mathrm{con}}(\rn))^*$.
Moreover, since,
for any $\nu\in\zz_+^n$ with $|\nu|\leq s$, and $j\in\zz_+$,
$A_j$ and $\widetilde{A}_\nu^{(j)}$ are $(p,q,s)_\alpha$-atoms,
we deduce that $\lambda_jA_j$,
$\widetilde{\lambda}_{j}\widetilde{A}_\nu^{(j)}
\in \widetilde{HK}_{(p,q,s)_\alpha}^{\mathrm{con}}(\rn)$,
which, combined with \eqref{lam-j}, \eqref{4.13x},
$(\frac{1}{\epsilon}-1)(\frac{1}{q}-\frac{1}{p}-\alpha)<0$,
and \eqref{M-Ml-JN*}, further implies that
\begin{align*}
\|M\|_{HK_{(p,q,s)_\alpha}^{\mathrm{con}}(\rn)}
&\leq\sum_{j\in\zz_+}\lf\|\lambda_j A_j
\r\|_{\widetilde{HK}_{(p,q,s)_\alpha}^{\mathrm{con}}(\rn)}
+\sum_{j\in\zz_+}\sum_{\{\nu\in\zz_+^n:\ |\nu|\leq s\}}
\lf\|\widetilde{\lambda}_{j}\widetilde{A}_\nu^{(j)}
\r\|_{\widetilde{HK}_{(p,q,s)_\alpha}^{\mathrm{con}}(\rn)}\\
&\leq\sum_{j\in\zz_+}|\lambda_j|+\sum_{\{\nu\in\zz_+^n:\ |\nu|\leq s\}}
\sum_{j\in\zz_+}\lf|\widetilde{\lambda}_{j}\r|\\
&\ls\sum_{j\in\zz_+}2^{jn[(\frac{1}{\epsilon}-1)(\frac{1}{q}-\frac{1}{p}-\alpha)]}
+\sum_{\{\nu\in\zz_+^n:\ |\nu|\leq s\}}\sum_{j\in\zz_+}
2^{jn[(\frac{1}{\epsilon}-1)(\frac{1}{q}-\frac{1}{p}-\alpha)]}
\sim 1,
\end{align*}
where the implicit constants are independent of $M$,
and hence $M\in HK_{(p,q,s)_\alpha}^{\mathrm{con}}(\rn)$.
This finishes the proof of Proposition \ref{HK-mole}.
\end{proof}

Using Proposition \ref{HK-mole}, we obtain the following useful
criterion for the boundedness of linear operators on
$HK_{(p,q,s)_\alpha}^{\mathrm{con}}(\rn)$.
In what follows, we equip $HK_{(p,q,s)_\alpha}^{\mathrm{con}}(\rn)$
and $HK_{(p,q,s)_\alpha}^{\mathrm{con-fin}}(\rn)$
with the same norm $\|\cdot\|_{HK_{(p,q,s)_\alpha}^{\mathrm{con}}(\rn)}$
unless specifically indicated.

\begin{theorem}\label{Bounded-HK-A}
Let $p_1$, $p_2\in (1,\fz)$, $q_1$, $q_2\in (1,\fz]$,
$\frac{1}{p_i}+\frac{1}{p_i'}=1=\frac{1}{q_i}+\frac{1}{q_i'}$
for any $i\in\{1,2\}$, $s_1$, $s_2\in\zz_+$, $\alpha_1\in\rr$,
and $\alpha_2\in (\frac{1}{q_2}-\frac{1}{p_2},\fz)$.
Let $A$ be a linear operator defined on
$HK_{(p_1,q_1,s_1)_{\alpha_1}}^{\mathrm{con-fin}}(\rn)$
and $\widetilde{A}$ a linear operator bounded
from $JN_{(p_2',q_2',s_2)_{\alpha_2}}^{\mathrm{con}}(\rn)$
to $JN_{(p_1',q_1',s_1)_{\alpha_1}}^{\mathrm{con}}(\rn)$.
If the following statements hold true:
\begin{enumerate}
\item[\rm (i)] there exists a positive constant $C_0$ such that,
for any $(p_1,q_1,s_1)_{\alpha_1}$-atom $a$
and some given $\epsilon\in(0,1)$ such that
$\frac{1}{\epsilon}
(\frac{1}{q_2}-\frac{1}{p_2}-\alpha_2)+\frac{1}{q_2'}+\frac{s_2}{n}<0$,
$\frac{A(a)}{C_0}$ is a $(p_2,q_2,
s_2,\alpha_2,\epsilon)$-molecule;
\item[\rm (ii)] for any $(p_1,q_1,s_1)_{\alpha_1}$-atom $a$ and any
$f\in JN_{(p_2',q_2',s_2)_{\alpha_2}}^{\mathrm{con}}(\rn)$,
$\langle A(a),f\rangle=\langle a, \widetilde{A}(f)\rangle$,
\end{enumerate}
then the linear operator $A$ has a unique continuous linear extension,
still denoted by $A$, from
$HK_{(p_1,q_1,s_1)_{\alpha_1}}^{\mathrm{con}}(\rn)$
to $HK_{(p_2,q_2,s_2)_{\alpha_2}}^{\mathrm{con}}(\rn)$,
namely, there exists a positive constant $C$ such that,
for any $g\in HK_{(p_1,q_1,s_1)_{\alpha_1}}^{\mathrm{con}}(\rn)$,
$$\|A(g)\|_{HK_{(p_2,q_2,s_2)_{\alpha_2}}^{\mathrm{con}}(\rn)}
\leq C\|g\|_{HK_{(p_1,q_1,s_1)_{\alpha_1}}^{\mathrm{con}}(\rn)}$$
and, moreover,
for any $f\in JN_{(p_2',q_2',s_2)_{\alpha_2}}^{\mathrm{con}}(\rn)$,
\begin{align}\label{3.20x}
\langle A(g),f\rangle
=\lf\langle g,\widetilde{A}(f)\r\rangle.
\end{align}
\end{theorem}

\begin{proof}
Let $p_1$, $p_2$, $q_1$, $q_2$,
$s_1$, $s_2$, $\alpha_1$, $\alpha_2$,
$A$, and $\widetilde{A}$ be as in the present theorem.
We first prove that $A$ has a unique continuous linear extension on
$HK_{(p_1,q_1,s_1)_{\alpha_1}}^{\mathrm{con}}(\rn)$.
To this end, we claim that $A$ is bounded from
$HK_{(p_1,q_1,s_1)_{\alpha_1}}^{\mathrm{con-fin}}(\rn)$
to $HK_{(p_2,q_2,s_2)_{\alpha_2}}^{\mathrm{con}}(\rn)$.
Indeed, let $g\in HK_{(p_1,q_1,s_1)_{\alpha_1}}^{\mathrm{con-fin}}(\rn)$.
Then, from (i) and (ii) of the present theorem,
Lemma \ref{t3.9}, and Propositions \ref{M-JN1} and \ref{HK-mole},
it follows that, for any
$f\in JN_{(p_2',q_2',s_2)_{\alpha_2}}^{\mathrm{con}}(\rn)$,
\begin{align*}
\lf|\lf\langle A(g), f\r\rangle\r|
&=\lf|\lf\langle g, \widetilde{A}(f)\r\rangle\r|
\lesssim\|g\|_{HK_{(p_1,q_1,s_1)_{\alpha_1}}^{\mathrm{con}}(\rn)}
\lf\|\widetilde{A}(f)\r\|_{JN_{(p_1',q_1',s_1)_{\alpha_1}}^{\mathrm{con}}(\rn)}\\
&\lesssim\|g\|_{HK_{(p_1,q_1,s_1)_{\alpha_1}}^{\mathrm{con}}(\rn)}
\|f\|_{JN_{(p_2',q_2',s_2)_{\alpha_2}}^{\mathrm{con}}(\rn)},
\end{align*}
which, together with Proposition \ref{Coro-JN}, further implies that
\begin{align}\label{key-A}
\|A(g)\|_{HK_{(p_2,q_2,s_2)_{\alpha_2}}^{\mathrm{con}}(\rn)}
\sim\|A(g)\|_{(JN_{(p_2',q_2',s_2)_{\alpha_2}}^{\mathrm{con}}(\rn))^*}
\lesssim\|g\|_{HK_{(p_1,q_1,s_1)_{\alpha_1}}^{\mathrm{con}}(\rn)}
\end{align}
and hence the above claim holds true.

Moreover, by \eqref{key-A},
Lemma \ref{HK-normed}, and Proposition \ref{HK-banach},
we conclude that the linear operator
$A$ is bounded from the dense subspace
$HK_{(p_1,q_1,s_1)_{\alpha_1}}^{\mathrm{con-fin}}(\rn)$
of $HK_{(p_1,q_1,s_1)_{\alpha_1}}^{\mathrm{con}}(\rn)$
to $HK_{(p_2,q_2,s_2)_{\alpha_2}}^{\mathrm{con}}(\rn)$.
From this and \cite[p.\,22, Proposition 5.4]{stein2011}, we deduce that
$A$ can be extended to a unique continuous linear operator,
still denoted by $A$, from
$HK_{(p_1,q_1,s_1)_{\alpha_1}}^{\mathrm{con}}(\rn)$
to $HK_{(p_2,q_2,s_2)_{\alpha_2}}^{\mathrm{con}}(\rn)$
in the sense that, for any
$g\in HK_{(p_1,q_1,s_1)_{\alpha_1}}^{\mathrm{con}}(\rn)$,
\begin{align}\label{4.21x}
\lim_{m\to\fz}\lf\|A(g)-A(g_m)
\r\|_{HK_{(p_2,q_2,s_2)_{\alpha_2}}^{\mathrm{con}}(\rn)}=0,
\end{align}
where $\{g_m\}_{m\in\nn}$ is a sequence of
$HK_{(p_1,q_1,s_1)_{\alpha_1}}^{\mathrm{con-fin}}(\rn)$
such that
\begin{align}\label{3.24x}
\lim_{m\to\fz}\|g-g_m\|_{HK_{(p_1,q_1,s_1)_{\alpha_1}}^{\mathrm{con}}(\rn)}=0,
\end{align}
$A(g)$ is independent of the choice of $\{g_m\}_{m\in\nn}$,
and, for any $g\in HK_{(p_1,q_1,s_1)_{\alpha_1}}^{\mathrm{con}}(\rn)$,
$$\|A(g)\|_{HK_{(p_2,q_2,s_2)_{\alpha_2}}^{\mathrm{con}}(\rn)}
\ls\|g\|_{HK_{(p_1,q_1,s_1)_{\alpha_1}}^{\mathrm{con}}(\rn)}.$$

Next, we show \eqref{3.20x}.
Indeed, for any $f\in JN_{(p_2',q_2', s_2)_{\alpha_2}}^{\mathrm{con}}(\rn)$
and $g\in HK_{(p_1,q_1,s_1)_{\alpha_1}}^{\mathrm{con}}(\rn)$,
let $g_m:=\sum_{i=1}^{M_m}\lambda_{m,i}a_{m,i}$ be as in \eqref{4.21x}
with $\{\lambda_{m,i}\}_{i=1}^{M_m}\subset\cc$
and $\{a_{m,i}\}_{i=1}^{M_m}$ being $(p_1,q_1,s_1)_{\alpha_1}$-atoms.
Then, by \eqref{4.21x}, Lemma \ref{t3.9}, (ii)
of the present theorem, and \eqref{3.24x}, we find that
\begin{align*}
\langle A(g),f\rangle
&=\lim_{m\to\fz}\lf\langle A(g_m),f\r\rangle
=\lim_{m\to\fz}\sum_{i=1}^{M_m}\lambda_{m,i}
\lf\langle A(a_{m,i}),f\r\rangle\\
&=\lim_{m\to\fz}\sum_{i=1}^{M_m}\lambda_{m,i}
\lf\langle a_{m,i},\widetilde{A}(f)\r\rangle
=\lf\langle g,\widetilde{A}(f)\r\rangle,
\end{align*}
which shows the above claim.
This finishes the proof of Theorem \ref{Bounded-HK-A}.
\end{proof}

\begin{remark}\label{R-T-JN}
Although we are not able to establish a molecular characterization
of $HK_{(p,q,s)_{\alpha}}^{\mathrm{con}}(\rn)$ because
the classical method seems no longer feasible
in the present setting, the proof of
Theorem \ref{Bounded-HK-A} contains some thoughts
of the molecular characterization of
$HK_{(p,q,s)_{\alpha}}^{\mathrm{con}}(\rn)$.
\end{remark}

\subsection{Boundedness of Calder\'on--Zygmund Operators
on $HK_{(p,q,s)_\alpha}^{\mathrm{con}}(\rn)$\label{S-T-HK}}

In this subsection, we show that the $s$-order
Calder\'on--Zygmund singular integral operator $T$,
having the vanishing moments up
to order $s$, can be extended to a unique continuous
linear operator on $HK_{(p,q,s)_\alpha}^{\mathrm{con}}(\rn)$.
To this end, we establish two technical lemmas.
We first show that, for any $(p,q,s)_\alpha$-atom
$a$, $T(a)$ is a $(p,q,s,\alpha,\epsilon)$-molecule
in the sense of multiplying a positive constant independent of $a$.

\begin{lemma}\label{Ta-N}
Let $p\in(1,\fz)$, $q\in (1,\fz)$, $\frac{1}{q}+\frac{1}{q'}=1$, $s\in\zz_+$,
$\alpha\in(\frac{1}{q}-\frac{1}{p},1-\frac{1}{p}+\frac{s+\dz}{n})$
with $\dz\in(0,1]$ as in Definition \ref{def-s-k},
and $\epsilon\in (0,1)$ be such that
$-\frac{1}{q'}-\frac{s+\dz}{n}\leq \frac{1}{\epsilon}(\frac{1}{q}
-\frac{1}{p}-\alpha)$.
Let $K$ be an $s$-order standard kernel and $T$ the $s$-order
Calder\'on--Zygmund singular integral operator with kernel $K$.
Then the following two statements are equivalent:
\begin{enumerate}
\item[\rm (i)]
there exists a positive constant $C$ such that,
for any $(p,q,s)_\alpha$-atom $a$,
$\frac{T(a)}{C}$ is a $(p,q,s,\alpha,\epsilon)$-molecule;
\item[\rm (ii)]
for any $\gamma\in\zz_+^n$ with $|\gamma|\leq s$, $T^*(x^{\gamma})=0$.
\end{enumerate}
\end{lemma}

\begin{proof}
Let $p$, $q$, $q'$, $\alpha$, $s$, $\alpha$, $\dz$, $\epsilon$,
$T$, and $K$ be as in the present lemma.
We first show (i) $\Rightarrow$ (ii).
Indeed, for any given $A\in L_s^q(\rn)$ with
$\|A\|_{L^q(\rn)}>0$ supported in a
ball $B(x_0,r_0)$ with $x_0\in\rn$ and $r_0\in(0,\fz)$, it is easy to see that
$\widetilde{A}:=\frac{|Q_{x_0}(2r_0)|^{\frac{1}{q}-\frac{1}{p}-\alpha}A}
{\|A\|_{L^q(\rn)}}$
is a $(p,q,s)_\alpha$-atom supported in $Q_{x_0}(2r_0)$.
By this, (i) of the present lemma, and \eqref{0-0'}, we conclude that,
for any $\gamma\in\zz_+^n$ with $|\gamma|\leq s$,
$$
\int_{B(x_0,r_0)}\widetilde{A}(x)\widetilde{T}(y^{\gamma})(x)\,dx
=\int_{\rn}T(\widetilde{A})(x)x^{\gamma}\,dx
=0.
$$
Using this and a proof similar to that of Proposition \ref{Assume},
we find that, for any $\gamma\in\zz_+^n$ with $|\gamma|\leq s$,
$\widetilde{T}(y^{\gamma})\in\mathcal{P}_s(\rn)$
after changing values on a set of measure zero,
which, combined with Proposition \ref{Assume}, further implies that,
for any $\gamma\in\zz_+^n$ with $|\gamma|\leq s$, $T^*(x^{\gamma})=0$.
This finishes the proof that (i) $\Rightarrow$ (ii).

Next, we show (ii) $\Rightarrow$ (i).
Let $a$ be a $(p,q,s)_\alpha$-atom supported in the cube
$Q_{z}(r)$ with $z\in \rn$ and $r\in (0,\fz)$.
We show that there exists a positive constant $C$,
independent of $a$, such that
$\frac{T(a)}{C}$ is a $(p,q,s,\alpha,\epsilon)$-molecule.
To this end, let $R_0:=2\sqrt nr$.
We first prove that there exists a positive
constant $C_1$, independent of $a$, such that
\begin{align}\label{Ta-i}
\lf\|T (a)\mathbf{1}_{Q_z(R_0)}\r\|_{L^q(\rn)}
\leq C_1\lf|Q_z(R_0)\r|^{\frac{1}{q}-\frac{1}{p}-\alpha}.
\end{align}
Indeed, using Lemma \ref{Duo01}, Definition \ref{d3.2}(ii),
and $r\sim R_0$, we find that
\begin{align*}
\lf\|T (a)\mathbf{1}_{Q_z(R_0)}\r\|_{L^q(\rn)}
\lesssim \|a\|_{L^q(Q_{z}(r))}
\lesssim |Q_z(r)|^{\frac{1}{q}-\frac{1}{p}-\alpha}
\sim \lf|Q_z(R_0)\r|^{\frac{1}{q}-\frac{1}{p}-\alpha}.
\end{align*}
This implies that \eqref{Ta-i} holds true.

Now, we show that there exists a positive
constant $C_2$, independent of $a$,
such that, for any $j\in\nn$,
\begin{align}\label{Ta-ii}
\lf\|T (a)\mathbf{1}_{Q_z(2^{j}R_0)
\setminus Q_z(2^{j-1}R_0)}\r\|_{L^q(\rn)}
\leq C_2\lf(2^{jn}\r)^{\frac{1}{\epsilon}
(\frac{1}{q}-\frac{1}{p}-\alpha)}
|Q_z(R_0)|^{\frac{1}{q}-\frac{1}{p}-\alpha}.
\end{align}
Indeed, for any $x\notin Q_z(R_0)$
and $y\in Q_z(r)$, we have
$|z-x|\geq 2|z-y|$,
which, combined with $x\notin \supp\,(a)$,
Definition \ref{d3.2}(iii),
the Taylor remainder theorem, \eqref{regular2-s},
the H\"older inequality, and Definition \ref{d3.2}(ii),
further implies that, for any $y\in Q_z(r)$,
there exists a $\widetilde{y}\in Q_z(r)$
such that, for any $x\notin Q_z(R_0)$,
\begin{align*}
|T(a)(x)|&=\lf|\int_{Q_z(r)}K(x,y)a(y)\,dy\r|\\
&\leq\int_{Q_z(r)}\lf|K(x,y)-\sum_{\{\gamma\in\zz_+^n:\ |\gamma|\leq s\}}
\frac{\partial_{(2)}^{\gamma}K(x,z)}{\gamma!}(y-z)^{\gamma}\r||a(y)|\,dy\\
&=\int_{Q_z(r)}\lf|\sum_{\{\gamma\in\zz_+^n:\ |\gamma|=s\}}
\frac{\partial_{(2)}^{\gamma}K(x,\widetilde{y})-\partial_{(2)}^{\gamma}K(x,z)}
{\gamma!}(y-z)^{\gamma}\r||a(y)|\,dy\\
&\lesssim \frac{r^{s+\delta}}{|z-x|^{n+s+\delta}}\int_{Q_z(r)}|a(y)|\,dy
\lesssim \frac{r^{\frac{n}{q'}+s+\delta}}{|z-x|^{n+s+\delta}}
\|a\|_{L^{q}(Q_z(r))}\\
&\lesssim \lf(\frac{r}{|z-x|}\r)^{\frac{n}{q'}+s+\delta}
|z-x|^{-\frac{n}{q}}|Q_z(r)|^{\frac{1}{q}-\frac{1}{p}-\alpha}.
\end{align*}
From this, $r\sim R_0$, and
$-\frac{1}{q'}-\frac{s+\dz}{n}\leq \frac{1}{\epsilon}(\frac{1}{q}
-\frac{1}{p}-\alpha)$,
we deduce that, for any $j\in\nn$ and
$x\in Q_z(2^jR_0)\setminus Q_z(2^{j-1}R_0)$,
\begin{align*}
\lf|T(a)(x)\r|
&\lesssim\lf(\frac{r}{2^jR_0}\r)^{\frac{n}{q'}+s+\delta}
(2^jR_0)^{-\frac{n}{q}}|Q_z(R_0)|^{\frac{1}{q}-\frac{1}{p}-\alpha}\\
&\lesssim (2^{jn})^{\frac{1}{\epsilon}(\frac{1}{q}-\frac{1}{p}-\alpha)}
\lf|Q_z(2^jR_0)\setminus Q_z(2^{j-1}R_0)\r|^{-\frac{1}{q}}
|Q_z(R_0)|^{\frac{1}{q}-\frac{1}{p}-\alpha}
\end{align*}
and hence
\begin{align*}
\lf\|T(a)\mathbf{1}_{Q_z(2^jR_0)\setminus Q_z(2^{j-1}R_0)}\r\|_{L^{q}(\rn)}
\lesssim (2^{jn})^{\frac{1}{\epsilon}(\frac{1}{q}-\frac{1}{p}-\alpha)}
|Q_z(R_0)|^{\frac{1}{q}-\frac{1}{p}-\alpha},
\end{align*}
which implies that \eqref{Ta-ii} holds true.

Moreover, by (ii) of the present lemma, we know that
$T(a)$ satisfies Definition \ref{Def-mole}(iii).
From this, \eqref{Ta-i}, and \eqref{Ta-ii},
we deduce that $\frac{T(a)}{C}$ satisfies
(i), (ii), and (iii) of Definition \ref{Def-mole},
where $C:=\max\{C_1,C_2\}$.
This finishes the proof of Lemma \ref{Ta-N}.
\end{proof}

\begin{lemma}\label{duity-T-Tw}
Let $p\in(1,\infty)$, $q\in(1,\fz)$,
$\frac{1}{p}+\frac{1}{p'}=1=\frac{1}{q}+\frac{1}{q'}$, $s\in\zz_+$,
and $\alpha\in(\frac{1}{q}-\frac{1}{p},1-\frac{1}{p}+\frac{s+\delta}{n})$
with $\delta\in(0,1]$ as in Definition \ref{def-s-k}.
Let $K$ be an $s$-order standard kernel and $T$ an $s$-order
Calder\'on--Zygmund singular integral operator
having the vanishing moments up to order $s$ with
kernel $K$.
Then, for any $(p,q,s)_\alpha$-atom $a$ and any
$f\in JN_{(p',q',s)_\alpha}^{\mathrm{con}}(\rn)$,
$$
\langle T(a),f\rangle=\lf\langle a,\widetilde{T}(f)\r\rangle,
$$
where $\widetilde{T}$ is as in Definition \ref{def-JN-CZO}
with kernel $\widetilde{K}$ as in \eqref{Kw-K}.
\end{lemma}

\begin{proof}
Let $p$, $q$, $s$, $\alpha$, $\dz$,
$T$, $\widetilde{T}$, $K$, and
$\widetilde{K}$ be as in the present lemma.
Also, let $a$ be a $(p,q,s)_\alpha$-atom
supported in the cube $Q_{x_0}(2r_0/\sqrt n)$
with $x_0\in\rn$ and $r_0\in(0,\fz)$,
and $B_0:=B(x_0,r_0)$. Then $Q_{x_0}(2r_0/\sqrt n)\subset B_0$.
From this, Lemma \ref{Ta-N},
Proposition \ref{M-JN1}, Remark \ref{rem-2.15},
and Definition \ref{d3.2}(iii),
we deduce that, for any $f\in JN_{(p',q',s)_\alpha}^{\mathrm{con}}(\rn)$,
\begin{align}\label{2x}
\langle T(a),f\rangle
&=\int_{\rn}T(a)(x)f(x)\,dx\\
&=\int_{\rn}\lim_{\eta\to0^+}\int_{B_0\setminus B(x,\eta)}
K(x,y)a(y)\,dy f(x)\,dx\noz\\
&=\int_{2B_0}\lim_{\eta\to0^+}\int_{B_0\setminus B(x,\eta)}
K(x,y)a(y)\,dy f(x)\,dx\noz\\
&\quad+\int_{\rn\setminus 2B_0}\int_{B_0}
\lf[K(x,y)-\sum_{\{\gamma\in\zz_+^n:\ |\gamma|\leq s\}}
\frac{\partial_{(2)}^{\gamma}K(x,x_0)}
{\gamma!}(y-x_0)^\gamma\r]a(y)\,dy f(x)\,dx\noz\\
&=:\mathrm{\Lambda}_1+\mathrm{\Lambda}_2,\noz
\end{align}
where $T_{\eta}$ is as in \eqref{def-T-eta}.

We first consider $\mathrm{\Lambda}_1$.
To this end, for any $\eta\in(0,\fz)$,
let $K_{\eta}$ be as in \eqref{p20x}.
Then, by $a\in L^q(\rn)$, Remark \ref{rem-2.15},
the fact that $f\mathbf{1}_{2B_0}\in L^{q'}(\rn)$,
and the Fubini theorem, we find that
\begin{align}\label{2xx}
\mathrm{\Lambda}_1
&=\lim_{\eta\to0^+}\int_{2B_0}\int_{B_0}K_{\eta}(x,y)a(y)\,dyf(x)\,dx\\
&=\lim_{\eta\to0^+}\int_{B_0}a(y)\int_{2B_0}K_{\eta}(x,y)f(x)\,dx\,dy\noz\\
&=\int_{B_0}a(y)\lim_{\eta\to0^+}\int_{2B_0}K_{\eta}(x,y)f(x)\,dx\,dy.\noz
\end{align}
This is a desired conclusion of $\mathrm{\Lambda}_1$.

Next, we consider $\mathrm{\Lambda}_2$. We first claim that
\begin{align}\label{3}
\int_{\rn\setminus 2B_0}\int_{B_0}
\lf|\lf[K(x,y)-\sum_{\{\gamma\in\zz_+^n:\ |\gamma|\leq s\}}
\frac{\partial_{(2)}^{\gamma}K(x,x_0)}{\gamma!}
(y-x_0)^\gamma\r]a(y)f(x)\r|\,dy\,dx<\fz.
\end{align}
Indeed, from \eqref{3'}, we deduce that
\begin{align}\label{4}
\int_{\rn\setminus 2B_0}\int_{B_0}
\lf|\lf[K(x,y)-\sum_{\{\gamma\in\zz_+^n:\ |\gamma|\leq s\}}
\frac{\partial_{(2)}^{\gamma}K(x,x_0)}{\gamma!}(y-x_0)^\gamma\r]a(y)\r|
\lf|P_{2B_0}^{(s)}(f)(x)\r|\,dy\,dx<\fz.
\end{align}
Moreover, using the Tonelli theorem,
the Taylor remainder theorem, \eqref{regular2-s},
Lemma \ref{I-JN} with $\beta:=s+\dz\in(s,\fz)$,
and $\alpha\in(\frac{1}{q}-\frac{1}{p},1-\frac{1}{p}+\frac{s+\delta}{n})$,
we find that, for any $y\in B_0$,
there exists a $\widetilde{y}\in B_0$ such that
\begin{align*}
&\int_{\rn\setminus 2B_0}\int_{B_0}
\lf|\lf[K(x,y)-\sum_{\{\gamma\in\zz_+^n:\ |\gamma|\leq s\}}
\frac{\partial_{(2)}^{\gamma}K(x,x_0)}{\gamma!}(y-x_0)^\gamma\r]a(y)
\lf[f(x)-P_{2B_0}^{(s)}(f)(x)\r]\r|\,dy\,dx\\
&\quad=\int_{B_0}\int_{\rn\setminus 2B_0}
\lf|\lf[K(x,y)-\sum_{\{\gamma\in\zz_+^n:\ |\gamma|\leq s\}}
\frac{\partial_{(2)}^{\gamma}K(x,x_0)}{\gamma!}(y-x_0)^\gamma
\r]a(y)\lf[f(x)-P_{2B_0}^{(s)}(f)(x)\r]\r|\,dx\,dy\\
&\quad=\int_{B_0}\int_{\rn\setminus 2B_0}
\lf|\sum_{\{\gamma\in\zz_+^n:\ |\gamma|=s\}}
\frac{\partial_{(2)}^{\gamma}K(x,\widetilde{y})
-\partial_{(2)}^{\gamma}K(x,x_0)}{\gamma!}
(y-x_0)^\gamma\r|
|a(y)|\lf|f(x)-P_{2B_0}^{(s)}(f)(x)\r|\,dx\,dy\\
&\quad\ls\int_{B_0}\int_{\rn\setminus 2B_0}
\frac{|\widetilde{y}-x_0|^{\dz}|y-x_0|^s}{|x-x_0|^{n+s+\dz}}
|a(y)|\lf|f(x)-P_{2B_0}^{(s)}(f)(x)\r|\,dx\,dy\\
&\quad\lesssim r_0^{s+\dz}
\|a\|_{L^1(B_0)}\int_{\rn\setminus 2B_0}
\frac{|f(x)-P_{2B_0}^{(s)}(f)(x)|}{|x-x_0|^{n+s+\dz}}\,dx\\
&\quad\ls r_0^{(-\frac{1}{p'}+\alpha)n}\|a\|_{L^1(B_0)}
\|f\|_{JN_{(p',q',s)_\alpha}^{\mathrm{con}}(\rn)}
<\fz,
\end{align*}
which, together with \eqref{4}, further implies that
\begin{align*}
&\int_{\rn\setminus 2B_0}\int_{B_0}
\lf|\lf[K(x,y)-\sum_{\{\gamma\in\zz_+^n:\ |\gamma|\leq s\}}
\frac{\partial_{(2)}^{\gamma}K(x,x_0)}{\gamma!}
(y-x_0)^\gamma\r]a(y)f(x)\r|\,dy\,dx\\
&\quad\leq\int_{\rn\setminus 2B_0}\int_{B_0}
\lf|\lf[K(x,y)-\sum_{\{\gamma\in\zz_+^n:\ |\gamma|\leq s\}}
\frac{\partial_{(2)}^{\gamma}K(x,x_0)}{\gamma!}
(y-x_0)^\gamma\r]a(y)P_{B_0}^{(s)}(f)(x)\r|\,dy\,dx\\
&\quad\quad+\int_{\rn\setminus 2B_0}\int_{B_0}
\lf|\lf[K(x,y)-\sum_{\{\gamma\in\zz_+^n:\ |\gamma|\leq s\}}
\frac{\partial_{(2)}^{\gamma}K(x,x_0)}{\gamma!}(y-x_0)^\gamma
\r]a(y)\lf[f(x)-P_{2B_0}^{(s)}(f)(x)\r]\r|\,dy\,dx\\
&\quad<\fz
\end{align*}
and hence \eqref{3} holds true. Then, by \eqref{3} and
the Fubini theorem, we obtain
\begin{align}\label{4x}
\mathrm{\Lambda}_2=\int_{B_0}\int_{\rn\setminus 2B_0}
\lf[K(x,y)-\sum_{\{\gamma\in\zz_+^n:\ |\gamma|\leq s\}}
\frac{\partial_{(2)}^{\gamma}K(x,x_0)}{\gamma!}
(y-x_0)^\gamma\r]f(x)\,dx\,a(y)\,dy.
\end{align}
This is a desired conclusion of $\mathrm{\Lambda}_2$.

Altogether, from \eqref{2x}, \eqref{2xx},
and \eqref{4x}, it follows that,
for any $f\in JN_{(p,q,s)_\alpha}^{\mathrm{con}}(\rn)$,
\begin{align*}
&\langle T(a),f\rangle\\
&\quad=\int_{B_0}a(y)\lim_{\eta\to0^+}
\int_{2B_0\setminus B(y,\eta)}K(x,y)f(x)\,dx\,dy\\
&\qquad+\int_{B_0}\int_{\rn\setminus 2B_0}
\lf[K(x,y)-\sum_{\{\gamma\in\zz_+^n:\ |\gamma|\leq s\}}
\frac{\partial_{(2)}^{\gamma}K(x,x_0)}{\gamma!}(y-x_0)^\gamma\r]f(x)\,dxa(y)\,dy\\
&\quad=\int_{B_0}a(y)\lim_{\eta\to0^+}\int_{\rn\setminus B(y,\eta)}
\lf[K(x,y)-\sum_{\{\gamma\in\zz_+^n:\ |\gamma|\leq s\}}
\frac{\partial_{(2)}^{\gamma}K(x,x_0)}{\gamma!}(y-x_0)^\gamma
\mathbf{1}_{\rn\setminus 2B_0}(x)\r]f(x)\,dx\,dy\noz\\
&\quad=\int_{B_0}a(y)\lim_{\eta\to0^+}\int_{\rn\setminus B(y,\eta)}
\lf[\widetilde{K}(y,x)-\sum_{\{\gamma\in\zz_+^n:\ |\gamma|\leq s\}}
\frac{\partial_{(1)}^{\gamma}\widetilde{K}(x_0,x)}{\gamma!}(y-x_0)^\gamma
\mathbf{1}_{\rn\setminus 2B_0}(x)\r]f(x)\,dx\,dy\\
&\quad=\int_{B_0}a(y)\widetilde{T}(f)(y)\,dy
=\lf\langle a,\widetilde{T}(f)\r\rangle.
\end{align*}
This finishes the proof of Lemma \ref{duity-T-Tw}.
\end{proof}

Now, we give the following main result of
this subsection, which shows the boundedness
of Calder\'on--Zygmund operators on
$HK_{(p,q,s)_\alpha}^{\mathrm{con}}(\rn)$.

\begin{theorem}\label{T-bounded-HK}
Let $p\in(1,\fz)$, $q\in(1,\fz)$,
$\frac{1}{p}+\frac{1}{p'}=1=\frac{1}{q}+\frac{1}{q'}$, $s\in\zz_+$,
and $\alpha\in\rr$ satisfy $\frac{1}{q}-\frac{1}{p}<\alpha<\frac{s+\delta}{n}$
with $\dz\in(0,1]$ as in Definition \ref{def-s-k}.
Let $K$ be an $s$-order standard kernel as
in Definition \ref{def-s-k} and $T$ the $s$-order
Calder\'on--Zygmund singular integral operator with kernel $K$.
Then the following two statements are equivalent:
\begin{enumerate}
\item[\rm (i)]
there exists a positive constant $C$ such that,
for any $(p,q,s)_\alpha$-atom $a$,
\begin{align}\label{Ta-C}
	\|T(a)\|_{HK_{(p,q,s)_\alpha}^{\mathrm{con}}(\rn)}\leq C
\end{align}
and $T$ can be extended to a unique continuous
linear operator, still denoted by $T$, on the space
$HK_{(p,q,s)_\alpha}^{\mathrm{con}}(\rn)$,
namely, there exists a positive constant $C$ such that,
for any $g\in HK_{(p,q,s)_\alpha}^{\mathrm{con}}(\rn)$,
\begin{align}\label{3.34y}
\|T(g)\|_{HK_{(p,q,s)_\alpha}^{\mathrm{con}}(\rn)}
\leq C\|g\|_{HK_{(p,q,s)_\alpha}^{\mathrm{con}}(\rn)}
\end{align}
and, moreover, for any $f\in JN_{(p',q',s)_\alpha}^{\mathrm{con}}(\rn)$,
$\widetilde{T}(f)\in JN_{(p',q',s)_\alpha}^{\mathrm{con}}(\rn)$ and
\begin{align}\label{3.34x}
\langle T(g),f\rangle=\lf\langle g,\widetilde{T}(f)\r\rangle,	
\end{align}
where $\widetilde{T}$ is as in Definition \ref{def-JN-CZO}
with kernel $\widetilde{K}$ as in \eqref{Kw-K};
\item[\rm (ii)]
for any $\gamma\in\zz_+^n$ with $|\gamma|\leq s$, $T^*(x^{\gamma})=0$.
\end{enumerate}
\end{theorem}

\begin{proof}
Let $p$, $q$, $s$, $\alpha$, $\dz$, $K$,
$T$, $\widetilde{K}$, and $\widetilde{T}$
be as in the present theorem.
We first show (i) $\Rightarrow$ (ii).
Indeed, using Lemma \ref{t3.9}, \eqref{3.34x}, and \eqref{3.34y},
we find that, for any
$f\in JN_{(p',q',s)_\alpha}^{\mathrm{con}}(\rn)$,
\begin{align*}
\lf\|\widetilde{T}(f)\r\|_{JN_{(p',q',s)_\alpha}^{\mathrm{con}}(\rn)}
&\sim\lf\|\mathcal{L}_{\widetilde{T}(f)}
\r\|_{(HK_{(p,q,s)_\alpha}^{\mathrm{con}}(\rn))^*}
\sim\sup_{\|g\|_{HK_{(p,q,s)_\alpha}^{\mathrm{con}}(\rn)}=1}
\lf|\lf\langle \mathcal{L}_{\widetilde{T}(f)},g\r\rangle\r|\\
&\sim\sup_{\|g\|_{HK_{(p,q,s)_\alpha}^{\mathrm{con}}(\rn)}=1}
\lf|\lf\langle g, \widetilde{T}(f)\r\rangle\r|
\sim\sup_{\|g\|_{HK_{(p,q,s)_\alpha}^{\mathrm{con}}(\rn)}=1}
\lf|\langle T(g), f\rangle\r|\\
&\ls\sup_{\|g\|_{HK_{(p,q,s)_\alpha}^{\mathrm{con}}(\rn)}=1}
\|T(g)\|_{HK_{(p,q,s)_\alpha}^{\mathrm{con}}(\rn)}
\|f\|_{JN_{(p',q',s)_\alpha}^{\mathrm{con}}(\rn)}\\
&\ls\|f\|_{JN_{(p',q',s)_\alpha}^{\mathrm{con}}(\rn)},
\end{align*}
where $\mathcal{L}_{\widetilde{T}(f)}$ is as in
\eqref{3.0x} with $f$ replaced by $\widetilde{T}(f)$.
This shows that $\widetilde{T}$ is bounded on
$JN_{(p',q',s)_\alpha}^{\mathrm{con}}(\rn)$,
which, combined with Theorem \ref{thm-bdd-JN},
further implies that, for any $\gamma\in\zz_+^n$
with $|\gamma|\leq s$, $T^*(x^{\gamma})=0$.
This finishes the proof that (i) $\Rightarrow$ (ii).

Next, we show (ii) $\Rightarrow$ (i).	
Let $\epsilon\in(0,1)$ be such that
$\frac{1}{\epsilon}
(\frac{1}{q}-\frac{1}{p}-\alpha)+\frac{1}{q'}+\frac{s}{n}<0$ and
$-\frac{1}{q'}-\frac{s+\dz}{n}\leq \frac{1}{\epsilon}(\frac{1}{q}
-\frac{1}{p}-\alpha)$.
Then, by Proposition \ref{HK-mole} and Lemma \ref{Ta-N}, we
conclude that, for any
$(p,q,s)_\alpha$-atom $a$, $T(a)$ is
a $(p,q,s,\alpha,\epsilon)$-molecule and
$\|T(a)\|_{HK_{(p,q,s)_\alpha}^{\mathrm{con}}(\rn)}\leq C$,
where the positive constant $C$
is independent of $a$. This implies that \eqref{Ta-C} holds true.
Moreover, from (ii) of the present theorem and Theorem \ref{thm-bdd-JN},
it follows that $\widetilde{T}$ is bounded on
$JN_{(p',q',s)_\alpha}^{\mathrm{con}}(\rn)$
and hence, for any $f\in JN_{(p',q',s)_\alpha}^{\mathrm{con}}(\rn)$,
$\widetilde{T}(f)\in JN_{(p',q',s)_\alpha}^{\mathrm{con}}(\rn)$.
Furthermore, by the boundedness of $\widetilde{T}$
on $JN_{(p',q',s)_\alpha}^{\mathrm{con}}(\rn)$ and
Lemmas \ref{Ta-N} and \ref{duity-T-Tw},
we conclude that (i) and (ii) of Theorem \ref{Bounded-HK-A} hold true
with $A:=T$, $\widetilde{A}:=\widetilde{T}$,
$p_1:=p=:p_2$, $q_1:=q=:q_2$,
$s_1:=s=:s_2$, $\alpha_1:=\alpha=:\alpha_2$,
which further implies that \eqref{3.34y} and \eqref{3.34x} hold true.
Then we complete the proof that (ii) $\Rightarrow$ (i)
and hence of Theorem \ref{T-bounded-HK}.
\end{proof}

\begin{remark}\label{4.21} 	
\begin{enumerate}
\item[\rm (i)]	
Theorem \ref{T-bounded-HK} implies that
the adjoint operator of $T$ on $HK_{(p,q,s)_\alpha}^{\rm{con}}(\rn)$
is just the operator $\widetilde{T}$
defined on $JN_{(p',q',s)_\alpha}^{\rm{con}}(\rn)$;
see, for instance, \cite[Theorem 5]{MS79}
and \cite[Theorem 3.1]{N17} for
similar results on Campanato-type spaces.

\item[\rm (ii)]
In Theorem \ref{T-bounded-HK},
$\frac{1}{q}-\frac{1}{p}<\alpha<\frac{s+\delta}{n}$
is a suitable assumption
because we need $\az\in(\frac{1}{q}-\frac{1}{p},\fz)$ when
proving that the molecule in Definition \ref{Def-mole} belongs to $HK_{(p,q,s)_\alpha}^{\rm{con}}(\rn)$
(see Proposition \ref{HK-mole}), and
$\alpha\in(-\fz,\frac{s+\delta}{n})$
when proving that $\widetilde{T}$ is bounded on
the congruent JNC space
(see Theorem \ref{thm-bdd-JN}).
It is still unknown whether or not
Theorem \ref{T-bounded-HK} holds true
with $\alpha\in (-\fz,\frac{1}{q}-\frac{1}{p}]$.

\item[\rm (iii)]	
Let $q\in(1,\fz)$, $s\in\zz_+$, $K$
be an $s$-order standard kernel, and $T$ the
$s$-order Calder\'on--Zygmund
singular integral operator with kernel $K$.
An essential ingredient in this article is
the boundedness of $T$ on $L^q(\rn)$,
which can be deduced from the boundedness of $T$ on $L^2(\rn)$
and (i) through (iv) of Definition \ref{def-s-k}.
However, if we directly assume that $T$ is bounded on $L^q(\rn)$,
then Theorems \ref{thm-bdd-JN} and \ref{T-bounded-HK}
still hold true under the assumption that $K$ only satisfies
(i) and (iii) of Definition \ref{def-s-k}.
\end{enumerate}
\end{remark}

\noindent\textbf{Acknowledgements}\quad
Dachun Yang would like to thank Professor Ibrahim Fofana
for him to bring their attention to references \cite{mfs2013,ffk2015}.

\bigskip

\noindent Hongchao Jia, Jin Tao,
Dachun Yang (Corresponding author), Wen Yuan and Yangyang Zhang

\smallskip

\noindent  Laboratory of Mathematics and Complex Systems
(Ministry of Education of China),
School of Mathematical Sciences, Beijing Normal University,
Beijing 100875, People's Republic of China

\smallskip

\noindent {\it E-mails}: \texttt{hcjia@mail.bnu.edu.cn} (H. Jia)

\noindent\phantom{{\it E-mails:}} \texttt{jintao@mail.bnu.edu.cn} (J. Tao)

\noindent\phantom{{\it E-mails:}} \texttt{dcyang@bnu.edu.cn} (D. Yang)

\noindent\phantom{{\it E-mails:}} \texttt{wenyuan@bnu.edu.cn} (W. Yuan)

\noindent\phantom{{\it E-mails:}} \texttt{yangyzhang@mail.bnu.edu.cn} (Y. Zhang)

\end{document}